\documentclass[a4,12pt]{article}
\usepackage[centertags]{amsmath}
\usepackage{amsmath}\usepackage{amssymb}\usepackage{amsthm}
\usepackage{color}\usepackage{dsfont}\usepackage{latexsym}
\usepackage[T1]{fontenc}\usepackage[latin1]{inputenc}
\usepackage{fancyhdr}\usepackage{indentfirst}\usepackage{xpatch}
\usepackage[title]{appendix}\usepackage{times}\usepackage{colortbl}
\usepackage{cite}\usepackage{float}\usepackage{algcompatible}
\usepackage{algorithm}\usepackage{algpseudocode}\usepackage{graphicx}
\usepackage{enumerate}\usepackage{booktabs}\usepackage[]{caption}
\usepackage[colorlinks=true,linkcolor=blue,citecolor=red,urlcolor=black]{hyperref}
\usepackage{framed}\usepackage{multicol}\usepackage{nomencl}
\makenomenclature
\definecolor{maroon}{cmyk}{0,0.87,0.68,0.32}

\newtheorem{definition}{Definition}[section]

\newtheorem{theorem}{Theorem}[section]
\newtheorem{proposition}{Proposition}[section]
\newtheorem{lemma}{Lemma}[section]
\newtheorem{remark}{Remark}[section]
\newtheorem{corollary}{Corollary}[section]

\renewcommand*\nompreamble{\begin{multicols}{2}}
\renewcommand*\nompostamble{\end{multicols}}
\makeatletter
\AtBeginDocument{\xpatchcmd{\@thm}{\thm@headpunct{.}}{\thm@headpunct{}}{}{}}
\def\ind{1\!\!1}
\numberwithin{equation}{section}
\title{\textbf{\Large Continuous and Impulse Controls Differential Game in Finite Horizon with Nash-Equilibrium and Application}}
\author{Brahim El Asri\thanks{Ibn Zohr University, Lab. LISAD, \'Equipe Aide à la D\'ecision, ENSA, B.P. 1136, Agadir, Morocco. E-mail: b.elasri@uiz.ac.ma.} \, and\, Hafid Lalioui\thanks{Ibn Zohr University, Lab. LISAD, \'Equipe Aide à la D\'ecision, ENSA, B.P. 1136, Agadir, Morocco. E-mail: hafid.lalioui@edu.uiz.ac.ma.}}
\providecommand{\keywords}[1]{\small\textbf{\textit{Keywords}} #1}
\providecommand{\MSC}[1]{\small\textbf{\textit{MS Classifications (2020)}} #1}
\providecommand{\JEL}[1]{\small\textbf{\textit{JEL Classifications (2020)}} #1}
\begin{document}
\date{}
\pagestyle{fancy}
\fancyhead[RE]{}
\fancyhead[RO]{}
\fancyhead[LO]{\leftmark}
\fancyhead[LE]{\leftmark}
\maketitle
\begin{abstract}
This paper considers a new class of deterministic finite-time horizon, two-player, zero-sum differential games (DGs) in which the maximizing player is allowed to take continuous and impulse controls whereas the minimizing player is allowed to take impulse control only. We seek to approximate the value function, and to provide a verification theorem for this class of DGs. By means of dynamic programming principle (DPP) in viscosity solution (VS) framework, we first characterize the value function as the unique VS to the related Hamilton-Jacobi-Bellman-Isaacs (HJBI) double-obstacle equation. Next, we prove that an approximate value function exists, that it is the unique solution to an approximate HJBI double-obstacle equation, and converges locally uniformly towards the value function of each player when the time discretization step goes to zero. Moreover, we provide a verification theorem which characterizes a Nash-equilibrium (NE) for the DG control problem considered. Finally, by applying our results, we derive a new continuous-time portfolio optimization model, and we provide related computational algorithms.
\end{abstract}
\vspace{0.25cm}\\
\MSC{49K35, 49L20, 49L25, 49N70, 49N90, 91G10.}\vspace{0.25cm}\\
\JEL{C61, C62, C63, C72, C73, G11.}
\nomenclature{symbol}{description}
\printnomenclature
\begin{table}[H]
\textbf{Acronym \& Nomenclature}\\

\begin{tabular}{|llll|}
\hline
BC&Bounded continuous&OC&Optimal control\\
BUC&Bounded uniformly continuous&ODE&Ordinary differential equation\\
DG&Differential game&PDE&Partial differential equation\\
DPE&Dynamic programming equation&QVI&Quasi-variational inequality\\
DPP&Dynamic programming principle&$V$&Value function\\
GAN&Generative adversarial network&$V^-$&Lower value function\\
$H$&Hamiltonian&$V^+$&Upper value function\\
$H_h$&Approximate Hamiltonian&$v_h$&Approximate value function\\
HJB&Hamilton-Jacobi-Bellman&VS&Viscosity solution\\
HJBI&Hamilton-Jacobi-Bellman-Isaacs&$\lambda$&Discount factor\\
HJBI$_h$&Approximate HJBI equation&$h$&Time-discretization step\\
NE&Nash-equilibrium&$\epsilon$&Tolerance\\
\hline
\end{tabular}
\end{table}
\section{Introduction}\label{Sect.1}
\textit{Optimal control (OC)} theory is an important field of research due to its connections with partial differential equations (PDEs) and many fields of engineering such as mathematical finance. As a consequence, OC problems can be used for designing numerical algorithms to nonlinear PDEs arising from many optimization problems, we refer for the instant to Bensoussan and Lions \cite{BL84}, Fleming and Rishel \cite{FR75}, Fleming and Soner \cite{FS06} and Pham \cite{Pha09} (see also \cite{BC97,B94,L82}). \textit{Impulse control} and \textit{differential game (DG)} problems appear in many practical situations, for example in mathematical finance one can consider the option pricing and the control of exchange rate problems by Bernhard \cite{B05}, Bernhard \& al. \cite{BELT06} and Bertola \& al. \cite{BRY16} (see also Barles \cite{Ba85}, Shaiju and Dharmatti \cite{DS05}, Dharmatti and Ramaswamy \cite{DR06}, Yong \cite{Y94}, Zhang \cite{Z11} and \cite{BL84} for more information). The rigorous mathematical study of OC problems and DGs gives rise to some non-linear partial differential equation (PDE), usually called \textit{Hamilton-Jacobi-Bellman (HJB)} equation for classic OC problems and \textit{Hamilton-Jacobi-Bellman-Isaacs (HJBI)} equation for DGs, satisfied by the \textit{value function} related to the control problem or DG. That is the value is the unique solution to the corresponding HJB or HJBI equation. In most cases, even in very simple, these value functions are not sufficiently smooth, then the related PDE needs to be studied in \textit{viscosity solution (VS)} framework. Introduced in 1980s by Crandall and Lions \cite{CL83} (see also Crandall \& al. \cite{CEL84, CIL92}) to circumvent the fact that the value function of control problems or DGs is not smooth enough, the notion of VS provides very powerful means to study in great generality and gives a rigorous formulation of the related PDEs to these control problems or DGs. The notion of value function has then a key role in the theory of OC problems and DGs, and the related PDEs should be considered in the viscosity sense.
\subsection{Continuous and Impulse Controls Differential Game}
In a previous work El Asri and Lalioui \cite{EL21}, a two-player, zero-sum, deterministic DG where each player uses both continuous and impulse controls in infinite-time horizon was studied (see also El Asri \& al. \cite{ELM21}). In \cite{EL21}, we proved under \textit{Isaac's condition} that the \textit{lower} and \textit{upper} value functions coincide. In \cite{ELM21}, the zero-sum deterministic impulse controls game problem we have considered involves only impulse controls in infinite-time horizon, where a \textit{new} HJBI \textit{quasi-variational inequality (QVI)} was defined to prove, under a \textit{proportional property assumption} on the maximizing player cost, that the value functions coincide and turn out to be the unique VS to the new HJBI QVI. The problem considered in this paper and the obtained results extend those in \cite{EL21,ELM21}, and provides an application to \textit{continuous-time portfolio} optimization problem.
\par This paper studies a \textit{new} class of deterministic finite-time horizon, two-player, zero-sum, continuous and impulse controls DG, defined by the $\mathbb{R}^n-$valued \textit{state vector} $y_{t,x}(s)$ solution of the dynamical equation (E1) below:
\begin{equation*}
\begin{aligned}
\text{(E1)}\;y_{t,x}(s)=&\;x+\int_{t}^{s}b\bigl(r,y_{t,x}(r);\theta(r)\bigr)dr+\sum_{m\geq 1}g_\xi\bigl(\tau_m,y_{t,x}(\tau_m^-);\xi_m\bigr)\ind_{[\tau_m,T]}(s)\prod_{k\geq 1}\ind_{\{\tau_m\neq\rho_k\}}\\
&+\sum_{k\geq 1}g_\eta\bigl(\rho_k,y_{t,x}(\rho_k^-);\eta_k\bigr)\ind_{[\rho_k,T]}(s),
\end{aligned}
\end{equation*}
for time variables $T\in(0,+\infty)$, $t\in[0,T]$ and $s\in[t,T]$, with initial state $y_{t,x}(t^-)=x\in\mathbb{R}^n$, where $y_{t,x}(t^-):=\lim_{t^\prime\uparrow t}y_{t,x}(t^\prime)$. In the differential form, for $s\neq\tau_m$, $s\neq\rho_k$ and the initial state $x$, the dynamical equation (E1) is governed by the following controlled ordinary differential equation (ODE): $$\dot y_{t,x}(s)=b\bigl(s,y_{t,x}(s);\theta(s)\bigr),\;\text{and}\;y_{t,x}(t^-)=x,$$ where $$\dot y_{t,x}(s):=\frac{dy_{t,x}(s)}{ds},$$ $\theta(.)\in\Theta(t,T)$ being the continuous control in $\Theta(t,T)$, the space of measurable functions from $[t,T]\subset\mathbb{R}_+$ into $\mathbb{R}^l$, and $b$ is a function that satisfies the following assumption:
\begin{itemize}
\item[\textbf{[$\textbf{H}_b$]}] \textbf{(Dynamic)} The function $b:(s,y,\theta)\in[0,+\infty)\times\mathbb{R}^n\times\mathbb{R}^l\rightarrow b(s,y;\theta)\in\mathbb{R}^n$ is continuous w.r.t. $s$ uniformly in $y$ and $\theta$, Lipschitz-continuous w.r.t. $y$ uniformly in $s$ and $\theta$ with constant $C_b>0$, and continuous w.r.t. $\theta$. Moreover, $b$ satisfies $\bigl\|b(s,y;\theta)\bigr\|_\infty\leq M$ for any $(s,y,\theta)\in[0,+\infty)\times\mathbb{R}^n\times\mathbb{R}^l$ and some positive constant $M$.
\end{itemize}
The state vector $y_{t,x}(s)$, in addition to the continuous evolution due to the ODE above, undergoes impulses (jumps) $\xi_m$ and $\eta_k$ at certain impulse stopping times $\tau_m$ and $\rho_k$, respectively, that is:
\begin{equation*}
\left\{
\begin{aligned}
y_{t,x}(\tau_m^+)=&\;y_{t,x}(\tau_m^-)+g_\xi\bigl(\tau_m,y_{t,x}(\tau_m^-);\xi_m\bigr)\prod_{k\geq 1}\ind_{\{\tau_m\neq\rho_k\}},\;t\leq\tau_m\leq T,\;\xi_m\neq 0;\\
y_{t,x}(\rho_k^+)=&\;y_{t,x}(\rho_k^-)+g_\eta\bigl(\rho_k,y_{t,x}(\rho_k^-);\eta_k\bigr),\;t\leq\rho_k\leq T,\;\eta_k\neq 0,
\end{aligned}
\right.
\end{equation*}
where $y_{t,x}(s^-):=\lim_{s^\prime\uparrow s}y_{t,x}(s^\prime)$ and $y_{t,x}(s^+):=\lim_{s^\prime\downarrow s}y_{t,x}(s^\prime)$, under the following assumption on the two functions $g_\xi$ and $g_\eta$:
\begin{itemize}
\item[\textbf{[$\textbf{H}_{g}$]}] \textbf{(Impulses Form)} The function $g_\xi:(s,y,\xi)\in [0,+\infty)\times\mathbb{R}^n\times\mathbb{R}^p\rightarrow g_\xi(s,y;\xi)\in\mathbb{R}^n$ $\bigl(\text{resp.}\;g_\eta:(s,y,\eta)\in [0,+\infty)\times\mathbb{R}^n\times\mathbb{R}^q\rightarrow g_\eta(s,y;\eta)\in\mathbb{R}^n\bigr)$ is Lipschitz-continuous w.r.t. $s$, uniformly in $y$ and $\xi$ (resp. $\eta$), with constant $\tilde{C}_{g_\xi}>0$ $(\text{resp.}\;\tilde{C}_{g_\eta}>0)$, and Lipschitz-continuous w.r.t. $y$, uniformly in $s$ and $\xi$ (resp. $\eta$), with constant $C_{g_\xi}>0$ $(\text{resp.}\;C_{g_\eta}>0)$.
\end{itemize}
\subsection{HJBI Equation and Approximate Equation}
Initiated in the 1950s by Bellman \cite{B57}, the \textit{dynamic programming principle (DPP)} leads, for our deterministic finite-time horizon, two-player, zero-sum, DG control problem, to a non-linear PDE satisfied by the Elliott-Kalton \cite{EK72,EK74} value function of the game, and given by the following system:
\begin{equation*}
\text{(HJBI)}\;\left\{
\begin{aligned}
&
\begin{aligned}
\max\biggl\{\min\Bigl[&-\frac{\partial}{\partial s}v(s,y)+\lambda v(s,y)+H\bigl(s,y,D_{y}v(s,y)\bigr),v(s,y)-\mathcal{H}_{sup}^c v(s,y)\Bigr];\\
&v(s,y)-\mathcal{H}_{inf}^\chi v(s,y)\biggr\}=0,\;\text{on}\;[t,T)\times\mathbb{R}^n;
\end{aligned}\\
&v(T,y)=G(y)\;\text{for all}\;y\in\mathbb{R}^n,
\end{aligned}
\right.
\end{equation*}
where the \textit{Hamiltonian} ($H$) and the two \textit{non-local cost operators} ($\mathcal{H}_{sup}^c$), ($\mathcal{H}_{inf}^\chi$) are classic expressions given in Section $\ref{Sect.2.3}$ below. The above system (HJBI) called \textit{Hamilton-Jacobi-Bellman-Isaacs (HJBI)} equation, or \textit{dynamic programming equation (DPE)}, and we will refer to it as \textit{HJBI equation}. By combining the notion of VS for the HJBI equation, with \textit{comparison principle} for these solutions, we characterize the value function of the zero-sum DG control problem studied as the unique VS of the HJBI equation, and this can then be used to obtain further results. Indeed, our paper will provide a discrete-time approximation for the HJBI equation. Let $h>0$ be the time discretization step, $h_0$ be a positive number, and $\Phi(h)$ be a continuous function such that $\Phi(0)=1$ and $0<\Phi(h)<1$ for $0<h<h_0$, the \textit{approximate equation} (HJBI$_h$) of the HJBI equation will be given by the following system:
\begin{equation*}
\text{(HJBI$_h$)}\;\left\{
\begin{aligned}
&
\begin{aligned}
\max\biggl\{\min\Bigl[&H_h\bigl(s,y,v_h(s,y)\bigr),v_h(s,y)-\Phi(h)\mathcal{H}_{sup}^c v_h(s,y)\Bigr];v_h(s,y)-\Phi(h)\mathcal{H}_{inf}^\chi v_h(s,y)\biggr\}=0,\\
&\;\text{on}\;[t,T)\times\mathbb{R}^n;\\
\end{aligned}\\
&v_h(T,y)=G(y)\;\text{for all}\;y\in\mathbb{R}^n,
\end{aligned}
\right.
\end{equation*}
where the \textit{approximate Hamiltonian} ($H_h$) is defined in Section $\ref{Sect.2.3}$ below. We may use this approximate equation (HJBI$_h$) to give some computational aspects for our zero-sum DG control problem. Indeed, the convergence of the \textit{approximate value function}, unique solution of the approximate equation (HJBI$_h$), to the unique bounded uniformly continuous (BUC) VS of the HJBI equation, leads to a numerical approach for the considered DG control problem. The \textit{Nash-equilibrium (NE)} of Section \ref{Sect.5} and the computational algorithms of Section $\ref{Sect.7}$ will give an illustration of this approach.
\subsection{Impulse Control Problems and Differential Games with Impulses}
Regarding the literature on optimal impulse control problems and DGs with impulse controls, one might find numerous variants of OC problems with impulse, for example Liu \& al. \cite{LTJW98} have considered the case where the number of jump instants is fixed, and Reddy \& al. \cite{RWZ16} have studied a problem when the impulse instants are known a priori. The literature on DGs with impulse controls is sparse, zero-sum games with one player using piece-wise continuous controls and the other using impulses were studied in a deterministic setting in \cite{Y94}, and in a stochastic setting in \cite{Z11} and Azimzadeh \cite{A17} (see also Issacs \cite{I65}). In \cite{Ba85} and El Farouq \& al. \cite{BBEl10} the authors studied a deterministic impulse control problem in infinite-time horizon and a finite-time horizon DG, respectively (see also El Asri \cite{El13}). Bernhard \& al. \cite{B05,BELT06} introduced impulse control in zero-sum DGs to study an option pricing problem. Impulse control problems are typically solved using two main approaches, one based on \textit{Bellman's DPP} \cite{B57}, and another using \textit{Pontryagin's maximum principle} \cite{BGMP62} to compute the value function (see e.g. \cite{BRY16} and Blaquiere \& al. \cite{B85}). Recent papers by Cosso \cite{C13} and El Asri and Mazid \cite{ElM18} consider dynamic programming approach for zero-sum stochastic DGs where both players use only impulse control (see also El Asri and Mazid \cite{ElM20} where a state and time dependent cost functions stochastic impulse control problem was considered). Works by A\"id \& al. \cite{ABCCT20}, Basei \& al. \cite{BCG21}, Campi and De Santis \cite{CS20} and Sadana \& al. \cite{SRZ21,SRBZ21} study some nonzero-sum DGs with impulse controls. We mention that in \cite{ABCCT20} authors studied a DG between two nations that have different targets for the currency exchange rate, and provided a system of QVIs that needs to be solved in order to compute the NE. In \cite{SRZ21} the necessary and sufficient conditions for the existence of an open-loop NE for a class of DGs with impulse control were formulated. In their recent work, Gammoudi and Zidani \cite{GZ22} have studied a two-player zero-sum DG with state constraints. Regarding discrete-time approximation of HJB equation of deterministic control theory, we cite the works by Falcone \cite{F87}, Gonzalez and Rofman \cite{GR85,GR85_2}, Capuzzo-Dolcetta \cite{CD83}, Capuzzo-Dolcetta and Ishii \cite{CDI84}, and recent works by El Farouq \cite{El17,El17_2} related to deterministic impulse control problems (see also \cite{BS91,S85,S85',CF97,CF99}). Another method for obtaining an approximate solution of the HJB equation is the \textit{adaptive dynamic programming} (see for example Mu \& al. \cite{MWN20,MWN21}).
\subsection{Financial Market Modeling, Contributions and Outline}
The use of OC methods to analyze financial market models has expanded at a remarkable rate after the revolutionary works by Markowitz \cite{Mar52} and Merton \cite{Mer69,Mer71,Mer90}. Many researches have dealt with the role of OC in portfolio optimization, including Eastham and Hastings \cite{EaHa88}, Hastings \cite{Has92} and Korn \cite{kor98}. In Section \ref{Sect.6}, a deterministic finite-time horizon, two-player, zero-sum, impulse controls DG approach for continuous-time portfolio optimization will be given. We first adjust the functions $b$, $g_\xi$ and $g_\eta$ of the dynamical equation (E1) to our portfolio optimization problem, then, for time variable $s\in[t,T]$ and a fixed positive real discount factor $\lambda$, we consider the following discounted terms:
\begin{enumerate}
\item A running gain/cost of integral type $\int_{t}^{T}f^\pi\bigl(s,y_{t,x}^{\psi,v}(s);\theta(s)\bigl)\exp\bigl(-\lambda(s-t)\bigr)ds$, giving by the running gain/cost function $f^\pi:=L^\pi-U^\pi$, where $L^\pi$ and $U^\pi$ denote, respectively, the investor's stocks holding cost and his instantaneous utility function;
\item The total jump costs $$-\sum_{m\geq1}c^\pi\bigl(\tau_m,y_{t,x}^{\psi,v}(\tau_m^-);\xi_m\bigr)\exp\bigl(-\lambda(\tau_m-t)\bigr)\ind_{\{\tau_m\leq T\}}\prod_{k\geq 1}\ind_{\{\tau_m\neq\rho_k\}},$$ and $\sum_{k\geq1}\chi^\pi\bigl(\rho_k,y_{t,x}^{\psi,v}(\rho_k^-);\eta_k\bigr)\exp\bigl(-\lambda(\rho_k-t)\bigr)\ind_{\{\rho_k\leq T\}}$ for the maximizing player$-\xi$ (market) and the minimizing player$-\eta$ (investor), respectively, with impulse stopping times $\tau_m,\rho_k$ and impulse values $\xi_m,\eta_k$;
\item A terminal gain/cost $G^\pi\bigl(y_{t,x}^{\psi,v}(T)\bigr)\exp\bigl(-\lambda(T-t)\bigr)$ giving by the function $G^\pi$,
\end{enumerate}
with the assumption that the flow of funds is between the market and the investor who reacts immediately to the market whereas the market is not so quick in reacting to the investor's moves. We note that $\psi:=\bigl(\theta(.),u:=(\tau_m,\xi_m)_{m\in\mathbb{N}^*}\bigr)$ represents the admissible continuous-impulse control for maximizing player$-\xi$ (market) and $v:=(\rho_k,\eta_k)_{k\in\mathbb{N}^*}$ is the admissible impulse control for minimizing player$-\eta$ (investor). Thus we make our deterministic finite-time horizon, DG framework for the continuous-time portfolio optimization problem. Using the three discounted terms in the above, we can define an Elliott-Kalton \cite{EK72,EK74} value function $v(t,w)$ for our portfolio optimization problem which represents the investor's lost in the worst-case scenario, we then apply our results to derive a new continuous-time portfolio optimization model. Following \cite{CDI84,GR85,GR85_2,F87}, we derive some computational aspects for $v(t,w)$ from the approximate equation (HJBI$_h$).
\par By establishing existence and uniqueness results for the considered class of DGs in viscosity sense, providing discrete-time approximation method of their HJBI equation which leads to a NE, applying to mathematical finance, and providing computational algorithms, our paper contributes to both the theory and applications of DGs with impulse controls. This paper also leads to a new continuous-time portfolio optimization model where the investor tries to counteract to dangerous scenarios that can happen because of market price fluctuations. To the best of our knowledge the literature on deterministic DGs does not provide any theoretical or computational means to study the class of DGs we have considered in this paper.
\par The outline of the paper is the following: in Section \ref{Sect.2}, we formulate the zero-sum DG control problem studied and we define its value function, then we give the DPP and regularity results. In Section \ref{Sect.3}, by means of the VS framework, we investigate the HJBI equation that characterizes the value function of the game studied. Section \ref{Sect.4} deals with the approximate equation (HJBI$_h$) and discusses the convergence of the \textit{approximation scheme}. More precisely, we prove that the approximate value function converges, as the discretization step goes to zero, locally uniformly towards the value function of the considered game. In Section \ref{Sect.5}, we expose a \textit{verification theorem} for identifying a NE strategy derived from the convergence result of Section \ref{Sect.4}. In Section \ref{Sect.6}, we apply the theory we have developed to derive a new continuous-time portfolio optimization model where the market is playing against the investor and wishes to maximize his discounted terminal cost, we give a \textit{portfolio strategy}. Finally, in Section \ref{Sect.7}, we provide related \textit{value and policy iteration} algorithms for our zero-sum DG control problem.
\section{Game Problem Formulation and Preliminary Results}\label{Sect.2}
\subsection{Zero-Sum Deterministic Differential Game Control Problem}
We will be given the precise statement of our two-player, zero-sum, deterministic DG control problem, the definition of its related value functions and some preliminary results. The \textit{state vector} $y_{t,x}(s)$ of the two-player, zero-sum, deterministic continuous and impulse controls DG considered is given, for finite-time horizon with time variables $T\in(0,+\infty)$, $t\in[0,T]$, by the solution of the following dynamical system:
\begin{equation*}
\text{(S)}\;\left\{
\begin{aligned}
\dot y_{t,x}(s)&=b\bigl(s,y_{t,x}(s);\theta(s)\bigr),\;s\neq\tau_m,\;s\neq\rho_k,\;s\in[t,T];\\
y_{t,x}(\tau_m^+)&=y_{t,x}(\tau_m^-)+g_\xi\bigl(\tau_m,y_{t,x}(\tau_m^-);\xi_m\bigr)\prod_{k\geq 1}\ind_{\{\tau_m\neq\rho_k\}},\;\tau_m\in[t,T],\;\xi_m\neq 0;\\
y_{t,x}(\rho_k^+)&=y_{t,x}(\rho_k^-)+g_\eta\bigl(\rho_k,y_{t,x}(\rho_k^-);\eta_k\bigr),\;\rho_k\in[t,T],\;\eta_k\neq 0;\\
y_{t,x}(t^-)&=x\in\mathbb{R}^n\;\text{(initial state)},
\end{aligned}
\right.
\end{equation*}
the evolution of the state system (S), described by the mapping $y_{t,x}:[t,T]\rightarrow\mathbb{R}^n$, is controlled by two players:
\begin{enumerate}[i.]
\item A \textit{maximizing player}$-\xi$ who uses both \textit{continuous control} $\theta(.)$ and \textit{impulse control} $u:=(\tau_m,\xi_m)_{m\in\mathbb{N}^*}$;
\item A \textit{minimizing player}$-\eta$ who adopts only \textit{impulse control} $v:=(\rho_k,\eta_k)_{k\in\mathbb{N}^*}$,
\end{enumerate}
where $\mathbb{N}^*:=\mathbb{N}\backslash\{0\}$. The mapping $y_{t,x}:[t,T]\rightarrow\mathbb{R}^n$ is called the \textit{response} or the \textit{state} corresponding to controls $\theta(.),u$ and $v$. These controls are defined, for our zero-sum DG control problem, as follows:
\begin{definition}[Continuous and Impulse Controls]
We let the continuous control $\theta(.)$ and the impulse controls $u$ and $v$, related to the zero-sum DG control problem studied, be defined by:
\begin{enumerate}[i.]
\item A continuous control $\theta(.)\in\Theta(t,T)$ is giving by a map $\theta:[t,T]\rightarrow\mathbb{R}^l$, where $\Theta(t,T)$ denotes the set of all measurable functions of $[t,T]\subset\mathbb{R}_+$ to $\mathbb{R}^l$;
\item An impulse control $u:=(\tau_m,\xi_m)_{m\in\mathbb{N}^*}\in\mathcal{U}(t,T)$ for player$-\xi$ $\bigl(\text{resp.}\;v:=(\rho_k,\eta_k)_{k\in\mathbb{N}^*}\in\mathcal{V}(t,T)$ for player$-\eta\bigr)$ is defined by the non-decreasing impulse time sequence $\{\tau_m\}_{m\in\mathbb{N}^*}\;\bigl(\text{resp.}\;\{\rho_k\}_{k\in\mathbb{N}^*}\bigr)$ of $[t,T]$, and by the impulse value (or, size) sequence $\{\xi_m\}_{m\in\mathbb{N}^*}$ $\bigl(\text{resp.}\;\{\eta_k\}_{k\in\mathbb{N}^*}\bigr)$ of elements of $U\subset\mathbb{R}^p\;\bigl(\text{resp.}\;V\subset\mathbb{R}^q\bigr)$, where $\mathcal{U}(t,T)\;\bigl(\text{resp.}\;\mathcal{V}(t,T)\bigl)$ is the space of all impulse controls $u\;(\text{resp.}\;v)$. \qed
\end{enumerate}
\end{definition}
We denote, for notational brevity, $\Theta=\Theta(t,T)$, $\mathcal{U}=\mathcal{U}(t,T)$, $\mathcal{V}=\mathcal{V}(t,T)$ and $\Psi=\Theta\times\mathcal{U}$. In the system (S), the product $\prod_{k\geq 1}\ind_{\{\tau_m\neq\rho_k\}}$ signifies that when the two players act together on the system at the same time, only the action of minimizing player$-\eta$ is tacking into account.
Assumptions on the data, related to system (S), were given in Section \ref{Sect.1}.
\begin{remark}
By assumption $\textbf{H}_{b}$ the existence of a unique global solution of the above dynamical system (S) is guaranteed and will be denoted by $y_{t,x}^{\psi,v}(s)$ at time $s$, for $\psi:=\bigl(\theta(.),u\bigr)\in\Psi$ and $v\in\mathcal{V}$. \qed
\end{remark}
The gain (resp. cost) functional $J$ for maximizing player$-\xi$ (resp. minimizing player$-\eta$) is defined, for $\psi:=\bigl(\theta(.),u\bigr)\in\Psi$ and $v\in$$\mathcal{V}$ being the admissible controls for the two players, as follows:
\begin{equation*}
\begin{aligned}
J(t,x;\psi,v):=&\int_t^T f\bigl(s,y_{t,x}^{\psi,v}(s);\theta(s)\bigr)\exp\bigl(-\lambda(s-t)\bigr)ds\\
&-\sum_{m\geq 1}c\bigl(\tau_m,y_{t,x}^{\psi,v}(\tau_m^-);\xi_m\bigr)\exp\bigl(-\lambda(\tau_m-t)\bigr)\ind_{\{\tau_m\leq T\}}\prod_{k\geq 1}\ind_{\{\tau_m\neq\rho_k\}}\\
&+\sum_{k\geq 1}\chi\bigl(\rho_k,y_{t,x}^{\psi,v}(\rho_k^-);\eta_k\bigr)\exp\bigl(-\lambda(\rho_k-t)\bigr)\ind_{\{\rho_k\leq T\}}\\
&+G\bigl(y_{t,x}^{\psi,v}(T)\bigr)\exp\bigl(-\lambda(T-t)\bigr),
\end{aligned}
\end{equation*}
where $y_{t,x}^{\psi,v}(s)$ is the \textit{response} to controls $\psi$ and $v$ at time $s$. The functional $J$ will be considered under the following classical assumptions on the given \textit{running gain/cost} function $f$, \textit{impulse cost} functions $c,\chi$, and \textit{terminal gain} $G$, where $\lambda$ is a fixed positive real that represents the \textit{discount factor}:
\begin{itemize}
\item[\textbf{[$\textbf{H}_f$]}] \textbf{(Running Gain)} We assume that the function $f:[0,+\infty)\times\mathbb{R}^n\times\mathbb{R}^l\rightarrow\mathbb{R}$ is continuous w.r.t. $s$ uniformly in $y$ and $\theta$, Lipschitz-continuous w.r.t. $y$ uniformly in $s$ and $\theta$ with constant $C_f>0$, and continuous w.r.t. $\theta$. Moreover, $f$ satisfies $\bigl\|f(s,y;\theta)\bigr\|_{\infty}\leq M$ for any $(s,y,\theta)\in[0,+\infty)\times\mathbb{R}^n\times\mathbb{R}^l$ and some positive constant $M$;
\item[\textbf{[$\textbf{H}_{c,\chi}$]}] \textbf{(Impulses Cost)} The impulse cost functions $c:[0,+\infty)\times\mathbb{R}^n\times U\subset\mathbb{R}^p\rightarrow\mathbb{R}_{+}^*$ and $\chi:[0,+\infty)\times\mathbb{R}^n\times V\subset\mathbb{R}^q\rightarrow\mathbb{R}_{+}^*$ are from $[0,+\infty)\times\mathbb{R}^n$ and two convex cones $U$ and $V$, respectively, into $\mathbb{R}_{+}^*$, non negative, and satisfy
$$\inf_{(s,y,\xi)\in[0,+\infty)\times\mathbb{R}^n\times U}c(s,y;\xi)> 0,\;\text{and}\;\inf_{(s,y,\eta)\in[0,+\infty)\times\mathbb{R}^n\times V}\chi(s,y;\eta)>0.$$
The function $c$ (resp. $\chi$) is Lipschitz-continuous w.r.t. $y$, uniformly in $s$ and $\xi$ (resp. $\eta$), with constant $C_{c}>0$ (resp. $C_{\chi}>0$) and continuous w.r.t. $s$ and $\xi$ (resp. $\eta$). Moreover, for all $(s,y)\in[0,+\infty)\times\mathbb{R}^n$, $\xi_1,\xi_2\in U$ and $\eta_1,\eta_2\in V$, we let the impulse cost functions satisfy the following:
\begin{equation*}
\left\{
\begin{aligned}
c(s,y;\xi_1+\xi_2)&\leq c(s,y;\xi_1)+c(s,y;\xi_2);\\
\chi(s,y;\eta_1+\eta_2)&\leq\chi(s,y;\eta_1)+\chi(s,y;\eta_2),
\end{aligned}
\right.
\end{equation*}
that is multiple impulses occurring at the same time are sub-optimal;
\item[\textbf{[$\textbf{H}_{G}$]}] \textbf{(Terminal Gain)} We let the function $G:\mathbb{R}^n\rightarrow\mathbb{R}$ be bounded, Lipschitz-continuous with constant $C_G>0$ and satisfies, for all $y\in\mathbb{R}^n$ at time $T$, the following \textit{no terminal impulse} condition:
$$\sup_{\xi\in U}\Bigl\{G\bigl(y+g_\xi(T,y;\xi)\bigr)-c(T,y;\xi)\Bigr\}\leq G(y)\leq\inf_{\eta\in V}\Bigl\{G\bigl(y+g_\eta(T,y;\eta)\bigr)+\chi(T,y;\eta)\Bigr\}.$$
\end{itemize}
\par Note that the functional $J$ represents a gain for the maximizing player and a cost for the minimizing, it is the criterion which player$-\xi$ wants to maximize and player$-\eta$ wants to minimize. In the other words, $-J$ is the cost player$-\eta$ has to pay, so the sum of the costs of the two players is null, which explains the name \textit{zero-sum}.
\begin{remark}
Assumptions $\textbf{H}_{g}$, $\textbf{H}_{f}$, $\textbf{H}_{c,\chi}$ and $\textbf{H}_{G}$ provide the classical framework for the study, in the VS framework, of the zero-sum DG control problem considered in this paper. In the rest of the paper:
\begin{enumerate}
\item We let $n$, $p$, $q$ and $l$ be some fixed positive integers, $k,m\in\mathbb{N}^*$, $T\in(0,+\infty)$, $t\in[0,T]$ and $s\in[t,T]$;
\item We denote by $|.|$ and $\|.\|$ the Euclidean vector norm in $\mathbb{R}$ and $\mathbb{R}^n$, respectively, and by $\|.\|_\infty$ the infinite norm in the space of bounded continuous functions. \qed
\end{enumerate}
\end{remark}
\par Before moving to the notions of \textit{non-anticipative strategy} and \textit{value function}, we first give the following Proposition \ref{Proposition2.1}:
\begin{proposition}[Estimates on the Trajectories]\label{Proposition2.1}
Assume $\textbf{H}_b$ and $\textbf{H}_g$. Then we have, for all $x,x^\prime\in\mathbb{R}^n$, $t\in[0,T]$ and $t^\prime\in[t,T]$, the following estimates on the trajectories:
\begin{enumerate}[i.]
\item $\bigl\|y_{t,x}^{\psi,v}(s)-x\bigl\|\leq M(s-t)\;\text{for any}\;s\in[t,T];$
\item $\bigl\|y_{t^\prime,x^\prime}^{\psi,v}(s)-y_{t,x}^{\psi,v}(s)\bigl\|\leq\exp\bigl(C(T-t^\prime)\bigr)\bigl(\|x^\prime-x\|+M(t^\prime-t)\bigr)\;\text{for any}\;s\in[t^\prime,T],$
\end{enumerate}
for all $\psi:=\bigl(\theta(.),u\bigr)\in\Psi$ and $v\in\mathcal{V}$, where $C$ and $M$ are two real positive constants.
\end{proposition}
\begin{proof}
The proof of this result is classic.
\end{proof}
We now assume that one player knows just the current and past choices of the control made by his opponent. Thus, following Elliott and Kalton \cite{EK72,EK74}, we are given an information pattern for the two players by introducing the notion of \textit{non-anticipative strategy} for our zero-sum DG control problem (see also Evans and Souganidis \cite{ES84}) as follows:
\begin{definition}[Non-Anticipative Strategy]
A strategy for player$-\xi$ is a map $\alpha:\mathcal{V}\rightarrow\Psi$; it is non-anticipative, if, for any $v_1,v_2\in\mathcal{V},\;T>0\;\text{and}\;t\in[0,T]$, $v_1\equiv v_2$ on $[t,T]$ implies $\alpha(v_1)\equiv\alpha(v_2)$ on $[t,T]$, i.e., if $\alpha(v_1):=\bigl(\theta_1(.),u_1\bigr)\in\Psi$ and $\alpha(v_2):=\bigl(\theta_2(.),u_2\bigr)\in\Psi$ with $v_1\equiv v_2$ then $\theta_1(s)=\theta_2(s)$ and $u_1\equiv u_2$ for any $t\leq s\leq T$. We denote by $\mathcal{A}$ the set of all non-anticipative strategies $\alpha$ for player$-\xi$.\\
Similarly, the set of all non-anticipative strategies $\beta$ for player$-\eta$ is denoted by $\mathcal{B}$ as
\begin{equation*}
\begin{aligned}
\mathcal{B}:=\Bigl\{\beta:&\;\Psi\rightarrow\mathcal{V}\;\Bigl|\;\theta_1(s)=\theta_2(s)\;\text{and}\;u_1\equiv u_2\;\text{on}\;[t,T]\;\text{for all}\;\theta_1(.),\theta_2(.)\in\Theta,\;u_1,u_2\in\mathcal{U},\\
&T>0,\;t\in[0,T]\;\text{implies}\;v_1:=\beta\bigl(\theta_1(s),u_1\bigr)\equiv v_2:=\beta\bigl(\theta_2(s),u_2\bigr)\;\text{on}\;[t,T]\;\text{for all}\;t\leq s\leq T\Bigr\}.
\end{aligned}
\end{equation*}\qed
\end{definition}
We then give the definitions of the \textit{lower} and the \textit{upper} value functions related to our problem.
\begin{definition}[Value Functions]
The definitions of the lower value function ($V^-$) and the upper value function ($V^+$) of the zero-sum DG control problem with the gain/cost functional $J:[0,T]\times\mathbb{R}^n\times\Psi\times\mathcal{V}\rightarrow\mathbb{R}$, related to system (S), are given by the following expressions:
\begin{equation*}
\begin{aligned}
V^-(t,x)&:=\inf_{\beta\in\mathcal{B}}\sup_{\psi\in\Psi}J\bigl(t,x;\psi,\beta(\psi)\bigr);\\
V^+(t,x)&:=\sup_{\alpha\in\mathcal{A}}\inf_{v\in\mathcal{V}}J\bigl(t,x;\alpha(v),v\bigr).
\end{aligned}
\end{equation*}
If $V^-(t,x)=V^+(t,x)$ we say that the game, with initial point $x\in\mathbb{R}^n$ at initial time $t\in[0,T]$, has a value. We denote the value function of the zero-sum DG control problem by:
\begin{equation*}
V(t,x):=V^-(t,x)=V^+(t,x).
\end{equation*}\qed
\end{definition}
Next, we give some properties concerning the value functions.
\begin{proposition}[Boundedness]\label{Proposition2.2}
Assume $\textbf{H}_b$, $\textbf{H}_g$, $\textbf{H}_f$, $\textbf{H}_{c,\chi}$ and $\textbf{H}_{G}$. Then the lower value and the upper value are bounded in $[0,T]\times\mathbb{R}^n$.
\end{proposition}
\begin{proof}
The proof of this result is classic, see e.g. \cite{BBEl10}.
\end{proof}
\begin{proposition}[Time-Continuity]\label{Proposition2.3}
Assume $\textbf{H}_b$, $\textbf{H}_g$, $\textbf{H}_f$, $\textbf{H}_{c,\chi}$ and $\textbf{H}_{G}$. Then the lower value and the upper value are continuous with respect to time variable.
\end{proposition}
\begin{proof}
The proof of this result is classic, see e.g. \cite{BBEl10}.
\end{proof}
The next section is devoted to announcing some regularity results for the value functions with respect to the state variable.
\subsection{Dynamic Programming Principle and Regularity Results}\label{Sect.2.2}
We first give the Bellman's \cite{B57} DPP for the two-player zero-sum DG control problem considered:
\begin{theorem}[Dynamic Programming Principle]\label{DPP}
Assume $\textbf{H}_b$, $\textbf{H}_g$, $\textbf{H}_f$, $\textbf{H}_{c,\chi}$ and $\textbf{H}_{G}$. For all $x\in\mathbb{R}^n$ and $t^\prime\in[t,T]$, the lower value and the upper value satisfy, respectively,
\begin{equation}\label{equation_2.1}
\begin{aligned}
V^-(t,x)=&\inf_{\beta\in\mathcal{B}}\sup_{\psi\in\Psi}\biggl\{\int_t^{t^\prime} f\bigl(s,y_{t,x}^{\psi,\beta(\psi)}(s);\theta(s)\bigr)\exp\bigl(-\lambda(s-t)\bigr)ds\\
&-\sum_{m\geq 1}c\bigl(\tau_m,y_{t,x}^{\psi,\beta(\psi)}(\tau_m^-);\xi_m\bigr)\exp\bigl(-\lambda(\tau_m-t)\bigr)\ind_{\{\tau_m<t^\prime\}}\prod_{k\geq 1}\ind_{\{\tau_m\neq\rho_k\}}\\
&+\sum_{k\geq 1}\chi\bigl(\rho_k,y_{t,x}^{\psi,\beta(\psi)}(\rho_k^-);\eta_k\bigr)\exp\bigl(-\lambda(\rho_k-t)\bigr)\ind_{\{\rho_k<t^\prime\}}\\
&+V^-\bigl(t^\prime,y_{t,x}^{\psi,\beta(\psi)}(t^\prime)\bigr)\exp\bigl(-\lambda(t^\prime-t)\bigr)\biggr\},
\end{aligned}
\end{equation}
and
\begin{equation*}
\begin{aligned}
V^+(t,x)=&\sup_{\alpha\in\mathcal{A}}\inf_{v\in\mathcal{V}}\biggl\{\int_t^{t^\prime} f\bigl(s,y_{t,x}^{\alpha(v),v}(s);\theta(s)\bigr)\exp\bigl(-\lambda(s-t)\bigr)ds\\
&-\sum_{m\geq 1}c\bigl(\tau_m,y_{t,x}^{\alpha(v),v}(\tau_m^-);\xi_m\bigr)\exp\bigl(-\lambda(\tau_m-t)\bigr)\ind_{\{\tau_m<t^\prime\}}\prod_{k\geq 1}\ind_{\{\tau_m\neq\rho_k\}}\\
&+\sum_{k\geq 1}\chi\bigl(\rho_k,y_{t,x}^{\alpha(v),v}(\rho_k^-);\eta_k\bigr)\exp\bigl(-\lambda(\rho_k-t)\bigr)\ind_{\{\rho_k<t^\prime\}}\\
&+V^+\bigl(t^\prime,y_{t,x}^{\alpha(v),v}(t^\prime)\bigr)\exp\bigl(-\lambda(t^\prime-t)\bigr)\biggr\}.
\end{aligned}
\end{equation*}
\end{theorem}
\begin{proof}
This proof is inspired by the results in chapter VIII of reference \cite{BC97} (see also \cite{EL21}). We give only the proof for the lower value $V^-$, similarly for the upper value $V^+$. Let $T>0$, $t\in[0,T]$ and $t^\prime\in[t,T]$, fix $\varepsilon>0$ and denote by $W_{t^\prime}(t,x)$ the right-hand side of equation (\ref{equation_2.1}). We first prove that $V^-(t,x)\leq W_{t^\prime}(t,x)$. For any $(s,z)\in[t,T]\times\mathbb{R}^n$ we pick a non-anticipative strategy $\beta_{z}\in\mathcal{B}$ for player$-\eta$ such that
\begin{equation}\label{equation_2.2}
V^-(s,z)\geq \sup_{\psi\in\Psi}J\bigl(s,z;\psi,\beta_{z}(\psi)\bigr)-\varepsilon,
\end{equation}
where $\psi:=\bigl(\theta(.),u\bigr)\in\Psi$. Then we choose $\overline{\beta}\in\mathcal{B}$ a non-anticipative strategy for player$-\eta$ that satisfies, for $u:=(\tau_m,\xi_m)_{m\in\mathbb{N}^*}$, the following inequality:
\begin{equation}\label{equation_2.3}
\begin{aligned}
W_{t^\prime}(t,x)\geq& \sup_{\psi\in\Psi}\biggl\{\int_t^{t^\prime}f\bigl(s,y_{t,x}^{\psi,\overline{\beta}(\psi)}(s);\theta(s)\bigr)\exp\bigl(-\lambda(s-t)\bigr)ds\\
&-\sum_{m\geq 1}c\bigl(\tau_m,y_{t,x}^{\psi,\overline{\beta}(\psi)}(\tau_m^-);\xi_m\bigr)\exp\bigl(-\lambda(\tau_m-t)\bigr)\ind_{\{\tau_m<t^\prime\}}\prod_{k\geq 1}\ind_{\{\tau_m\neq\overline{\rho_k}\}}\\
&+\sum_{k\geq 1}\chi\bigl(\rho_k,y_{t,x}^{\psi,\overline{\beta}(\psi)}(\overline{\rho_k}^-);\overline{\eta_k}\bigr)\exp\bigl(-\lambda(\overline{\rho_k}-t)\bigr)\ind_{\{\rho_k<t^\prime\}}\\
&+V^-\bigl(t^\prime,y_{t,x}^{\psi,\overline{\beta}(\psi)}(t^\prime)\bigr)\exp\bigl(-\lambda(t^\prime-t)\bigr)\biggr\}-\varepsilon,
\end{aligned}
\end{equation}
where $$\overline{\beta}(\psi):=(\overline{\rho_k},\overline{\eta_k})_{k\in\mathbb{N}^*}.$$
Next, we define $\beta\in\mathcal{B}$, a non-anticipative strategy for player$-\eta$, as follows:
\begin{equation*}
\beta\bigl(\theta(s),u\bigr):=\left\{
\begin{aligned}
&\overline{\beta}\bigl(\theta(s),u\bigr),\;s\leq t^\prime;\\
&\beta_{z}\bigl(\theta^{z}(s-t^\prime),u^z\bigr),\;t^\prime<s\leq T,
\end{aligned}
\right.
\end{equation*}
where $z:=y_{t,x}^{\psi,\overline{\beta}(\psi)}(t^\prime)$, $\theta^{z}(.)\equiv\theta(.+t^\prime)$ and $u^z:=(\tau_m^z,\xi_m^z)_{m\in\mathbb{N}^*}$ with $\tau_m^z\in[t^\prime,T]$. Since we have for all $t\in[0,T]$, $$y_{t,x}^{\psi,\beta(\psi)}(t+t^\prime)=y_{t^\prime,z}^{\psi_z,\beta_z(\psi_z)}(t),$$
where $\psi_z:=\bigl(\theta^z(.),u^z\bigr)$, then by the change of variable $s=t+t^\prime$ we get
\begin{equation*}
\begin{aligned}
J\bigl(t^\prime,z;\psi_z,\beta_z(\psi_z)\bigr)=&\int_{t^\prime}^{T} f\bigl(s,y_{t^\prime,z}^{\psi,\beta(\psi)}(s);\theta(s)\bigr)\exp\bigl(-\lambda (s-t^\prime)\bigr)ds\\
&-\sum_{m\geq 1}c\bigl(\tau_m,y_{t^\prime,z}^{\psi,\beta(\psi)}(\tau_m^-);\xi_m\bigr)\exp\bigl(-\lambda(\tau_m-t^\prime)\bigr)\ind_{\{t^\prime\leq\tau_m\leq T\}}\prod_{k\geq 1}\ind_{\{\tau_m\neq\rho_k\}}\\
&+\sum_{k\geq 1}\chi\bigl(\rho_k,y_{t^\prime,z}^{\psi,\beta(\psi)}(\rho_k^-);\eta_k\bigr)\exp\bigl(-\lambda(\rho_k-t^\prime)\bigr)\ind_{\{t^\prime\leq\rho_k\leq T\}},
\end{aligned}
\end{equation*}
where $$\beta(\psi):=(\rho_k,\eta_k)_{k\in\mathbb{N}^*}.$$
Then by (\ref{equation_2.2}) and (\ref{equation_2.3}) we deduce
\begin{equation*}
\begin{aligned}
W_{t^\prime}(t,x)\geq& \sup_{\psi\in\Psi}\biggl\{\int_t^{T}f\bigl(s,y_{t,x}^{\psi,\beta(\psi)}(s);\theta(s)\bigr)\exp\bigl(-\lambda(s-t)\bigr)ds\\
&-\sum_{m\geq 1}c\bigl(\tau_m,y_{t,x}^{\psi,\beta(\psi)}(\tau_m^-);\xi_m\bigr)\exp\bigl(-\lambda(\tau_m-t)\bigr)\prod_{k\geq 1}\ind_{\{\tau_m\neq\rho_k\}}\\
&+\sum_{k\geq 1}\chi\bigl(\rho_k,y_{t,x}^{\psi,\beta(\psi)}(\rho_k^-);\eta_k\bigr)\exp\bigl(-\lambda(\rho_k-t)\bigr)\biggr\}-2\varepsilon\\
&\geq V^-(t,x)-2\varepsilon,
\end{aligned}
\end{equation*}
thus, since $\varepsilon$ is arbitrary, we get the desired inequality. We next prove that $W_{t^\prime}(t,x)\leq V^-(t,x)$. For any $(s,z)\in[t,T]\times\mathbb{R}^n$ we pick the non-anticipative strategy $\beta_z\in\mathcal{B}$ for player$-\eta$ which satisfies the inequality (\ref{equation_2.2}). We then pick $\overline{\psi}:=\bigl(\overline{\theta}(.),\overline{u}:=(\overline{\tau_m},\overline{\xi_m})_{m\in\mathbb{N}^*}\bigr)\in\Psi$, the control for player$-\xi$ that satisfies the following:
\begin{equation}\label{equation_2.4}
\begin{aligned}
W_{t^\prime}(t,x)\leq&\int_t^{t^\prime} f\bigl(s,y_{t,x}^{\overline{\psi},\beta_z(\overline{\psi})}(s);\overline{\theta}(s)\bigr)\exp\bigl(-\lambda(s-t)\bigr)ds\\
&-\sum_{m\geq 1}c\bigl(\overline{\tau_m},y_{t,x}^{\overline{\psi},\beta_z(\overline{\psi})}(\overline{\tau_m}^-);\overline{\xi_m}\bigr)\exp\bigl(-\lambda(\overline{\tau_m}-t)\bigr)\ind_{\{\overline{\tau_m}<t^\prime\}}\prod_{k\geq 1}\ind_{\{\overline{\tau_m}\neq\rho_k^z\}}\\
&+\sum_{k\geq 1}\chi\bigl(\rho_k^z,y_{t,x}^{\overline{\psi},\beta_z(\overline{\psi})}({\rho_k^z}^-);\eta_k^z\bigr)\exp\bigl(-\lambda(\rho_k^z-t)\bigr)\ind_{\{{\rho_k^z}<t^\prime\}}\\
&+V^-\bigl(t^\prime,y_{t,x}^{\overline{\psi},\beta_z(\overline{\psi})}(t^\prime)\bigr)\exp\bigl(-\lambda(t^\prime-t)\bigr)+\varepsilon,
\end{aligned}
\end{equation}
where $$\beta_z(\overline{\psi}):=(\rho_k^z,\eta_k^z)_{k\in\mathbb{N}^*}.$$
For any $\psi:=\bigl(\theta(.),u\bigr)\in\Psi$, we define the control $\tilde{\psi}:=\bigl(\tilde{\theta}(.),\tilde{u}\bigr)\in\Psi$ for player$-\xi$ as follows:
\begin{equation}\label{equation_2.5}
\bigl(\tilde{\theta}(s),\tilde{u}\bigr):=\left\{
\begin{aligned}
&\bigl(\overline{\theta}(s),\overline u\bigr),\;s\leq t^\prime;\\
&\bigl(\theta(s-t^\prime),u\bigr),\;t^\prime<s\leq T,
\end{aligned}
\right.
\end{equation}
where $u:=(\tau_m,\xi_m)_{m\in\mathbb{N}^*}\in\mathcal{U}$ with $\tau_m\in[t^\prime,T]$. Moreover, we define $\beta\in\mathcal{B}$ a non-anticipative strategy for player$-\eta$ as follows:
\begin{equation}\label{equation_2.6}
\beta\bigl(\theta(s),u\bigr):=\beta_z\bigl(\tilde{\theta}(s+t^\prime),\tilde{u}\bigr).
\end{equation}
Next, set
\begin{equation}\label{equation_2.7}
z_1:=y_{t,x}^{\overline{\psi},\beta_{z}(\overline{\psi})}(t^\prime),
\end{equation}
and choose $\psi\in\Psi$ such that
\begin{equation}\label{equation_2.8}
V^-(t^\prime,z_1)\leq J\bigl(t^\prime,z_1;\psi,\beta(\psi)\bigr)+\varepsilon.
\end{equation}
Observe that, by (\ref{equation_2.5}) and (\ref{equation_2.6}), we have
\begin{equation*}
y_{t,x}^{\tilde{\psi},\beta_{z}(\tilde{\psi})}(s)=\left\{
\begin{aligned}
&y_{t,x}^{\overline{\psi},\beta_{z}(\overline{\psi})}(s),\;s\leq t^\prime;\\
&y_{t^\prime,z_1}^{\psi,\beta(\psi)}(s-t^\prime),\;t^\prime<s\leq T,
\end{aligned}
\right.
\end{equation*}
so by the change of variable $s=t+t^\prime$ we deduce for $u:=(\tau_m,\xi_m)_{m\in\mathbb{N}^*}$ that
\begin{equation}\label{equation_2.9}
\begin{aligned}
J\bigl(t^\prime,z_1;\psi,\beta(\psi)\bigr)=&\int_{t^\prime}^{T} f\bigl(s,y_{t^\prime,z_1}^{\tilde{\psi},\beta_{z}(\tilde{\psi})}(s);\tilde{\theta}(s)\bigr)\exp\bigl(-\lambda(s-t^\prime)\bigr)ds\\
&-\sum_{m\geq 1}c\bigl(\tau_m,y_{t^\prime,z_1}^{\tilde{\psi},\beta_{z}(\tilde{\psi})}(\tau_m^-);\xi_m\bigr)\exp\bigl(-\lambda(\tau_m-t^\prime)\bigr)\ind_{\{t^\prime\leq\tau_m\leq T\}}\prod_{k\geq 1}\ind_{\{\tau_m\neq\tilde{\rho}_k\}}\\
&+\sum_{k\geq 1}\chi\bigl(\rho_k,y_{t^\prime,z_1}^{\tilde{\psi},\beta_{z}(\tilde{\psi})}(\tilde{\rho}_k^-);\tilde{\eta}_k\bigr)\exp\bigl(-\lambda(\tilde{\rho}_k-t^\prime)\bigr)\ind_{\{t^\prime\leq\tilde{\rho}_k\leq T\}},
\end{aligned}
\end{equation}
where $$\beta_z\bigl(\tilde{\theta}(.),\tilde{u}\bigr):=(\tilde{\rho}_k,\tilde{\eta}_k)_{k\in\mathbb{N}^*}.$$
Now we use (\ref{equation_2.4}), (\ref{equation_2.5}), (\ref{equation_2.7}), (\ref{equation_2.8}) and (\ref{equation_2.9}) to get
\begin{equation*}
W_{t^\prime}(t,x)\leq J\bigl(t,x;\tilde{\psi},\beta_{z}(\tilde{\psi})\bigr)+2\varepsilon,
\end{equation*}
thus, from inequality (\ref{equation_2.2}), we deduce that $W_{t^\prime}(t,x)\leq V^-(t,x)+3\varepsilon$. Then, since $\varepsilon$ is arbitrary, we obtain the desired inequality.
\end{proof}
\begin{proposition}[State-Continuity]\label{Proposition2.4}
Assume $\textbf{H}_b$, $\textbf{H}_g$, $\textbf{H}_f$, $\textbf{H}_{c,\chi}$ and $\textbf{H}_{G}$. Then there exists a real positive constant $M$ such that for all $x,x^\prime\in\mathbb{R}^n$, and $t\in[0,T]$, the lower value and the upper value satisfy
$$\bigl|v(t,x)-v(t,x^\prime)\bigr|\leq M\|x-x^\prime\|.$$
\end{proposition}
\begin{proof}
\par We give only the proof for the lower value $V^-$, similarly for the upper value $V^+$. Let $t\in[0,T]$, fix $x,x^\prime\in\mathbb{R}^n$ and an arbitrary $\varepsilon>0$, and first pick a non-anticipative strategy $\beta^\varepsilon\in\mathcal{B}$ for minimizing player$-\eta$ such that the following inequality holds true:
$$V^-(t,x^\prime)\geq \sup_{\psi\in\Psi}J\bigl(t,x^\prime;\psi,\beta^\varepsilon(\psi)\bigr)-\frac{\varepsilon}{2},$$
then we choose $\psi^\varepsilon:=\bigl(\theta^\varepsilon(.),u^\varepsilon:=(\tau_m^\varepsilon,\xi_m^\varepsilon)_{m\in\mathbb{N}^*}\bigr)\in\Psi$, an admissible continuous-impulse control for maximizing player$-\xi$, such that
\begin{equation*}
\begin{aligned}
V^-(t,x)&\leq\sup_{\psi\in\Psi}J\Bigl(t,x;\psi,\beta^\varepsilon(\psi)\Bigr)\\
&\leq J\Bigl(t,x;\psi^\varepsilon,\beta^\varepsilon(\psi^\varepsilon)\Bigr)+\frac{\varepsilon}{2}.
\end{aligned}
\end{equation*}
Thus we get
$$V^-(t,x)-V^-(t,x^\prime)\leq J\bigl(t,x;\psi^\varepsilon,\beta^\varepsilon(\psi^\varepsilon)\bigr)-J\bigl(t,x^\prime;\psi^\varepsilon,\beta^\varepsilon(\psi^\varepsilon)\bigr)+\varepsilon.$$
It follows, for $\beta^\varepsilon(\psi^\varepsilon):=(\rho_k^\varepsilon,\eta_k^\varepsilon)_{k\in\mathbb{N}^*}\in\mathcal{V}$, that
\begin{equation*}
\begin{aligned}
V^-(t,x)-V^-(t,x^\prime)\leq&\int_t^{T} \Bigl[f\bigl(s,y_{t,x}^{\psi^\varepsilon,\beta^\varepsilon(\psi^\varepsilon)}(s);\theta^\varepsilon(s)\bigr)-f\bigl(s,y_{t,x^\prime}^{\psi^\varepsilon,\beta^\varepsilon(\psi^\varepsilon)}(s);\theta^\varepsilon(s)\bigr)\Bigr]\exp\bigl(-\lambda(s-t)\bigr)ds\\
&-\sum_{m\geq 1}c\bigl(\tau_m^\varepsilon,y_{t,x}^{\psi^\varepsilon,\beta^\varepsilon(\psi^\varepsilon)}({\tau_m^\varepsilon}^-);\xi_m^\varepsilon\bigr)\exp\bigl(-\lambda(\tau_m^\varepsilon-t)\bigr)\ind_{\{\tau_m^\varepsilon\leq T\}}\prod_{k\geq 1}\ind_{\{\tau_m^\varepsilon\neq\rho_k^\varepsilon\}}\\
&+\sum_{k\geq 1}\chi\bigl(\rho_k^\varepsilon,y_{t,x}^{\psi^\varepsilon,\beta^\varepsilon(\psi^\varepsilon)}({\rho_k^\varepsilon}^-);\eta_k^\varepsilon\bigr)\exp\bigl(-\lambda(\rho_k^\varepsilon-t)\bigr)\ind_{\{\rho_k^\varepsilon\leq T\}}\\
&+\sum_{m\geq 1}c\bigl(\tau_m^\varepsilon,y_{t,x^\prime}^{\psi^\varepsilon,\beta^\varepsilon(\psi^\varepsilon)}({\tau_m^\varepsilon}^-);\xi_m^\varepsilon\bigr)\exp\bigl(-\lambda(\tau_m^\varepsilon-t)\bigr)\ind_{\{\tau_m^\varepsilon\leq T\}}\prod_{k\geq 1}\ind_{\{\tau_m^\varepsilon\neq\rho_k^\varepsilon\}}\\
&-\sum_{k\geq 1}\chi\bigl(\rho_k^\varepsilon,y_{t,x^\prime}^{\psi^\varepsilon,\beta^\varepsilon(\psi^\varepsilon)}({\rho_k^\varepsilon}^-);\eta_k^\varepsilon\bigr)\exp\bigl(-\lambda(\rho_k^\varepsilon-t)\bigr)\ind_{\{\rho_k^\varepsilon\leq T\}}\\
&+\Bigl[G\bigl(y_{t,x}^{\psi^\varepsilon,\beta^\varepsilon(\psi^\varepsilon)}(T)\bigr)-G\bigl(y_{t,x^\prime}^{\psi^\varepsilon,\beta^\varepsilon(\psi^\varepsilon)}(T)\bigr)\Bigr]\exp\bigl(-\lambda(T-t)\bigr)+\varepsilon.
\end{aligned}
\end{equation*}
Then, from the DPP property (\ref{equation_2.1}) for $t^\prime>t$, we get
\begin{equation*}
\begin{aligned}
V^-(t,x)-V^-(t,x^\prime)\leq&\int_t^{t^\prime}\Bigl[f\bigl(s,y_{t,x}^{\psi^\varepsilon,\beta^\varepsilon(\psi^\varepsilon)}(s);\theta^\varepsilon(s)\bigr)-f\bigl(s,y_{t,x^\prime}^{\psi^\varepsilon,\beta^\varepsilon(\psi^\varepsilon)}(s);\theta^\varepsilon(s)\bigr)\Bigr]\exp\bigl(-\lambda(s-t)\bigr)ds\\
&-\sum_{m\geq 1}c\bigl(\tau_m^\varepsilon,y_{t,x}^{\psi^\varepsilon,\beta^\varepsilon(\psi^\varepsilon)}({\tau_m^\varepsilon}^-);\xi_m^\varepsilon\bigr)\exp\bigl(-\lambda(\tau_m^\varepsilon-t)\bigr)\ind_{\{\tau_m^\varepsilon<t^\prime\}}\prod_{k\geq 1}\ind_{\{\tau_m^\varepsilon\neq\rho_k^\varepsilon\}}\\
&+\sum_{k\geq 1}\chi\bigl(\rho_k^\varepsilon,y_{t,x}^{\psi^\varepsilon,\beta^\varepsilon(\psi^\varepsilon)}({\rho_k^\varepsilon}^-);\eta_k^\varepsilon\bigr)\exp\bigl(-\lambda(\rho_k^\varepsilon-t)\bigr)\ind_{\{\rho_k^\varepsilon<t^\prime\}}\\
&+\sum_{m\geq 1}c\bigl(\tau_m^\varepsilon,y_{t,x^\prime}^{\psi^\varepsilon,\beta^\varepsilon(\psi^\varepsilon)}({\tau_m^\varepsilon}^-);\xi_m^\varepsilon\bigr)\exp\bigl(-\lambda(\tau_m^\varepsilon-t)\bigr)\ind_{\{\tau_m^\varepsilon<t^\prime\}}\prod_{k\geq 1}\ind_{\{\tau_m^\varepsilon\neq\rho_k^\varepsilon\}}\\
&-\sum_{k\geq 1}\chi\bigl(\rho_k^\varepsilon,y_{t,x^\prime}^{\psi^\varepsilon,\beta^\varepsilon(\psi^\varepsilon)}({\rho_k^\varepsilon}^-);\eta_k^\varepsilon\bigr)\exp\bigl(-\lambda(\rho_k^\varepsilon-t)\bigr)\ind_{\{\rho_k^\varepsilon<t^\prime\}}\\
&+\Bigl[V^-\bigl(t^\prime,y_{t,x}^{\psi^\varepsilon,\beta^\varepsilon(\psi^\varepsilon)}(t^\prime)\bigr)-V^-\bigl(t^\prime,y_{t,x^\prime}^{\psi^\varepsilon,\beta^\varepsilon(\psi^\varepsilon)}(t^\prime)\bigr)\Bigr]\exp\bigl(-\lambda(t^\prime-t)\bigr)+\varepsilon.
\end{aligned}
\end{equation*}
Thus, by assumptions on functions $f,c$ and $\chi$, we get
\begin{equation*}
\begin{aligned}
V^-(t,x)-&V^-(t,x^\prime)\leq\int_t^{t^\prime} C_f\bigl\|y_{t,x}^{\psi^\varepsilon,\beta^\varepsilon(\psi^\varepsilon)}(s)-y_{t,x^\prime}^{\psi^\varepsilon,\beta^\varepsilon(\psi^\varepsilon)}(s)\bigr\|\exp\bigl(-\lambda(s-t)\bigr)ds\\
&-\sum_{m\geq 1}C_c\bigl\|y_{t,x}^{\psi^\varepsilon,\beta^\varepsilon(\psi^\varepsilon)}({\tau_m^\varepsilon}^-)-y_{t,x^\prime}^{\psi^\varepsilon,\beta^\varepsilon(\psi^\varepsilon)}({\tau_m^\varepsilon}^-)\bigr\|\exp\bigl(-\lambda(\tau_m^\varepsilon-t)\bigr)\ind_{\{\tau_m^\varepsilon<t^\prime\}}\prod_{k\geq 1}\ind_{\{\tau_m^\varepsilon\neq\rho_k^\varepsilon\}}\\
&+\sum_{k\geq 1}C_{\chi}\bigl\|y_{t,x}^{\psi^\varepsilon,\beta^\varepsilon(\psi^\varepsilon)}({\rho_k^\varepsilon}^-)-y_{t,x^\prime}^{\psi^\varepsilon,\beta^\varepsilon(\psi^\varepsilon)}({\rho_k^\varepsilon}^-)\bigr\|\exp\bigl(-\lambda(\rho_k^\varepsilon-t)\bigr)\ind_{\{\rho_k^\varepsilon<t^\prime\}}\\
&+\Bigl|V^-\bigl(t^\prime,y_{t,x}^{\psi^\varepsilon,\beta^\varepsilon(\psi^\varepsilon)}(t^\prime)\bigr)-V^-\bigl(t^\prime,y_{t,x^\prime}^{\psi^\varepsilon,\beta^\varepsilon(\psi^\varepsilon)}(t^\prime)\bigr)\Bigr|\exp\bigl(-\lambda (t^\prime-t)\bigr)+\varepsilon.
\end{aligned}
\end{equation*}
By Propositions \ref{Proposition2.1} and \ref{Proposition2.2}, we deduce that there exist some constants $C>0$ and $C_v>0$ such that
\begin{equation}\label{equation_2.10}
\begin{aligned}
V^-(t,x)-V^-(t,x^\prime)\leq&\;C_f\bigl\|x-x^\prime\bigr\|\int_t^{t^\prime}\exp\bigl((C-\lambda)(s-t)\bigr)ds\\
&-C_c\bigl\|x-x^\prime\bigr\|\sum_{m\geq 1}\exp\bigl((C-\lambda)(\tau_m^\varepsilon-t)\bigr)\ind_{\{\tau_m^\varepsilon<t^\prime\}}\prod_{k\geq 1}\ind_{\{\tau_m^\varepsilon\neq\rho_k^\varepsilon\}}\\
&+C_{\chi}\bigl\|x-x^\prime\bigr\|\sum_{k\geq 1}\exp\bigl((C-\lambda)(\rho_k^\varepsilon-t)\bigr)\ind_{\{\rho_k^\varepsilon<t^\prime\}}\\
&+2C_v\exp\bigl(-\lambda(t^\prime-t)\bigr)+\varepsilon.
\end{aligned}
\end{equation}
Now, if $C<\lambda$ the sums in the right-hand side of (\ref{equation_2.10}) are finite, then there exists a positive constant $K$ such that we have
\begin{equation}\label{equation_2.11}
\begin{aligned}
V^-(t,x)-V^-(t,x^\prime)\leq&\;\frac{C_f}{C-\lambda}\bigl\| x-x^\prime\bigr\|\Bigl[\exp\bigl((C-\lambda)(t^\prime-t)\bigr)-1\Bigr]\\
&+K\bigl\|x-x^\prime\bigr\|+2C_v\exp\bigl(-\lambda(t^\prime-t)\bigr)+\varepsilon,
\end{aligned}
\end{equation}
tacking into account the boundedness of $\exp\bigl((C-\lambda)(t^\prime-t)\bigr)$ and letting $t^\prime$ be such that $\exp\bigl(-\lambda(t^\prime-t)\bigr)=\bigl\|x-x^\prime\bigr\|$ with $\bigl\|x-x^\prime\bigr\|<1$ for any $t\in[0,T]$, then using the arbitrariness of $\varepsilon$ and the fact that $x$ and $x^\prime$ play symmetrical roles in the left hand side of the above inequality one might deduce the existence of a positive constant $M$ which satisfies
$$\bigl|V^-(t,x)-V^-(t,x^\prime)\bigr|\leq M\|x-x^\prime\|\;\text{for all}\;t\in[0,T].$$
In the case where $\lambda<C$, we choose $t^\prime$ such that $\exp\bigl(-C(t^\prime-t)\bigr)=\bigl\|x-x^\prime\bigr\|^{1/2}$ with $\bigl\|x-x^\prime\bigr\|<1$. Hence, in the right-hand side of (\ref{equation_2.11}), the first term equals to
$$\frac{C_f}{C-\lambda}\bigl\|x-x^\prime\bigr\|^{1/2}\Bigl(\exp\bigl(-\lambda (t^\prime-t)\bigr)-\bigl\|x-x^\prime\bigr\|^{1/2}\Bigr),$$
where the term $\exp\bigl(-\lambda(t^\prime-t)\bigr)$ is bounded for any $t\in[0,T]$. We then deduce from the fact that $\exp\bigl(-\lambda(t^\prime-t)\bigr)=\bigl\|x-x^\prime\bigr\|^{1/2}\exp\bigl((C-\lambda)(t^\prime-t)\bigr)$ and the arbitrariness of $\varepsilon$ that there exists a positive constant $M_1$ which satisfies $$V^-(t,x)-V^-(t,x^\prime)\leq M_1\|x-x^\prime\|\;\text{for all}\;t\in[0,T],$$
again the roles of $x$ and $x^\prime$ being symmetrical, we then conclude. Finally, in the case where $C=\lambda$, it suffice to let some constant $\hat{\lambda}<\lambda=C$, so we go back to the above inequality (\ref{equation_2.10}) and we proceed, since
$$\exp\bigl((C-\lambda) (t^\prime-t)\bigr)<\exp\bigl((C-\hat{\lambda})(t^\prime-t)\bigr)\;\text{and}\; \exp\bigl(-\lambda(t^\prime-t)\bigr)<\exp\bigl(-\hat{\lambda}(t^\prime-t)\bigr),$$
as above with the case $\hat{\lambda}\neq C$. Thus the lower value function is Lipschitz-continuous w.r.t. state variable, which completes the proof.
\end{proof}
Next, we give the uniform continuity of the functions $x\rightarrow\mathcal{H}_{inf}^\chi v(t,x)$ and $x\rightarrow\mathcal{H}_{sup}^c v(t,x)$ for $t\in[0,T]$, the proof in obvious.
\begin{proposition}\label{Proposition2.5}
Let $t\in[0,T]$ and $x\rightarrow v(t,x)$ be a uniformly continuous function in $\mathbb{R}^n$. Then the two functions $x\rightarrow\mathcal{H}_{inf}^\chi v(t,x)$ and $x\rightarrow\mathcal{H}_{sup}^c v(t,x)$ are uniformly continuous in $\mathbb{R}^n$.
\end{proposition}
\subsection{Hamilton-Jacobi-Bellman-Isaacs Equation and Approximate Equation}\label{Sect.2.3}
Since in the definition of the lower value the inf is taken over non-anticipative strategies whereas in the definition of the upper value it is taken over admissible controls, and similarly the sup is taken over different sets in the definitions of the lower and the upper value functions, then the inequality $V^+(t,x)\leq V^-(t,x)$ for all $(t,x)\in[0,T]\times\mathbb{R}^n$ is false in general. In addition the inequality $V^-(t,x)\leq V^+(t,x)$ for all $(t,x)\in[0,T]\times\mathbb{R}^n$ is not obvious at first glance. We then prove, in a rather indirect way by using the associated HJBI equation, that the zero-sum DG control problem studied has a value. The \textit{dynamic programming equation (DPE)} associated to our deterministic finite-time horizon, two-player, zero-sum DG control problem, which turns out to be the same for the two value functions because the two players cannot act simultaneously on the system, is derived from DPP and is given by the following expression:
\begin{equation*}
\text{(HJBI)}\;\left\{
\begin{aligned}
&
\begin{aligned}
\max\biggl\{\min\Bigl[&-\frac{\partial}{\partial s}v(s,y)+\lambda v(s,y)+H\bigl(s,y,D_{y}v(s,y)\bigr),v(s,y)-\mathcal{H}_{sup}^c v(s,y)\Bigr];\\
&v(s,y)-\mathcal{H}_{inf}^\chi v(s,y)\biggr\}=0,\;\text{on}\;[t,T)\times\mathbb{R}^n;
\end{aligned}\\
&v(T,y)=G(y)\;\text{for all}\;y\in\mathbb{R}^n,
\end{aligned}
\right.
\end{equation*}
where $\frac{\partial}{\partial s}v(s,y)$ denotes the time derivative, and $D_{y}v(s,y)$ the spatial gradient of the function $v(s,y):[t,T]\times\mathbb{R}^n\rightarrow\mathbb{R}$, with $D_y:=\bigl(\frac{\partial}{\partial y_1},\dots,\frac{\partial}{\partial y_n}\bigr)^\top$. The associated \textit{first-order Hamiltonian} ($H$) and the two obstacles, defined through the use of the \textit{maximum} and \textit{minimum non-local cost operators} ($\mathcal{H}_{sup}^c$) and ($\mathcal{H}_{inf}^\chi$), respectively, are given by the following:
\begin{definition}[Hamiltonian and Cost Operators]
For any function $v:[t,T]\times\mathbb{R}^n\rightarrow\mathbb{R}$, we define the first-order Hamiltonian ($H$) by:
\begin{equation*}
H\bigl(s,y,D_{y}v(s,y)\bigr):=\inf_{\theta\in\mathbb{R}^l}\bigl\{-D_{y}v(s,y).b(s,y;\theta)-f(s,y;\theta)\bigr\},
\end{equation*}
where $"."$ denotes the inner product in $\mathbb{R}^n$, and the two non-local cost operators ($\mathcal{H}_{sup}^c$) and ($\mathcal{H}_{inf}^\chi$) by:
\begin{equation*}
\begin{aligned}
\mathcal{H}_{sup}^c v(s,y):=&\;\sup_{\xi\in U}\Bigl\{v\bigl(s,y+g_\xi(s,y;\xi)\bigr)-c(s,y;\xi)\Bigr\};\\
\mathcal{H}_{inf}^\chi v(s,y):=&\;\inf_{\eta\in V}\Bigl\{v\bigl(s,y+g_\eta(s,y;\eta)\bigr)+\chi(s,y;\eta)\Bigr\}.
\end{aligned}
\end{equation*}\qed
\end{definition}
\par We propose in this paper, for all $(s,y)\in[t,T]\times\mathbb R^n$, the approximate equation (HJBI$_{h}$) an approximation of the classic HJBI equation. Let $h$ be the time discretization step for the approximation we will be given, $h_0$ be a positive number, and $\Phi(h)$ be a continuous function such that $\Phi(0)=1$ and $0<\Phi(h)<1$ for $0<h<h_0$:
\begin{equation*}
\text{(HJBI$_h$)}\;\left\{
\begin{aligned}
&
\begin{aligned}
\max\biggl\{\min\Bigl[&H_h\bigl(s,y,v_h(s,y)\bigr),v_h(s,y)-\Phi(h)\mathcal{H}_{sup}^c v_h(s,y)\Bigr];v_h(s,y)-\Phi(h)\mathcal{H}_{inf}^\chi v_h(s,y)\biggr\}=0,\\
&\;\text{on}\;[t,T)\times\mathbb{R}^n;
\end{aligned}\\
&v_h(T,y)=G(y)\;\text{for all}\;y\in\mathbb{R}^n,
\end{aligned}
\right.
\end{equation*}
where the approximate Hamiltonian ($H_h$) is defined as follows:
\begin{definition}[Approximate Hamiltonian]\label{def2.5}
For any function $v_h:[t,T]\times\mathbb{R}^n\rightarrow\mathbb{R}$, we define the approximate Hamiltonian ($H_h$) by:
\begin{equation*}
H_h\bigl(s,y,v_h(s,y)\bigr):=\inf_{\theta\in\mathbb{R}^l}\Bigl\{v_h(s,y)-(1-\lambda h)v_h\bigl(s+h,y+hb(s,y;\theta)\bigr)-hf(s,y;\theta)\Bigr\}.
\end{equation*}\qed
\end{definition}
\begin{remark}\label{Remark2.3}
The contribution of the paper is four-fold:
\begin{enumerate}
\item First, we prove that the lower value $V^-$ and the upper value $V^+$ are viscosity solutions to the HJBI equation. Then we show, by proving a comparison principle, that the HJBI equation has a unique solution in viscosity sense, i.e., the zero-sum DG control problem studied admits the value $V$;
\item Second, we prove that an approximate value function $v_h$ exists, that it is the unique solution of the approximate equation (HJBI$_{h}$). Then we show that $v_h$ converges, as the time discretization step $h$ goes to zero, locally uniformly towards the value function of the zero-sum DG control problem;
\item Third, we prove a verification theorem for the zero-sum DG control problem considered, that is, the game has a NE strategies. This result will lead to some computational algorithms for the zero-sum DG studied;
\item Fourth, we apply our theory to continuous-time portfolio optimization problem to derive a new optimization model and to give a new portfolio strategy. \qed
\end{enumerate}
\end{remark}
\section{Viscosity Characterization of the Value Functions}\label{Sect.3}
The \textit{value function} of an OC problem is a solution to the corresponding
\textit{HJB} (or, \textit{HJBI}) \textit{equation} whenever it has sufficient regularity (see e.g. \cite{FS06}). In other word, it requires that the HJB (or, HJBI) equation admits classical solutions, meaning that the solutions be smooth enough. Unfortunately, this is not necessarily the case even in very simple cases. To overcome this difficulty, the so-called \textit{viscosity solution (VS)} was introduced in the early 80's \cite{CL83,CEL84,CIL92}. This new notion is a kind of non-smooth solutions, where if the value function is continuous, then, it is a solution to the HJB (or, HJBI) equation in the VS sense, whose key feature is to replace the conventional derivatives while maintaining the uniqueness of solutions under very mild conditions. These make the theory of VS a powerful tool in tackling OC problems and DGs \cite{BC97,ES84,FS06,Pha09,S85,S85'}. We recall here the definition of a VS of the HJBI equation following \cite{CL83,CEL84,CIL92}:
\begin{definition}[Viscosity Solution]\label{definition3.1}
Let $v:[t,T]\times\mathbb{R}^n\rightarrow\mathbb{R}$ be a continuous function such that $v(T,y)=G(y)$ for any $y\in\mathbb{R}^n$. $v$ is called:
\begin{enumerate}[i.]
\item A viscosity sub-solution of the Hamilton-Jacobi-Bellman-Isaacs equation (HJBI) if for any $(\overline s,\overline y)\in [t,T)\times\mathbb{R}^n$ and any function $\phi\in C^{1,1}\bigl([t,T)\times\mathbb{R}^n\bigr)$ such that $v(\overline s,\overline y)=\phi(\overline s,\overline y)$ and $(\overline s,\overline y)$ is a local maximum point of $v-\phi$, we have
\begin{equation*}
\begin{aligned}
\max\biggl\{\min\Bigl[&-\frac{\partial\phi}{\partial s}(\overline s,\overline y)+\lambda v(\overline s,\overline y)+H\bigl(\overline s,\overline y,D_{y}\phi(\overline s,\overline y)\bigr),v(\overline s,\overline y)-\mathcal{H}_{sup}^c v(\overline s,\overline y)\Bigr];\\
&v(\overline s,\overline y)-\mathcal{H}_{inf}^\chi v(\overline s,\overline y)\biggr\}\leq 0;
\end{aligned}
\end{equation*}
\item A viscosity super-solution of the Hamilton-Jacobi-Bellman-Isaacs equation (HJBI) if for any $(\underline s,\underline y)\in [t,T)\times\mathbb{R}^n$ and any function $\phi\in C^{1,1}\bigl([t,T)\times\mathbb{R}^n\bigr)$ such that $v(\underline s,\underline y)=\phi(\underline s,\underline y)$ and $(\underline s,\underline y)$ is a local minimum point of $v-\phi$, we have
\begin{equation*}
\begin{aligned}
\max\biggl\{\min\Bigl[&-\frac{\partial\phi}{\partial s}(\underline s,\underline y)+\lambda v(\underline s,\underline y)+H\bigl(\underline s,\underline y,D_{y}\phi(\underline s,\underline y)\bigr),v(\underline s,\underline y)-\mathcal{H}_{sup}^c v(\underline s,\underline y)\Bigr];\\
&v(\underline s,\underline y)-\mathcal{H}_{inf}^\chi v(\underline s,\underline y)\biggr\}\geq 0;
\end{aligned}
\end{equation*}
\item A viscosity solution of the Hamilton-Jacobi-Bellman-Isaacs equation (HJBI) if it is both a viscosity sub-solution and super-solution of the Hamilton-Jacobi-Bellman-Isaacs equation (HJBI). \qed
\end{enumerate}
\end{definition}
The remainder of this section deals with the first contribution of the paper, as it is mentioned in Remark \ref{Remark2.3}.
\subsection{Existence of Viscosity Solutions for the HJBI Equation}
The main result of this section, Theorem \ref{theorem3.1}, shows that the value functions of the zero-sum DG control problem studied satisfy the HJBI equation in VS sense. The proof of this theorem requires the following technical Lemmas \ref{Lemma3.1} and \ref{Lemma3.2}:
\begin{lemma}\label{Lemma3.1}
Let assumptions $\textbf{H}_b$, $\textbf{H}_g$, $\textbf{H}_f$, $\textbf{H}_{c,\chi}$ and $\textbf{H}_{G}$ hold. Given any $(s,y)\in[t,T)\times\mathbb{R}^n$, the lower value and the upper value satisfy the following equation:
$$\max\Bigl\{\min\bigl[0,v(s,y)-\mathcal{H}_{sup}^c v(s,y)\bigr];v(s,y)-\mathcal{H}_{inf}^\chi v(s,y)\Bigr\}=0.$$
\end{lemma}
\begin{proof}
We give only the proof for the lower value $V^-$, similarly for the upper value $V^+$. Let $(s,y)\in[t,T)\times\mathbb{R}^n$, $\psi:=\bigl(\theta(.),u:=(\tau_m,\xi_m)_{m\in\mathbb{N}^*}\bigr)\in\Psi$, then consider the non-anticipative strategy $\beta\in\mathcal{B}$ for player$-\eta$ where $\beta(\psi)=v:=(\rho_k,\eta_k)_{k\in\mathbb{N}^*}\in\mathcal{V}$. Now choose $\beta^\prime\in\mathcal{B}$ such that $\beta^\prime\bigl(\theta(.),u\bigr):=\bigl(s,\eta;\rho_2,\eta_2;\rho_3,\eta_3;\dots\bigr)$, we then obtain $$V^-(s,y)\leq\sup_{\psi\in\Psi}J\bigl(s,y;\psi,\beta^\prime(\psi)\bigr),$$
thus
$$V^-(s,y)=\sup_{\psi\in\Psi}J\bigl(s,y+g_\eta(s,y;\eta);\psi,\beta(\psi)\bigr)+\chi(s,y;\eta),$$
from which we get
$$V^-(s,y)\leq V^-\bigl(s,y+g_\eta(s,y;\eta)\bigr)+\chi(s,y;\eta).$$
Then from the arbitrariness of $\eta$ we get $$V^-(s,y)\leq\mathcal{H}_{inf}^\chi V^-(s,y).$$ Next, we assume that $V^-(s,y)<\mathcal{H}_{inf}^\chi V^-(s,y)$ for some $(s,y)\in[t,T)\times\mathbb{R}^n$. The dynamic programming property (\ref{equation_2.1}) for the lower value, when $T=0$, yields
\begin{equation*}
\begin{aligned}
V^{-}(s,y)=&\inf_{\rho_0\in\{s,T\},\;\eta\in V}\sup_{\underset{\tau_0\in\{s,T\},\;\xi\in U}{\theta(.)\in\Theta}}\Bigl[ -c(s,y;\xi)\ind_{\{\tau_0=s\}}\ind_{\{\rho_0=T\}}+\chi(s,y;\eta)\ind_{\{\rho_0=s\}}\\
&+V^-\bigl(s,y+g_\xi(s,y;\xi)\ind_{\{\tau_0=s\}}\ind_{\{\rho_0=T\}}+g_\eta(s,y;\eta)\ind_{\{\rho_0=s\}}\bigr)\Bigr],
\end{aligned}
\end{equation*}
therefore
\begin{equation*}
\begin{aligned}
V^-(s,y)=&\inf_{\rho_0\in\{s,T\}}\biggl[\inf_{\eta\in V}\Bigl[\chi(s,y;\eta)+V^-\bigl(s,y+g_\eta(s,y;\eta)\bigr)\Bigr]\ind_{\{\rho_0=s\}}\\
&+\sup_{\underset{\tau_0\in\{s,T\},\;\xi\in U}{\theta(.)\in\Theta}}\Bigl[-c(s,y;\xi)\ind_{\{\tau_0=s\}}+V^-\bigl(s,y+g_\xi(s,y;\xi)\ind_{\{\tau_0=0\}}\bigr)\Bigr]\ind_{\{\rho_0=T\}}\biggr].
\end{aligned}
\end{equation*}
From the fact that $V^-(s,y)<\mathcal{H}_{inf}^\chi V^-(s,y)$, we get
$$V^-(s,y)=\sup_{\tau_0\in\{s,T\},\;\xi\in U}\Bigl[-c(s,y;\xi)\ind_{\{\tau_0=s\}}+V^-\bigl(s,y+g_\xi(s,y;\xi)\ind_{\{\tau_0=s\}}\bigr)\Bigr].$$
Hence
$$V^-(s,y)\geq\sup_{\xi\in U}\Bigl[V^-\bigl(s,y+g_\xi(s,y;\xi)\bigr)-c(s,y;\xi)\Bigr],$$
which completes the proof.
\end{proof}
\begin{remark}\label{Remark3.1}
From the above Lemma \ref{Lemma3.1} one may deduce, for any $(s,y)\in[t,T)\times\mathbb{R}^n$, that the lower value and the upper value satisfy the following:
\begin{enumerate}[i.]
\item $v(s,y)\leq\mathcal{H}_{inf}^\chi v(s,y)$;
\item $\text{If}\;v(s,y)<\mathcal{H}_{inf}^\chi v(s,y)\;\text{then}\;v(s,y)\geq\mathcal{H}_{sup}^c v(s,y).$
\end{enumerate}
So we may regard $\mathcal{H}_{sup}^c v(s,y)$ as a lower obstacle and $\mathcal{H}_{inf}^\chi v(s,y)$ as an upper obstacle. \qed
\end{remark}
\begin{lemma}\label{Lemma3.2}
Let assumptions $\textbf{H}_b$, $\textbf{H}_g$, $\textbf{H}_f$, $\textbf{H}_{c,\chi}$ and $\textbf{H}_{G}$ hold. Given any function $\phi\in C^{1,1}\bigl([t,T)\times\mathbb{R}^n\bigr)$ such that for all $(s,y)\in[t,T)\times\mathbb{R}^n$ we have
$$-\frac{\partial}{\partial s}\phi(s,y)+\lambda\phi(s,y)+H\bigl(s,y,D_y\phi(s,y)\bigr)=\gamma>0,$$
then there exists a non-anticipative strategy $\beta^\gamma\in\mathcal{B}$ for minimizing player$-\eta$ such that, for any $\psi:=\bigl(\theta(.),u\bigr)\in\Psi$ and $s$ tends to $t$, we have that
\begin{equation*}
\begin{aligned}
\int_{t}^{s}\Bigl\{&\frac{\partial}{\partial r}\phi\bigl(r,y_{t,x}^{\psi,\beta^\gamma(\psi)}(r)\bigr)-\lambda \phi\bigl(r,y_{t,x}^{\psi,\beta^\gamma(\psi)}(r)\bigr)+D_y\phi\bigl(r,y_{t,x}^{\psi,\beta^\gamma(\psi)}(r)\bigr).b\bigl(r,y_{t,x}^{\psi,\beta^\gamma(\psi)}(r);\theta(r)\bigr)\\
&+f\bigl(r,y_{t,x}^{\psi,\beta^\gamma(\psi)}(r);\theta(r)\bigr)\Bigr\}\exp\bigl(-\lambda(r-t)\bigr)dr\leq-\frac{\gamma}{4}(s-t),
\end{aligned}
\end{equation*}
where $\beta^\gamma(\psi)\in\mathcal{V}.$
\end{lemma}
\begin{proof}
Let $(s,y)\in[t,T)\times\mathbb{R}^n$ and $\phi\in C^{1,1}\bigl([t,T)\times\mathbb{R}^n\bigr)$ be such that
\begin{equation}\label{equation_3.1}
-\frac{\partial}{\partial s}\phi(s,y)+\lambda\phi(s,y)+H\bigl(s,y,D_y\phi(s,y)\bigr)=\gamma>0.
\end{equation}
Following \cite{BC97}, we define for $s^\prime>0$, $z\in\mathbb{R}^n$, $\theta(.)\in\Theta$,
$$\Gamma\bigl(s^\prime,z;\theta(s^\prime)\bigr)=-\frac{\partial}{\partial s}\phi(s^\prime,z)+\lambda\phi(s^\prime,z)-D_y\phi(s^\prime,z).b\bigl(s^\prime,z;\theta(s^\prime)\bigr)-f\bigl(s^\prime,z;\theta(s^\prime)\bigr).$$
By (\ref{equation_3.1}) and the definition of the Hamiltonian ($H$) we get $$\inf_{\theta\in\mathbb{R}^l}\Gamma(s,y;\theta)=\gamma,$$
then for any $\theta(.)\in\Theta$ we have $\Gamma\bigl(s,y;\theta(t)\bigr)\geq\gamma$. Since $\theta\rightarrow\Gamma\bigl(s,y;\theta\bigr)$ is UC in $\mathbb{R}^l$, we have in fact
$$\Gamma\bigl(s,y;\zeta(.)\bigr)\geq\frac{3\gamma}{4}\;\text{for all}\;\zeta(.)\in B_{r(.)}\bigl(\theta(.)\bigr)\cap\Theta,$$
where $B_{r(.)}\bigl(\theta(.)\bigr)$ denotes the open ball of radius $r(.):=r\bigl(\theta(.)\bigr)>0$ centered at $\theta(.)$. Without loss of generality, for $\kappa$ a compact subset of $\mathbb{R}^l$ and $\Theta$ being $\kappa-$valued, there exist finitely many points $\bigl(\theta_1(.),\theta_2(.),\dots,\theta_n(.)\bigr)$ and $\bigl(r_1(.),r_2(.),\dots,r_n(.)\bigr)$ such that $\theta_i(.)\subset\kappa$, $r_i(.)>0$ for $i=1,2,\dots,n$ and $$\Theta\subseteq\cup_{i=1}^{n}B_{r_i(.)}\bigl(\theta_i(.)\bigr),$$
where $r_i(.):=r_i\bigl(\theta_i(.)\bigr)>0$, and
$$\Gamma\bigl(s,y;\zeta(.)\bigr)\geq\frac{3\gamma}{4}\;\text{for all}\;\zeta(.)\in B_{r_i(.)}\bigl(\theta_i(.)\bigr)\cap\Theta.$$
By the continuity of $\Gamma$ and Proposition \ref{Proposition2.1} there exists $t^\prime>0$ such that
$$\Gamma\Bigl(s,y_{t,x}(s);\theta(s)\Bigr)\geq\frac{\gamma}{2}\;\text{for all}\;t\leq s\leq t^\prime\;\text{and all}\;\theta(.)\in\Theta.$$
Finally we multiply both sides of the last inequality by $\exp(-\lambda s)$ and integrate from $t$ to $t^\prime$ to obtain the result for $t^\prime-t$ small enough.
\end{proof}
We are now in a position to prove the following Theorem \ref{theorem3.1}:
\begin{theorem}\label{theorem3.1}
Assume $\textbf{H}_b$, $\textbf{H}_g$, $\textbf{H}_f$ $\textbf{H}_{c,\chi}$ and $\textbf{H}_{G}$. The lower value and the upper value are viscosity solutions to the Hamilton-Jacobi-Bellman-Isaacs equation (HJBI).
\end{theorem}
\begin{proof}
We give the proof for the lower value $V^-$, similarly for the upper value $V^+$. The proof is inspired from \cite{Ba85,BC97} and based on DPP. We start by proving the sub-solution property. Let $\phi$ be a function in $C^{1,1}\bigl([t,T)\times\mathbb{R}^n\bigr)$ and $(\overline{t},\overline{x})\in[t,T)\times\mathbb{R}^n$ be such that $V^{-}-\phi$ achieves a local maximum at $(\overline{t},\overline{x})$ and $V^-(\overline{t},\overline{x})=\phi(\overline{t},\overline{x})$. If $V^-(\overline{t},\overline{x})-\mathcal H_{sup}^c V^-(\overline{t},\overline{x})\leq 0$ the prove is finished, since from Remark \ref{Remark3.1} $V^-(\overline{t},\overline{x})\leq\mathcal H_{inf}^\chi V^-(\overline{t},\overline{x})$. Otherwise, for $\varepsilon>0$ and without loss of generality, we assume that $V^-(\overline{t},\overline{x})-\mathcal H_{sup}^c V^-(\overline{t},\overline{x})\geq\varepsilon>0$, then we proceed by contradiction. Since, from Remark \ref{Remark3.1}, we have $V^-(\overline{t},\overline{x})\leq\mathcal H_{inf}^\chi V^-(\overline{t},\overline{x})$, we now explore the result of Lemma \ref{Lemma3.2} by assuming first that
$$-\frac{\partial\phi}{\partial s}(\overline{t},\overline{x})+\lambda\phi(\overline{t},\overline{x})+H\bigl(\overline{t},\overline{x},D_y\phi(\overline{t},\overline{x})\bigr)=\gamma>0,$$
then one can find a non-anticipative strategy $\beta^\gamma\in\mathcal{B}$ for minimizing player$-\eta$ such that, for any $\psi:=\bigl(\theta(.),u\bigr)\in\Psi$ and $s$ tends to $t$, we have that
\begin{equation*}
\begin{aligned}
\int_{t}^{s}\Bigl\{&\frac{\partial}{\partial r}\phi\bigl(r,y_{\overline t,\overline x}^{\psi,\beta^\gamma(\psi)}(r)\bigr)-\lambda \phi\bigl(r,y_{\overline t,\overline x}^{\psi,\beta^\gamma(\psi)}(r)\bigr)+D_y\phi\bigl(r,y_{\overline t,\overline x}^{\psi,\beta^\gamma(\psi)}(r)\bigr).b\bigl(r,y_{\overline t,\overline x}^{\psi,\beta^\gamma(\psi)}(r);\theta(r)\bigr)\\
&+f\bigl(r,y_{\overline t,\overline x}^{\psi,\beta^\gamma(\psi)}(r);\theta(r)\bigr)\Bigr\}\exp\bigl(-\lambda (r-t)\bigr)dr\leq-\frac{\gamma}{4}(s-t),
\end{aligned}
\end{equation*}
where $\beta^\gamma(\psi)\in\mathcal{V}$. Thus,
\begin{equation}\label{equation_3.2}
\int_{t}^{s}f\bigl(r,y_{\overline t,\overline x}^{\psi,\beta^\gamma(\psi)}(r);\theta(r)\bigr)\exp\bigl(-\lambda(r-t)\bigr)dr+\exp\bigl(-\lambda(s-t)\bigr)\phi\bigl(s,y_{\overline t,\overline x}^{\psi,\beta^\gamma(\psi)}(s)\bigr)-\phi(\overline t,\overline x)\leq-\frac{\gamma}{4}(s-t).
\end{equation}
Since $V^{-}-\phi$ has a local maximum at $(\overline t,\overline x)$ and $V^{-}(\overline t,\overline x)=\phi(\overline t,\overline x)$ we have, for $t=\overline{t}$ and $s-t$ small enough, that
$$\bigl\|y_{\overline t,\overline x}^{\psi,\beta^\gamma(\psi)}(s)-\overline x\bigr\|\rightarrow 0,$$
which yields
$$\exp\bigl(-\lambda(s-t)\bigr)\phi\bigl(s,y_{\overline t,\overline x}^{\psi,\beta^\gamma(\psi)}(s)\bigr)-\phi(\overline t,\overline x)\geq\exp\bigl(-\lambda(s-t)\bigr)V^{-}\bigl(s,y_{\overline t,\overline x}^{\psi,\beta^\gamma(\psi)}(s)\bigr)-V^{-}(\overline t,\overline x).$$
By plugging this into the inequality (\ref{equation_3.2}) we obtain, for $s=t^\prime$ and $t^\prime-t$ small enough,
\begin{equation*}
\begin{aligned}
\inf_{\beta\in\mathcal{B}}\sup_{\psi\in\Psi}\biggl\{&\int_t^{t^\prime} f\bigl(r,y_{\overline t,\overline x}^{\psi,\beta(\psi)}(r);\theta(r)\bigr)\exp\bigl(-\lambda(r-t)\bigr)dr\\
&+V^-\bigl(t^\prime,y_{\overline t,\overline x}^{\psi,\beta(\psi)}(t^\prime)\bigr)\exp\bigl(-\lambda(t^\prime-t)\bigr)\biggr\}-V^{-}(\overline t,\overline x)\leq-\frac{\gamma}{4}(t^\prime-t)<0,
\end{aligned}
\end{equation*}
which, without loss of generality when $t=\overline{t}$ and $t^\prime-t<\tau_0\wedge\rho_0$ and taking into account the DPP (\ref{equation_2.1}), yields a contradiction. Hence the lower value $V^{-}$ is a viscosity sub-solution to the HJBI equation.\\
\par Next, we show the super-solution property. Let $\phi$ be a function in $C^{1,1}\bigl([t,T)\times\mathbb{R}^n\bigr)$ and $(\underline{t},\underline{x})\in[t,T)\times\mathbb{R}^n$ be such that $V^{-}-\phi$ achieves a local minimum at $(\underline{t},\underline{x})$ in $I\times B_{\delta}(\underline x)$, where $B_{\delta}(\underline x)$ is the open ball of radius $\delta>0$ centered at $\underline x$ and $I:=[\underline t-\delta,\underline t+\delta]$, and $V^-(\underline t,\underline x)=\phi(\underline t,\underline x)$. Now, for $\varepsilon>0$ and without loss of generality, we assume that $V^-(\underline t,\underline x)-\mathcal H_{inf}^\chi V^-(\underline t,\underline x)<\varepsilon<0$ on $I\times B_\delta(\underline x)$, otherwise, i.e., $V^-(\underline t,\underline x)=\mathcal H_{inf}^\chi V^-(\underline t,\underline x)$, the proof is finished. Therefore Remark \ref{Remark3.1} leads us to $V^-(\underline t,\underline x)\geq\mathcal H_{sup}^c V^-(\underline t,\underline x)$. We define
$$s'=\inf\bigl\{s\geq t:s\notin I\;\text{and}\;y_{\underline{t},\underline{x}}(s)\notin B_\delta(\underline x)\bigr\},$$
then we let $t\leq s\leq s'$, and we proceed by contradiction. Assume that
$$-\frac{\partial\phi}{\partial s}(\underline t,\underline x)+\lambda\phi(\underline t,\underline x)+H\bigl(\underline t,\underline x,D_y\phi(\underline t,\underline x)\bigr)=-\gamma<0.$$
By the definition of the first-order Hamiltonian ($H$), one can find an element $\theta$ of $\mathbb{R}^l$ such that
$$-\frac{\partial\phi}{\partial s}(\underline t,\underline x)+\lambda \phi(\underline t,\underline x)-D_y\phi(\underline t,\underline x).b(\underline t,\underline x;\theta)-f(\underline t,\underline x;\theta)\leq-\gamma.$$
Thus, there exists a non-anticipative strategy $\alpha^\gamma\in\mathcal{A}$ for maximizing player$-\xi$ such that for any $v\in\mathcal{V}$, $\alpha^\gamma(v)=\psi^\gamma:=\bigl(\theta(.),u\bigr)$, and for $s-t$ small enough and any $\beta\in\mathcal{B}$, we have that
\begin{equation*}
\begin{aligned}
-\frac{\partial}{\partial s}\phi\bigl(s,y_{\underline t,\underline x}^{\psi^\gamma,\beta(\psi^\gamma)}(s)\bigr)&+\lambda\phi\bigl(s,y_{\underline t,\underline x}^{\psi^\gamma,\beta(\psi^\gamma)}(s)\bigr)-D_y\phi\bigl(s,y_{\underline t,\underline x}^{\psi^\gamma,\beta(\psi^\gamma)}(s)\bigr).b\bigl(s,y_{\underline t,\underline x}^{\psi^\gamma,\beta(\psi^\gamma)}(s);\theta(s)\bigr)\\
&-f\bigl(s,y_{\underline t,\underline x}^{\psi^\gamma,\beta(\psi^\gamma)}(s);\theta(s)\bigr)\leq-\frac{\gamma}{2}.
\end{aligned}
\end{equation*}
Now we multiply both sides of the last inequality by $\exp\bigl(-\lambda (s-t)\bigr)$ and integrate from $t$ to $t^\prime$ to obtain
\begin{equation}\label{equation_3.3}
\phi(\underline t,\underline x)-\exp\bigl(-\lambda(t^\prime-t)\bigr)\phi\bigl(t^\prime,y_{\underline t,\underline x}^{\psi^\gamma,\beta(\psi^\gamma)}(t^\prime)\bigr)-\int_{t}^{t^\prime}f\bigl(s,y_{\underline t,\underline x}^{\psi^\gamma,\beta(\psi^\gamma)}(s);\theta(s)\bigr)\exp\bigl(-\lambda(s-t)\bigr)ds\leq-\frac{\gamma}{4}(t^\prime-t).
\end{equation}
Since $V^{-}-\phi$ has a local minimum at $(\underline t,\underline x)$ and $V^{-}(\underline t,\underline x)=\phi(\underline t,\underline x)$ we have, for $t=\underline{t}$ and $s-t$ small enough, that
$$\bigl\|y_{\underline t,\underline x}^{\psi^\gamma,\beta(\psi^\gamma)}(s)-\underline x\bigr\|\rightarrow 0,$$
which gives
$$\exp\bigl(-\lambda(s-t)\bigr)\phi\bigl(s,y_{\underline t,\underline x}^{\psi^\gamma,\beta(\psi^\gamma)}(s)\bigr)-\phi(\underline t,\underline x)\leq\exp\bigl(-\lambda (s-t)\bigr)V^{-}\bigl(y_{\underline t,\underline x}^{\psi^\gamma,\beta(\psi^\gamma)}(s)\bigr)-V^{-}(\underline t,\underline x),$$
thus
$$\exp\bigl(-\lambda (t^\prime-t)\bigr)V^{-}\bigl(t^\prime,y_{\underline t,\underline x}^{\psi^\gamma,\beta(\psi^\gamma)}(t^\prime)\bigr)+\int_{t}^{t^\prime}f\bigl(s,y_{\underline t,\underline x}^{\psi^\gamma,\beta(\psi^\gamma)}(s);\theta(s)\bigr)\exp\bigl(-\lambda (s-t)\bigr)ds\geq\frac{\gamma}{2}(t^\prime-t)+V^{-}(\underline t,\underline x).$$
By plugging this into (\ref{equation_3.3}), for $t^\prime-t$ small enough, we obtain
\begin{equation*}
\begin{aligned}
\inf_{\beta\in\mathcal{B}}\sup_{\psi\in\Psi}\biggl\{&\int_t^{t^\prime}f\bigl(s,y_{\underline t,\underline x}^{\psi,\beta(\psi)}(s);\theta(s)\bigr)\exp\bigl(-\lambda(s-t)\bigr)ds\\
&+V^-\bigl(t^\prime,y_{\underline t,\underline x}^{\psi,\beta(\psi)}(t^\prime)\bigr)\exp\bigl(-\lambda (t^\prime-t)\bigr)\biggr\}-V^{-}(\underline t,\underline x)>0,
\end{aligned}
\end{equation*}
which, without loss of generality when $t=\underline{t}$ and $t^\prime-t<\tau_0\wedge\rho_0$ and taking into account the DPP (\ref{equation_2.1}), yields a contradiction, then the lower value $V^{-}$ is a viscosity super-solution to the HJBI equation. Hence deducing the thesis.
\end{proof}
\subsection{Uniqueness of Viscosity Solutions for the HJBI Equation}\label{Sect.3.2}
This section proves the \textit{comparison principle} of viscosity solutions to the HJBI equation, and shows that this equation has a unique BUC VS. As a consequence, the value functions coincide, since they are viscosity solutions to the HJBI equation. Thus the zero-sum DG considered has a value. We start by proving the useful Proposition \ref{Proposition3.1}, which is inspired from \cite{El17,El17_2}.
\begin{proposition}\label{Proposition3.1}
The Hamilton-Jacobi-Bellman-Isaacs equation (HJBI) is equivalent to the following equation:
\begin{equation*}
\left\{
\begin{aligned}
&
\begin{aligned}
\lambda v(s,y)=\min_{i\in\{0,1\}}\biggl\{(1-i)\max_{j\in\{0,1\}}\Bigl[&(1-j)\Bigl(\frac{\partial}{\partial s}v(s,y)+\sup_{\theta\in\mathbb{R}^l}\bigl\{D_yv(s,y).b(s,y;\theta)+f(s,y;\theta)\bigr\}\Bigr)\\
&+j\lambda\mathcal{H}_{sup}^c v(s,y)\Bigr]+i\lambda\mathcal{H}_{inf}^\chi v(s,y)\biggr\},\;\text{on}\;[t,T)\times\mathbb{R}^n.
\end{aligned}\\
&v(T,y)=G(y)\;\text{for all}\;y\in\mathbb{R}^n.
\end{aligned}
\right.
\end{equation*}
\end{proposition}
\begin{proof}
For any positive numbers $a,\;b,\;a^\prime$ and $b^\prime$, solving an equation of the form $\max\bigl\{\min[A,B];C\bigr\}=0$ is equivalent to solve the equation
\begin{equation}\label{equivequ}
\max_{i\in\{0,1\}}\Bigl\{(1-i)a\min_{j\in\{0,1\}}\bigl[(1-j)a^\prime A+jb^\prime B\bigr]+ibC\Bigr\}=0,
\end{equation}
the same for the inequalities
$$\max\bigl\{\min[A,B];C\bigr\}\leq 0,\;\text{and}\;\max\bigl\{\min[A,B];C\bigr\}\geq 0.$$
We use (\ref{equivequ}), for $a=a^\prime=1$ and $b=b^\prime=\lambda$, to rewrite the HJBI equation as follows
\begin{equation*}
\begin{aligned}
\max_{i\in\{0,1\}}\biggl\{(1-i)\min_{j\in\{0,1\}}\Bigl[&(1-j)\inf_{\theta\in\mathbb{R}^l}\bigl\{-\frac{\partial}{\partial s}v(s,y)+\lambda v(s,y)-D_yv(s,y).b(s,y;\theta)-f(s,y;\theta)\bigr\}\\
&+j\lambda\bigl(v(s,y)-\mathcal{H}_{sup}^c v(s,y)\bigr)\Bigr]+i\lambda\bigl(v(s,y)-\mathcal{H}_{inf}^\chi v(s,y)\bigr)\biggr\}=0,
\end{aligned}
\end{equation*}
where $v$ being a continuous function in $[t,T]\times\mathbb{R}^n$. We then get
\begin{equation*}
\begin{aligned}
\max_{i\in\{0,1\}}\biggl\{(1-i)\min_{j\in\{0,1\}}\Bigl[&\lambda v(s,y)-j\lambda v(s,y)+(1-j)\inf_{\theta\in\mathbb{R}^l}\bigl\{-\frac{\partial}{\partial s}v(s,y)-D_yv(s,y).b(s,y;\theta)-f(s,y;\theta)\bigr\}\\
&+j\lambda\bigl(v(s,y)-\mathcal{H}_{sup}^c v(s,y)\bigr)\Bigr]+i\lambda\bigl(v(s,y)-\mathcal{H}_{inf}^\chi v(s,y)\bigr)\biggr\}=0,
\end{aligned}
\end{equation*}
thus
\begin{equation*}
\begin{aligned}
\max_{i\in\{0,1\}}\biggl\{(1-i)\min_{j\in\{0,1\}}\Bigl[&\lambda v(s,y)-(1-j)\sup_{\theta\in\mathbb{R}^l}\bigl\{\frac{\partial}{\partial s}v(s,y)+D_yv(s,y).b(s,y;\theta)+f(s,y;\theta)\bigr\}\\
&-j\lambda\mathcal{H}_{sup}^c v(s,y)\Bigr]+i\lambda\bigl(v(s,y)-\mathcal{H}_{inf}^\chi v(s,y)\bigr)\biggr\}=0.
\end{aligned}
\end{equation*}
Then it follows
\begin{equation*}
\begin{aligned}
\min_{i\in\{0,1\}}\biggl\{(1-i)\max_{j\in\{0,1\}}\Bigl[&-\lambda v(s,y)+(1-j)\sup_{\theta\in\mathbb{R}^l}\bigl\{\frac{\partial}{\partial s}v(s,y)+D_yv(s,y).b(s,y;\theta)+f(s,y;\theta)\bigr\}\\
&+j\lambda\mathcal{H}_{sup}^c v(s,y)\Bigr]-i\lambda\bigl(v(s,y)-\mathcal{H}_{inf}^\chi v(s,y)\bigr)\biggr\}=0,
\end{aligned}
\end{equation*}
from which we deduce
\begin{equation*}
\begin{aligned}
\min_{i\in\{0,1\}}\biggl\{-\lambda v(s,y)+(1-i)\max_{j\in\{0,1\}}\Bigl[&(1-j)\sup_{\theta\in\mathbb{R}^l}\bigl\{\frac{\partial}{\partial s}v(s,y)+D_yv(s,y).b(s,y;\theta)+f(s,y;\theta)\bigr\}\\
&+j\lambda\mathcal{H}_{sup}^c v(s,y)\Bigr]+i\lambda\mathcal{H}_{inf}^\chi v(s,y)\biggr\}=0.
\end{aligned}
\end{equation*}
Finally we deduce the desired expression for the HJBI equation
\begin{equation*}
\begin{aligned}
\lambda v(s,y)=\min_{i\in\{0,1\}}\biggl\{(1-i)\max_{j\in\{0,1\}}\Bigl[&(1-j)\Bigl(\frac{\partial}{\partial s}v(s,y)+\sup_{\theta\in\mathbb{R}^l}\bigl\{D_yv(s,y).b(s,y;\theta)+f(s,y;\theta)\bigr\}\Bigr)\\
&+j\lambda\mathcal{H}_{sup}^c v(s,y)\Bigr]+i\lambda\mathcal{H}_{inf}^\chi v(s,y)\biggr\},
\end{aligned}
\end{equation*}
which completes the proof.
\end{proof}
\begin{remark}
By using Proposition \ref{Proposition3.1}, the Definition \ref{definition3.1} of the viscosity solution of the Hamilton-Jacobi-Bellman-Isaacs equation (HJBI) could be rewritten as the following:\\
A continuous function $v$ in $[t,T]\times\mathbb{R}^n$ which satisfies $v(T,y)=G(y)$ for any $y\in\mathbb{R}^n$, is a viscosity sub-solution (resp. super-solution) of the equation if and only if for any function $\phi\in C^{1,1}\bigl([t,T)\times\mathbb{R}^n\bigr)$ and $(\overline s,\overline y)$ $\bigl(\text{resp.}\;(\underline s,\underline y)\bigr)\in [t,T)\times\mathbb{R}^n$ a local maximum (resp. minimum) point of $v-\phi$ such that $v(\overline s,\overline y)=\phi(\overline s,\overline y)$ $\bigl(\text{resp.}\;v(\underline s,\underline y)=\phi(\underline s,\underline y)\bigr)$, we have
\begin{equation*}
\begin{aligned}
\lambda v(\overline s,\overline y)\leq \min_{i\in\{0,1\}}\biggl\{(1-i)\max_{j\in\{0,1\}}\Bigl[&(1-j)\Bigl(\frac{\partial\phi}{\partial s}(\overline s,\overline y)+\sup_{\theta\in\mathbb{R}^l}\Bigl\{D_y\phi(\overline s,\overline y).b(\overline s,\overline y;\theta)+f(\overline s,\overline y;\theta)\Bigr\}\Bigr)\\
&+j\lambda\mathcal{H}_{sup}^c v(\overline s,\overline y)\Bigr]+i\lambda\mathcal{H}_{inf}^\chi v(\overline s,\overline y)\biggr\}
\end{aligned}
\end{equation*}
\begin{equation*}
\begin{aligned}
\biggl(\text{resp.}\;\lambda v(\underline s,\underline y)\geq \min_{i\in\{0,1\}}\biggl\{(1-i)\max_{j\in\{0,1\}}\Bigl[&(1-j)\Bigl(\frac{\partial\phi}{\partial s}(\underline s,\underline y)+\sup_{\theta\in\mathbb{R}^l}\Bigl\{D_y\phi(\underline s,\underline y).b(\underline s,\underline y;\theta)+f(\underline s,\underline y;\theta)\Bigr\}\Bigr)\\
&+j\lambda\mathcal{H}_{sup}^c v(\underline s,\underline y)\Bigr]+i\lambda\mathcal{H}_{inf}^\chi v(\underline s,\underline y)\biggr\}\biggr).
\end{aligned}
\end{equation*}\qed
\end{remark}
Lemma \ref{Lemma3.3} below, for which the proof is obvious, will be useful later for deducing the thesis.
\begin{lemma}\label{Lemma3.3}
If a continuous function $v$ is a viscosity solution to the Hamilton-Jacobi-Bellman-Isaacs equation (HJBI) such that $v(T,y)=G(y)$ for any $y\in\mathbb{R}^n$, then for any $0<\mu<1$ the function $\mu v$ is a viscosity solution to the following equation (HJBI$_\mu$):
\begin{equation*}
\text{(HJBI$_\mu$)}\;\left\{
\begin{aligned}
&
\begin{aligned}
\max\biggl\{\min\Bigl[&-\frac{\partial}{\partial s}v(s,y)+\lambda v(s,y)+H_\mu\bigl(s,y,D_{y}v(s,y)\bigr),v(s,y)-\mathcal{H}_{sup}^{c,\mu}v(s,y)\Bigr];\\
&v(s,y)-\mathcal{H}_{inf}^{\chi,\mu}v(s,y)\biggr\}=0,\;\text{on}\;[t,T)\times\mathbb{R}^n;
\end{aligned}\\
&v(T,y)=\mu G(y)\;\text{for all}\;y\in\mathbb{R}^n,
\end{aligned}
\right.
\end{equation*}
where
$$\mathcal{H}_{sup}^{c,\mu}v(s,y):=\sup_{\xi\in U}\Bigl\{v\bigl(s,y+g_\xi(s,y;\xi)\bigr)-\mu c(s,y;\xi)\Bigr\},$$
$$\mathcal{H}_{inf}^{\chi,\mu} v(s,y):=\inf_{\eta\in V}\Bigl\{v\bigl(s,y+g_\eta(s,y;\eta)\bigr)+\mu\chi(s,y;\eta)\Bigr\},$$
and
$$H_\mu\bigl(s,y,D_{y}v(s,y)\bigr):=\inf_{\theta\in\mathbb{R}^l}\Bigl\{-D_{y}v(s,y).b(s,y;\theta)-\mu f(s,y;\theta)\Bigr\}.$$ \qed
\end{lemma}
Now we are in a position to give the proof of the comparison theorem stated as follows:
\begin{theorem}[Comparison Theorem]\label{CThe}
Assume $\textbf{H}_b$, $\textbf{H}_g$, $\textbf{H}_f$, $\textbf{H}_{c,\chi}$ and $\textbf{H}_{G}$. If $u$ and $v$ are, respectively, a bounded uniformly continuous viscosity sub-solution and super-solution to the Hamilton-Jacobi-Bellman-Isaacs equation (HJBI), satisfying $u(T,.)\leq v(T,.)$, then we have $$\forall\;(t,x)\in[0,T]\times\mathbb{R}^n:\;u(t,x)\leq v(t,x).$$
\end{theorem}
\begin{proof}
The proof is inspired from \cite{BBEl10, El17_2}. Let $u$ and $v$ be, respectively, a BUC viscosity sub-solution and super-solution to the HJBI equation. Recalling, for all $0<\mu<1$, Proposition \ref{Proposition3.1} and Lemma \ref{Lemma3.3}, to get that $\mu u$ is a viscosity sub-solution to the following equation:
\begin{equation}\label{equation_3.5}
\left\{
\begin{aligned}
&
\begin{aligned}
\lambda u(s,y)=\min_{i\in\{0,1\}}\biggl\{(1-i)\max_{j\in\{0,1\}}\Bigl[&(1-j)\Bigl(\frac{\partial}{\partial s}u(s,y)+\sup_{\theta\in\mathbb{R}^l}\Bigl\{D_yu(s,y).b(s,y;\theta)+\mu f(s,y;\theta)\Bigr\}\Bigr)\\
&+j\lambda\mathcal{H}_{sup}^{c,\mu} u(s,y)\Bigr]+i\lambda\mathcal{H}_{inf}^{\chi,\mu} u(s,y)\biggr\},\;\text{on}\;[t,T)\times\mathbb{R}^n;
\end{aligned}\\
&u(T,y)=\mu G(y)\;\text{for all}\;y\in\mathbb{R}^n.
\end{aligned}
\right.
\end{equation}
where the operators $\mathcal{H}_{inf}^{\chi,\mu}$ and $\mathcal{H}_{sup}^{c,\mu}$ are defined as in Lemma \ref{Lemma3.3} and the function $u$ is from $[t,T]\times\mathbb{R}^n$ into $\mathbb{R}$. First, we assume that $M=\sup_{(t,x)\in[0,T)\times\mathbb{R}^n}\bigl(u(t,x)-v(t,x)\bigr)>0$, if it is not the case, i.e., $M\leq 0$, the proof is then finished. If $\|u\|_{\infty}=0$ we have $$M_\mu=\sup_{(t,x)\in[0,T)\times\mathbb{R}^n}\bigl(\mu u(t,x)-v(t,x)\bigr)>0,$$
otherwise, by letting $1-M/(2\|u\|_{\infty})\leq\mu<1$ we also get that $M_\mu>0$. Next, we divide the proof into the following three steps:\\
\par \textbf{Step 1.} Let $\varepsilon>0$, $\beta>0$ and consider for all $t\in[0,T)$ and $x,y\in\mathbb{R}^n$ the following test function:
$$\Gamma_{\mu,\varepsilon,\beta}(t,x,y)=\mu u(t,x)-v(t,y)-\frac{\|x-y\|^2}{\varepsilon^2}-\beta\bigl(\|x\|^2+\|y\|^2\bigr).$$
Since $\Gamma_{\mu,\varepsilon,\beta}$ is a continuous function going to infinity when $x$ or $y$ does, then it admits a maximum point $(t_m,x_m,y_m)$ satisfying $M_{\Gamma_{\mu,\varepsilon,\beta}}=\Gamma_{\mu,\varepsilon,\beta}(t_m,x_m,y_m).$
We have for all $t\in[0,T)$ and $x,y\in\mathbb{R}^n$,
\begin{equation}\label{equation_3.6}
\mu u(t_m,x_m)-v(t_m,y_m)-\frac{\|x_m-y_m\|^2}{\varepsilon^2}-\beta\bigl(\|x_m\|^2+\|y_m\|^2\bigr)\geq \mu u(t,x)-v(t,y)-\frac{\|x-y\|^2}{\varepsilon^2}-\beta\bigl(\|x\|^2+\|y\|^2\bigr).
\end{equation}
\begin{itemize}
\item Firstly, using inequality (\ref{equation_3.6}) with $(t,y)=(t_m,y_m)$, we get that $(t_m,x_m)$ is a maximal point of $\mu u(t,x)-\phi_u(t,x)$, where $$\phi_u(t,x)=\frac{\|x-y_m\|^2}{\varepsilon^2}+\beta\|x\|^2,$$
then, since $\mu u$ is viscosity sub-solution of (\ref{equation_3.5}), we get
\begin{equation}\label{equation_3.7}
\begin{aligned}
\lambda\mu u(t_m,x_m)\leq\min_{i\in\{0,1\}}\biggl\{(1-i)&\max_{j\in\{0,1\}}\biggl[(1-j)\sup_{\theta\in\mathbb{R}^l}\Bigl\{\Bigl\langle\frac{2\|x_m-y_m\|}{\varepsilon^2}+2\beta x_m,b(t_m,x_m;\theta)\Bigr\rangle\\
&+\mu f(t_m,x_m;\theta)\Bigr\}+j\lambda\mathcal{H}_{sup}^{c,\mu}\mu u(t_m,x_m)\biggr]+i\lambda\mathcal{H}_{inf}^{\chi,\mu}\mu u(t_m,x_m)\biggr\}.
\end{aligned}
\end{equation}
\item Secondly, using inequality (\ref{equation_3.6}) with $(t,x)=(t_m,x_m)$, we get that $(t_m,y_m)$ is a minimal point of $v(t,y)-\phi_v(t,y)$, where $$\phi_v(t,y)=-\frac{\|x_m-y\|^2}{\varepsilon^2}-\beta\|y\|^2,$$
then, since $v$ is viscosity super-solution of the HJBI equation, by applying Proposition \ref{Proposition3.1} we get
\begin{equation}\label{equation_3.8}
\begin{aligned}
\lambda v(t_m,y_m)\geq\min_{i\in\{0,1\}}\biggl\{(1-i)\max_{j\in\{0,1\}}\biggl[&(1-j)\sup_{\theta\in\mathbb{R}^l}\Bigl\{\Bigl\langle \frac{2\|x_m-y_m\|}{\varepsilon^2}-2\beta y_m,b(t_m,y_m;\theta)\Bigr\rangle\\
&+f(t_m,y_m;\theta)\Bigr\}+j\lambda\mathcal{H}_{sup}^{c} v(t_m,y_m)\biggr]+i\lambda\mathcal{H}_{inf}^{\chi} v(t_m,y_m)\biggr\}.
\end{aligned}
\end{equation}
\end{itemize}
Hence, using above inequalities (\ref{equation_3.7}) and (\ref{equation_3.8}), we get
\begin{equation*}
\begin{aligned}
\lambda\bigl(\mu u(t_m,x_m)-&v(t_m,y_m)\bigr)\leq \min_{i\in\{0,1\}}\biggl\{(1-i)\max_{j\in\{0,1\}}\biggl[(1-j)\sup_{\theta\in\mathbb{R}^l}\Bigl\{\Bigl\langle \frac{2\|x_m-y_m\|}{\varepsilon^2}+2\beta x_m,\\
&b(t_m,x_m;\theta)\Bigr\rangle+\mu f(t_m,x_m;\theta)\Bigr\}+j\lambda\mathcal{H}_{sup}^{c,\mu}\mu u(t_m,x_m)\biggr]+i\lambda\mathcal{H}_{inf}^{\chi,\mu}\mu u(t_m,x_m)\biggr\}\\
&+\max_{i\in\{0,1\}}\biggl\{(1-i)\min_{j\in\{0,1\}}\biggl[(1-j)\inf_{\theta\in\mathbb{R}^l}\Bigl\{-\Bigl\langle \frac{2\|x_m-y_m\|}{\varepsilon^2}-2\beta y_m,b(t_m,y_m;\theta)\Bigr\rangle\\
&-f(t_m,y_m;\theta)\Bigr\}-j\lambda\mathcal{H}_{sup}^{c} v(t_m,y_m)\biggr]-i\lambda\mathcal{H}_{inf}^{\chi} v(t_m,y_m)\biggr\},
\end{aligned}
\end{equation*}
then we get
\begin{equation*}
\begin{aligned}
\lambda\bigl(\mu u(t_m,x_m)-&v(t_m,y_m)\bigr)\leq \max_{i\in\{0,1\}}\biggl\{(1-i)\min_{j\in\{0,1\}}\biggl[(1-j)\inf_{\theta\in\mathbb{R}^l}\Bigl\{\Bigl\langle \frac{2\|x_m-y_m\|}{\varepsilon^2},\\
&b(t_m,x_m;\theta)-b(t_m,y_m;\theta)\Bigr\rangle+2\beta\bigl\langle x_m,b(t_m,x_m;\theta)\bigr\rangle+2\beta\bigl\langle y_m,b(t_m,y_m;\theta)\bigr\rangle\\
&+\mu f(t_m,x_m;\theta)-f(t_m,y_m;\theta)\Bigr\}+j\lambda\bigl(\mathcal{H}_{sup}^{c,\mu}\mu u(t_m,x_m)-\mathcal{H}_{sup}^{c} v(t_m,y_m)\bigr)\biggr]\\
&+i\lambda\bigl(\mathcal{H}_{inf}^{\chi,\mu}\mu u(t_m,x_m)-\mathcal{H}_{inf}^{\chi} v(t_m,y_m)\bigr)\biggr\}.
\end{aligned}
\end{equation*}
Thus from standing assumptions
\begin{equation}\label{equation_3.9}
\begin{aligned}
\lambda\bigl(\mu u(t_m,x_m)-v(t_m,y_m)\bigr)\leq& \max\biggl\{\min\biggl[2C_b\frac{\|x_m-y_m\|^2}{\varepsilon^2}+2\beta\|b\|_\infty\Bigl(\|x_m\|+\|y_m\|\Bigr)+(1-\mu)\|f\|_\infty,\\
&\lambda\Bigl(\mathcal{H}_{sup}^{c,\mu} \mu u(t_m,x_m)-\mathcal{H}_{sup}^{c,\mu} \mu u(t_m,y_m)+\bigl\|\bigl(\mathcal{H}_{sup}^{c,\mu}\mu u-\mathcal{H}_{sup}^{c}v\bigr)^+\bigr\|_\infty\Bigr)\biggr];\\
&\lambda\Bigl(\mathcal{H}_{inf}^{\chi,\mu}\mu u(t_m,x_m)-\mathcal{H}_{inf}^{\chi,\mu}\mu u(t_m,y_m)+\bigl\|\bigl(\mathcal{H}_{inf}^{\chi,\mu}\mu u-\mathcal{H}_{inf}^{\chi}v\bigr)^+\bigr\|_\infty\Bigr)\biggr\}.
\end{aligned}
\end{equation}
The following two steps investigate the right-hand side of inequality (\ref{equation_3.9}):\\
\par\textbf{Step 2.} We prove hereafter that
\begin{equation}\label{equation_3.10}
\forall\eta>0,\;\exists\varepsilon_0>0,\;\beta_0>0,\;\forall\varepsilon\leq\varepsilon_0,\;\beta\leq\beta_0:\;\frac{\|x_m-y_m\|^2}{\varepsilon^2}+\beta\bigl(\|x_m\|^2+\|y_m\|^2\bigr)\leq\eta.
\end{equation}
We use inequality (\ref{equation_3.6}) for $x=y$ in the right-hand side, then we get $M_{\Gamma_{\mu,\varepsilon,\beta}}\geq \mu u(t,x)-v(t,x)-2\beta\|x\|^2$.
Further, we let $\sup_{(t,x)\in[0,T)\times\mathbb{R}^n}\bigl(\mu u(t,x)-v(t,x)\bigr)$ be reached in a point $(t^*,x^*)$, within $\delta>0$ arbitrary small, thus $\mu u(t^*,x^*)-v(t^*,x^*)\geq M_\mu-\delta$. Now we choose $\delta$ and $\beta$ such that $M_\mu-\delta-2\beta\|x^*\|^2>0$, which is possible since $(t^*,x^*)$ depends only on $\delta$. Thus we deduce
\begin{equation}\label{equation_3.11}
\begin{aligned}
M_{\Gamma_{\mu,\varepsilon,\beta}}&\geq \mu u(t^*,x^*)-v(t^*,x^*)-2\beta\|x^*\|^2\\
&\geq M_\mu-\delta-2\beta\|x^*\|^2\\
&>0.
\end{aligned}
\end{equation}
By letting $r^2=\mu\|u\|_\infty+\|v\|_\infty$, we get
$$\|u\|_\infty\leq M_{\Gamma_{\mu,\varepsilon,\beta}}\leq r^2-\frac{\|x_m-y_m\|^2}{\varepsilon^2}-\beta\bigl(\|x_m\|^2+\|y_m\|^2\bigr),$$
then
\begin{equation}\label{equation_3.12}
\|x_m-y_m\|\leq r\varepsilon.
\end{equation}
Therefore, we introduce the following increasing function:
$$m(w)=\sup_{t\in[0,T),\|x-y\|\leq w}\bigl|v(t,x)-v(t,y)\bigr|,$$
then, combining with (\ref{equation_3.12}), we obtain
$$\mu u(t_m,x_m)-v(t_m,y_m)=\mu u(t_m,x_m)-v(t_m,x_m)+v(t_m,x_m)-v(t_m,y_m)\leq M_\mu+m(r\varepsilon).$$
Thus, from (\ref{equation_3.11}) using the definition of $M_{\Gamma_{\mu,\varepsilon,\beta}}$, we get
$$M_\mu-\delta-2\beta\|x^*\|^2\leq M_{\Gamma_{\mu,\varepsilon,\beta}}\leq M_\mu+m(r\varepsilon)-\frac{\|x_m-y_m\|^2}{\varepsilon^2}-\beta\bigl(\|x_m\|^2+\|y_m\|^2\bigr),$$
then $$\frac{\|x_m-y_m\|^2}{\varepsilon^2}+\beta\bigl(\|x_m\|^2+\|y_m\|^2\bigr)\leq\delta+2\beta\|x^*\|^2+m(r\varepsilon).$$
Now, we choose $\eta<4M_\mu/3$ and we take $\delta=\eta/4$ and $\beta_0=1$ if $\|x^*\|=0$, $\beta_0=\varepsilon/(4\|x^*\|^2)$ if $\|x^*\|\neq 0$, to get the desired inequality (\ref{equation_3.10}). The proof is then complete. We also get for any $\beta\leq\beta_0$,
\begin{equation}\label{equation_3.13}
0<M_\mu-\frac{3\eta}{4}\leq M_\mu-\delta-2\beta\|x^*\|^2\leq M_{\Gamma_{\mu,\varepsilon,\beta}}\leq \mu u(t_m,x_m)-v(t_m,y_m).
\end{equation}
\par\textbf{Step 3.} To complete the proof it remains to show contradiction. By (\ref{equation_3.10}), for $\varepsilon\leq\varepsilon_0$ and $\beta\leq\beta_0$ we have
$$2C_b\|x_m-y_m\|^2/\varepsilon^2\leq 2C_b\eta,\;\beta\|x_m\|\leq\sqrt{\beta\eta},\;\text{and}\;\beta\|y_m\|\leq\sqrt{\beta\eta}.$$
Then, for all $\beta\leq\beta_1=\min\bigl\{\beta_0,\eta/\|b\|^2_\infty\bigr\}$, we get $2\beta\|b\|_\infty\bigl(\|x_m\|+\|y_m\|\bigr)\leq 4\eta$. Moreover, for all $\varepsilon\leq\varepsilon_1=\min\bigl\{\varepsilon_0,\sqrt{\eta}/C_f\bigr\}$, we have $C_f\bigl(\|x_m-y_m\|\bigr)\leq\eta$. By Proposition \ref{Proposition2.5}, the two functions $x\rightarrow\mathcal{H}_{inf}^{\chi,\mu}\mu u(t,x)$ and $x\rightarrow\mathcal{H}_{sup}^{c,\mu}\mu u(t,x)$ are UC for any $t\in[0,T)$, then, tacking into account (\ref{equation_3.12}), we can find $\varepsilon_2\leq\varepsilon_1$ such that for $\varepsilon\leq\varepsilon_2$, $$\mathcal{H}_{inf}^{\chi,\mu} \mu u(t_m,x_m)-\mathcal{H}_{inf}^{\chi,\mu} \mu u(t_m,y_m)\leq\eta,\;\text{and}\;\mathcal{H}_{sup}^{c,\mu} \mu u(t_m,x_m)-\mathcal{H}_{sup}^{c,\mu} \mu u(t_m,y_m)\leq\eta.$$
Thus, from (\ref{equation_3.9}) for all $\varepsilon\leq\varepsilon_2$ and $\beta\leq\beta_1$, we get
\begin{equation*}
\begin{aligned}
\lambda\bigl(\mu u(t_m,x_m)-v(t_m,y_m)\bigr)\leq\max\biggl\{\min\Bigl[& (1-\mu)\|f\|_\infty,\lambda\bigl\|\bigl(\mathcal{H}_{sup}^{c,\mu}\mu u-\mathcal{H}_{sup}^{c}v\bigr)^+\bigr\|_\infty\Bigr];\\
&\lambda\bigl\|\bigl(\mathcal{H}_{inf}^{\chi,\mu}\mu u-\mathcal{H}_{inf}^{\chi}v\bigr)^+\bigr\|_\infty\biggr\}+(5+2C_b+\lambda)\eta,
\end{aligned}
\end{equation*}
from (\ref{equation_3.13}) and the fact that $\eta$ is arbitrary we deduce
\begin{equation*}
\begin{aligned}
\lambda\bigl\|(\mu u-v)^+\bigr\|_\infty\leq\max\biggl\{\min\Bigl[& (1-\mu)\|f\|_\infty,\lambda\bigl\|\bigl(\mathcal{H}_{sup}^{c,\mu}\mu u-\mathcal{H}_{sup}^{c}v\bigr)^+\bigr\|_\infty\Bigr];\\
&\lambda\bigl\|\bigl(\mathcal{H}_{inf}^{\chi,\mu}\mu u-\mathcal{H}_{inf}^{\chi}v\bigr)^+\bigr\|_\infty\biggr\},
\end{aligned}
\end{equation*}
thus
\begin{equation}\label{equation_3.14}
\lambda\bigl\|(\mu u-v)^+\bigr\|_\infty\leq\max\Bigl[ (1-\mu)\|f\|_\infty,\lambda\bigl\|\bigl(\mathcal{H}_{inf}^{\chi,\mu}\mu u-\mathcal{H}_{inf}^{\chi}v\bigr)^+\bigr\|_\infty\Bigr].
\end{equation}
Since for all $(s,y)\in[t,T)\times\mathbb{R}^n$,
\begin{equation}\label{equation_3.15}
\mathcal{H}_{inf}^{\chi,\mu}\mu u(s,y)-\mathcal{H}_{inf}^{\chi}v(s,y)\leq\sup_{\eta\in V}\Bigl(\mu u\bigl(s,y+g_\eta(s,y;\eta)\bigr)-v\bigl(s,y+g_\eta(s,y;\eta)\bigr)\Bigr)+\sup_{\eta\in V}\bigl((\mu-1)\chi(s,y;\eta)\bigr),
\end{equation}
and, from standing assumptions for all $(s,y)\in[t,T)\times\mathbb{R}^n$ and $\eta\in V\backslash\{0\}$, we have $\chi(s,y;\eta)>0$. We then deduce, from (\ref{equation_3.15}) for $0<\mu<1$, that
\begin{equation}\label{equation_3.16}
\Bigl\|\bigl(\mathcal{H}_{inf}^{\chi,\mu}\mu u-\mathcal{H}_{inf}^{\chi}v\bigr)^+\Bigr\|_\infty<\bigl\|(\mu u-v)^+\bigr\|_\infty.
\end{equation}
Therefore, combining the two inequalities (\ref{equation_3.14}) and (\ref{equation_3.16}) yield that
$$\lambda\|(\mu u-v)^+\|_\infty\leq (1-\mu)\|f\|_\infty.$$
Hence, by letting $\mu\rightarrow 1$ and using the fact that $f$ is bounded, we get $\|(u-v)^+\|_\infty\leq 0$, which leads us to a contradiction and gives the desired comparison result, for any $(t,x)\in[0,T]\times\mathbb{R}^n$, $u(t,x)\leq v(t,x)$.
\end{proof}
\begin{theorem}
Assume $\textbf{H}_b$, $\textbf{H}_g$, $\textbf{H}_f$, $\textbf{H}_{c,\chi}$ and $\textbf{H}_{G}$. The Hamilton-Jacobi-Bellman-Isaacs equation (HJBI) has a unique bounded uniformly continuous viscosity solution.
\end{theorem}
\begin{proof}
Assume that $u$ and $v$ are two viscosity solutions to the HJBI equation. We first use $u$ as a BUC viscosity sub-solution and $v$ as a BUC viscosity super-solution and we recall the comparison principle. Then we change the role of $u$ and $v$ to get $u(t,x)=v(t,x)$ for all $(t,x)\in[0,T]\times\mathbb{R}^n$.
\end{proof}
Let us now give the Corollary \ref{corollary3.1} to summarize the principal results of this section, thus it gives the first contribution of the paper as mentioned in Remark \ref{Remark2.3}.
\begin{corollary}\label{corollary3.1}
Assuming $\textbf{H}_b$, $\textbf{H}_g$, $\textbf{H}_f$, $\textbf{H}_{c,\chi}$ and $\textbf{H}_{G}$, the lower value and the upper value coincide, and the value function of the deterministic finite-time horizon, two-player, zero-sum DG control problem is the unique bounded uniformly continuous viscosity solution to the Hamilton-Jacobi-Bellman-Isaacs equation (HJBI). \qed
\end{corollary}
Next, we focus on the second contribution of the paper as mentioned in Remark \ref{Remark2.3}. Using the fact that the value function is the unique VS to the HJBI equation and studying the approximate equation (HJBI$_{h}$), a family of value functions converging to the value function of each player is introduced. The limit for this family, when the time discretization step $h$ goes to zero, is characterized either as the unique VS to the HJBI equation, or as the limit, when $h$ goes to zero, of the unique solution of the approximate equation (HJBI$_{h}$).
\section{Discrete Approximation of the HJBI Equation}\label{Sect.4}
This section discusses an \textit{approximation scheme} to the solution of the HJBI equation. In the other words, it gives an approximation scheme to the value function of the zero-sum DG control problem studied. We mainly prove that the approximate equation (HJBI$_{h}$) has, for any time discretization step $0<h<1/\lambda$, a unique bounded continuous solution $v_h$ which converges locally uniformly towards the value function when $h$ goes to zero. Such a result will be useful to characterize, by means of a \textit{verification theorem}, a NE strategy for both players. This will be the subject of Section \ref{Sect.5}. These results leads then to some numerical aspects for computing the value function and the related optimal controls of NE and the optimal evolution of the state.
\subsection{Uniqueness of the Approximate Value Function}
We begin by giving the useful Proposition \ref{proposition4.1} below, then we prove that the approximate equation (HJBI$_{h}$) has a unique bounded continuous solution for any time discretization step $0<h<1/\lambda$.
\begin{proposition}\label{proposition4.1}
Solving the approximate Hamilton-Jacobi-Bellman-Isaacs equation (HJBI$_{h}$) is equivalent to solve the following equation:
\begin{equation*}
\left\{
\begin{aligned}
&
\begin{aligned}
v_h(s,y)=\min_{i\in\{0,1\}}\biggl\{&(1-i)\max_{j\in\{0,1\}}\Bigl[(1-j)\sup_{\theta\in\mathbb{R}^l}\Bigl\{(1-\lambda h)v_h\bigl(s+h,y+hb(s,y;\theta)\bigr)+hf(s,y;\theta)\Bigr\}\\
&+j\Phi(h)\mathcal{H}_{sup}^c v_h(s,y)\Bigr]+i\Phi(h)\mathcal{H}_{inf}^\chi v_h(s,y)\biggr\},\;\text{on}\;[t,T)\times\mathbb{R}^n;
\end{aligned}\\
&v_h(T,y)=G(y)\;\text{for all}\;y\in\mathbb{R}^n.
\end{aligned}
\right.
\end{equation*}
\end{proposition}
\begin{proof}
Similar to the proof of Proposition \ref{Proposition3.1}.
\end{proof}
Now we give the proof of the following Theorem \ref{theorem4.1}:
\begin{theorem}\label{theorem4.1}
For any time discretization step $0<h<1/\lambda$, there exists a unique bounded continuous function $v_h$ solution to the approximate Hamilton-Jacobi-Bellman-Isaacs equation (HJBI$_{h}$).
\end{theorem}
\begin{proof}
We first rewrite the approximate equation (HJBI$_{h}$) as a fixed-point problem. Let $(s,y)\in[t,T]\times\mathbb{R}^n$ and $v_h(T,y)=G(y)$, from Proposition \ref{proposition4.1} we get that the approximate equation (HJBI$_{h}$) is equivalent to $Fv_h(s,y)=v_h(s,y)$, where $F$ is a function from the space of bounded continuous (BC) functions on $[t,T]\times\mathbb{R}^n$ into the same space defined as follows:
\begin{equation}\label{F}
\left\{
\begin{aligned}
&
\begin{aligned}
Fv(s,y)=\min_{i\in\{0,1\}}\biggl\{&(1-i)\max_{j\in\{0,1\}}\Bigl[(1-j)\sup_{\theta\in\mathbb{R}^l}\Bigl\{(1-\lambda h)v\bigl(s+h,y+hb(s,y;\theta)\bigr)+hf(s,y;\theta)\Bigr\}\\
&+j\Phi(h)\mathcal{H}_{sup}^c v(s,y)\Bigr]+i\Phi(h)\mathcal{H}_{inf}^\chi v(s,y)\biggr\},\;\text{on}\;[t,T)\times\mathbb{R}^n;
\end{aligned}\\
&Fv(T,y)=G(y)\;\text{for all}\;y\in\mathbb{R}^n.
\end{aligned}
\right.
\end{equation}
Next, we let $v_1$ and $v_2$ be two functions in $BC\bigl([t,T]\times\mathbb{R}^n\bigr)$, then for any $(s,y)\in[t,T)\times\mathbb{R}^n$ we have that
\begin{equation*}
\begin{aligned}
Fv_1(s,y)-&Fv_2(s,y)\leq\max_{i\in\{0,1\}}\biggl\{(1-i)\min_{j\in\{0,1\}}\Bigl[(1-j)\inf_{\theta\in\mathbb{R}^l}\Bigl\{(1-\lambda h)\Bigl(v_1\bigl(s+h,y+hb(s,y;\theta)\bigr)\\
&-v_2\bigl(s+h,y+hb(s,y;\theta)\bigr)\Bigr)\Bigr\}+j\Phi(h)\inf_{\xi\in U}\Bigl(v_1\bigl(s,y+g_\xi(s,y;\xi)\bigr)-v_2\bigl(s,y+g_\xi(s,y;\xi)\bigr)\Bigr)\Bigr]\\
&+i\Phi(h)\sup_{\eta\in V}\Bigl(v_1\bigl(s,y+g_\eta(s,y;\eta)\bigr)-v_2\bigl(s,y+g_\eta(s,y;\eta)\bigr)\Bigr)\biggr\},
\end{aligned}
\end{equation*}
thus
$$Fv_1(s,y)-Fv_2(s,y)\leq\max\bigl\{1-\lambda h,\Phi(h)\bigr\}\|v_1-v_2\|_\infty.$$
We proceed similarly to get
$$\bigl\|Fv_2(s,y)-Fv_1(s,y)\bigr\|_\infty\leq\max\bigl\{1-\lambda h,\Phi(h)\bigr\}\|v_1-v_2\|_\infty.$$
Finally, by the contraction mapping principle for any $0<h<1/\lambda$, there exists a unique BC function $v_h$ solution of the approximate equation (HJBI$_{h}$).
\end{proof}
\begin{remark}
The function $v_h$, unique bounded continuous solution of the approximate Hamilton-Jacobi-Bellman-Isaacs equation (HJBI$_{h}$), will be called the approximate value function. \qed
\end{remark}
\subsection{Convergence of the Approximate Value Function}\label{Sect.4.2}
In this section, we prove the convergence result for the approximate value function $v_h$. We mainly prove that the limit, when $h$ tends to zero, of $v_h$ is a VS to the HJBI equation. We prove first the following Lemma \ref{Lemma4.1}:
\begin{lemma}\label{Lemma4.1}
Let $v_h$ be the approximate value function, the family $\{v_h\}$ is uniformly equicontinuous with respect to state variable and uniformly bounded in $[t,T]\times\mathbb{R}^n$ by $\|f\|_\infty/\lambda$.
\end{lemma}
\begin{proof}
First, let a function $v^0$ be non-negative and BUC with respect to $y\in\mathbb{R}^n$. We have, for any $s\in[t,T)$, that
$$\forall\varepsilon>0,\exists\delta_0>0\;\text{such that}\;\forall y_1,y_2\in\mathbb{R}^n,\;\|y_1-y_2\|<\delta_0,\;\text{implies}\;\bigl|v^0(s,y_1)-v^0(s,y_2)\bigr|<\varepsilon/2.$$
Define now, for $0<h<1/\lambda,\;(s,y)\in[t,T)\times\mathbb{R}^n$ and $F$ defined as in (\ref{F}), a family of functions $v_{h}^{n}(s,y)=F^nv^0(s,y)$ which converges uniformly towards the unique solution $v_h$ to the approximate equation (HJBI$_{h}$). We have, for any $s\in[t,T)$, that
\begin{equation*}
\begin{aligned}
\bigl|v_h^1(s,y_1)&-v_h^1(s,y_2)\bigr|\leq\max_{i\in\{0,1\}}\biggl\{(1-i)\min_{j\in\{0,1\}}\biggl[(1-j)\inf_{\theta\in\mathbb{R}^l}\Bigl\{(1-\lambda h)\Bigl|v^0\bigl(s+h,y_1+hb(s,y_1;\theta)\bigr)\\
&-v^0\bigl(s+h,y_2+hb(s,y_2;\theta)\bigr)\Bigr|+h\bigl|f(s,y_1;\theta)-f(s,y_2;\theta)\bigr|\Bigr\}\\
&+j\Phi(h)\inf_{\xi\in U}\Bigl(\Bigl|v^0\bigl(s,y_1+g_\xi(s,y_1;\xi)\bigr)-v^0\bigl(s,y_2+g_\xi(s,y_2;\xi)\bigr)\Bigr|+\bigl|c(s,y_1;\xi)-c(s,y_2;\xi)\bigr|\Bigr)\biggr]\\
&+i\Phi(h)\sup_{\eta\in V}\Bigl(\Bigl|v^0\bigl(s,y_1+g_\eta(s,y_1;\eta)\bigr)-v^0\bigl(s,y_2+g_\eta(s,y_2;\eta)\bigr)\Bigr|+\bigl|\chi(s,y_1;\eta)-\chi(s,y_2;\eta)\bigr|\Bigr)\biggr\}.
\end{aligned}
\end{equation*}
By letting $\delta_1=\min\bigl\{\delta_0/(1+C_b/\lambda),\lambda\varepsilon/2C_f,\delta_0/C_{g_\xi},\delta_0/C_{g_\eta},\varepsilon/2C_c,\varepsilon/2C_\chi\bigr\}$, we get for all $0<h<1/\lambda$, for all $s\in[t,T)$ and for all $\|y_1-y_2\|<\delta_1$, that $\bigl|v^1_h(s,y_1)-v^1_h(s,y_2)\bigr|<\varepsilon$. The family $\{v^1_h\}$ is then uniformly equicontinuous with respect to state variable and by induction for all $n\geq1$, the family $\{v^n_h\}$ is uniformly equicontinuous. Now we prove that the family $\{v_h\}$ is also uniformly equicontinuous. Let $s\in[t,T)$, since we have that
$$\forall\varepsilon>0,\;\exists\delta>0,\;\forall y_1,y_2\in\mathbb{R}^n,\;\|y_1-y_2\|\leq\delta,\;\text{implies}\;\bigl\|v^n_h(s,y_1)-v^n_h(s,y_2)\bigr\|<\varepsilon/3,$$
and $\forall\varepsilon>0,\;\exists N>0,\;\forall n>N,\;\|y_1-y_2\|\leq\delta,$ implies $\bigl\|v^n_h(s,y_1)-v_h(s,y_1)\bigr\|<\varepsilon/3,$ and $\bigl\|v^n_h(s,y_2)-v_h(s,y_2)\bigr\|<\varepsilon/3$.
It follows that the family $\{v_h\}$ is uniformly equicontinuous.
\par Since $v^n_h=F^nv^0$ tends to the approximate value function $v_h$ and $f$ is non-negative, we get that $v_h$ is non-negative, we then use the fact that $Fv_h=v_h$ to deduce that $$\|v_h\|_\infty\leq(1-\lambda h)\|v_h\|_\infty+h\|f\|_\infty,$$
thus $v_h$ results to be uniformly bounded by $\|f\|_\infty/\lambda$.
\end{proof}
\begin{theorem}\label{theorem4.2}
The approximate value function $v_h$, as the time discretization step $h$ goes to zero, converges locally uniformly towards the value function of the zero-sum DG control problem.
\end{theorem}
\begin{proof}
Let $v_h$ be the approximate value function. From Lemma \ref{Lemma4.1}, the family $\{v_h\}$ is uniformly equicontinuous and uniformly bounded, we then get, using the Ascoli-Arzelà theorem (see e.g. \cite{BC97}), that from any sequence $h_r$ converging towards $0$, there exists a sub-sequence $h_s$ of $h_r$ and a BUC function $v$ such that $v_{h_s}$ converges locally uniformly in $[t,T]\times\mathbb{R}^n$ towards $v$. Now, we only need to prove that $v$ is a VS to the HJBI equation. Let $\phi\in C^{1,1}\bigl([t,T)\times\mathbb{R}^n\bigr)$ and $(\overline{t},\overline{x})$ be a strict local maximum point of $v-\phi$. Then there exists $\overline{B}_\delta(\overline{x})$ a closed ball in $\mathbb{R}^n$ of radius $\delta>0$ centered at $\overline{x}$ such that
$$(v-\phi)(\overline{t},\overline{x})>(v-\phi)(t,x),\;\text{for all}\;(t,x)\in I\backslash\{\overline{t}\}\times\overline{B}_\delta(\overline{x})\backslash\{\overline{x}\},$$
where $I:=[\overline{t}-\delta,\overline{t}+\delta]\subset[0,T]$. Let $(\overline{t}_{h_s},\overline{x}_{h_s})$ be a maximum point of $v_{h_s}-\phi$ over $I\times\overline{B}_\delta(\overline{x})$, let $\overline{t}_0$ and $\overline{x}_0$ be clusters point of the sequences $\{\overline{t}_{h_s}\}$ and $\{\overline{x}_{h_s}\}$, respectively, and denote $\{\overline{t}_{h_{s_p}}\}$ and $\{\overline{x}_{h_{s_p}}\}$ two sub-sequences converging to $\overline{t}_0$ and $\overline{x}_0$, respectively. By definition we have
$$(v_{h_{s_p}}-\phi)(\overline{t}_{h_{s_p}},\overline{x}_{h_{s_p}})\geq(v_{h_{s_p}}-\phi)(t,x),\;\text{for all}\;(t,x)\in I\times\overline{B}_\delta(\overline{x}).$$
Using the continuity of $v_{h_{s_p}}$ and $\phi$, and the fact that $v_{h_s}$ converges locally uniformly towards $v$, we get
$$(v-\phi)(\overline{t}_0,\overline{x}_0)\geq(v-\phi)(t,x),\;\text{for all}\;(t,x)\in I\times\overline{B}_\delta(\overline{x}).$$
Thus, by the uniqueness of the maximum, $(\overline{t}_0,\overline{x}_0)=(\overline{t},\overline{x})$ which means that the clusters point $\overline{t}_0$ and $\overline{x}_0$ are unique, we then get that the whole sequences $\overline{t}_{h_s}$ and $\overline{x}_{h_s}$ converge toward $\overline{t}$ and $\overline{x}$, respectively. Since $h_s$ is a small number and $b$ is assumed to be bounded, we get that the points $\overline{t}_{h_s}+h_s$ and $\overline{x}_{h_s}+h_{s}b(\overline{t}_{h_s},\overline{x}_{h_s};\theta)$ remain in $I$ and $\overline{B}_\delta(\overline{x})$, respectively, for all $\theta\in\mathbb{R}^l$. Then it follows
$$v_{h_s}(\overline{t}_{h_s},\overline{x}_{h_s})-\phi(\overline{t}_{h_s},\overline{x}_{h_s})\geq v_{h_s}\bigl(\overline{t}_{h_s}+h_s,\overline{x}_{h_s}+h_sb(\overline{t}_{h_s},\overline{x}_{h_s};\theta)\bigr)-\phi\bigl(\overline{t}_{h_s}+h_s,\overline{x}_{h_s}+h_sb(\overline{t}_{h_s},\overline{x}_{h_s};\theta)\bigr).$$
Since $v_h$ is a solution to the approximate equation (HJBI$_{h}$), the last inequality combined to the expression of the approximate equation (HJBI$_{h}$) in the proof of Proposition \ref{proposition4.1} gives
\begin{equation*}
\begin{aligned}
\max_{i\in\{0,1\}}\biggl\{(1-i)\min_{j\in\{0,1\}}\Bigl[&(1-j)\inf_{\theta\in\mathbb{R}^l}\Bigl\{\phi(\overline{t}_{h_s},\overline{x}_{h_s})-\phi\bigl(\overline{t}_{h_s}+h_s,\overline{x}_{h_s}+h_sb(\overline{t}_{h_s},\overline{x}_{h_s};\theta)\bigr)\\
&+\lambda h_sv_{h_s}\bigl(\overline{t}_{h_s}+h_s,\overline{x}_{h_s}+{h_s}b(\overline{t}_{h_s},\overline{x}_{h_s};\theta)\bigr)-h_sf(\overline{t}_{h_s},\overline{x}_{h_s};\theta)\Bigr\}\\
&+j\bigl(v_{h_s}(\overline{t}_{h_s},\overline{x}_{h_s})-\Phi(h_s)\mathcal{H}_{sup}^c v_{h_s}(\overline{t}_{h_s},\overline{x}_{h_s})\bigr)\Bigr]\\
&+i\bigl(v_{h_s}(\overline{t}_{h_s},\overline{x}_{h_s})-\Phi(h_s)\mathcal{H}_{inf}^\chi v_{h_s}(\overline{t}_{h_s},\overline{x}_{h_s})\bigr)\biggr\}\leq 0.
\end{aligned}
\end{equation*}
Since $\phi\in C^{1,1}\bigl([t,T)\times\mathbb{R}^n\bigr)$, then for some points $s$ and $y$ satisfying
$$s\in[\overline{t}_{h_s},\overline{t}_{h_s}+h_s]\;\text{and}\;y\in\bigl[\overline{x}_{h_s},\overline{x}_{h_s}+h_sb(\overline{t}_{h_s},\overline{x}_{h_s};\theta)\bigr],$$
we get that
\begin{equation*}
\begin{aligned}
\max_{i\in\{0,1\}}\biggl\{&(1-i)\min_{j\in\{0,1\}}\Bigl[(1-j)\inf_{\theta\in\mathbb{R}^l}\Bigl\{-h_s\frac{\partial\phi}{\partial s}(s,y)+\lambda h_sv_{h_s}\bigl(\overline{t}_{h_s}+h_s,\overline{x}_{h_s}+{h_s}b(\overline{t}_{h_s},\overline{x}_{h_s};\theta)\bigr)\\
&-h_sD_y\phi(s,y).b(s,y;\theta)-h_sf(\overline{t}_{h_s},\overline{x}_{h_s};\theta)\Bigr\}+j\bigl(v_{h_s}(\overline{t}_{h_s},\overline{x}_{h_s})-\Phi(h_s)\mathcal{H}_{sup}^c v_{h_s}(\overline{t}_{h_s},\overline{x}_{h_s})\bigr)\Bigr]\\
&+i\bigl(v_{h_s}(\overline{t}_{h_s},\overline{x}_{h_s})-\Phi(h_s)\mathcal{H}_{inf}^\chi v_{h_s}(\overline{t}_{h_s},\overline{x}_{h_s})\bigr)\biggr\}\leq 0.
\end{aligned}
\end{equation*}
Then the result in Proposition \ref{proposition4.1} yields to the following inequality
\begin{equation*}
\begin{aligned}
\max_{i\in\{0,1\}}\biggl\{&(1-i)\min_{j\in\{0,1\}}\Bigl[(1-j)\inf_{\theta\in\mathbb{R}^l}\Bigl\{-\frac{\partial\phi}{\partial s}(s,y)+\lambda v_{h_s}\bigl(\overline{t}_{h_s}+h_s,\overline{x}_{h_s}+{h_s}b(\overline{t}_{h_s},\overline{x}_{h_s};\theta)\bigr)\\
&-D_y\phi(s,y).b(s,y;\theta)-f(\overline{t}_{h_s},\overline{x}_{h_s};\theta)\Bigr\}+j\bigl(v_{h_s}(\overline{t}_{h_s},\overline{x}_{h_s})-\Phi(h_s)\mathcal{H}_{sup}^c v_{h_s}(\overline{t}_{h_s},\overline{x}_{h_s})\bigr)\Bigr]\\
&+i\bigl(v_{h_s}(\overline{t}_{h_s},\overline{x}_{h_s})-\Phi(h_s)\mathcal{H}_{inf}^\chi v_{h_s}(\overline{t}_{h_s},\overline{x}_{h_s})\bigr)\biggr\}\leq 0,
\end{aligned}
\end{equation*}
we then let $h_s$ goes to zero and use Proposition \ref{Proposition2.5} to get the convergence of the terms $\mathcal{H}_{sup}^{c}v_{h_s}(\overline{t}_{h_s},\overline{x}_{h_s})$ and $\mathcal{H}_{inf}^{\chi}v_{h_s}(\overline{t}_{h_s},\overline{x}_{h_s})$ toward $\mathcal{H}_{sup}^{c}v(\overline{t},\overline{x})$ and $\mathcal{H}_{inf}^{\chi}v(\overline{t},\overline{x})$, respectively, to finally deduce that
\begin{equation*}
\begin{aligned}
\max_{i\in\{0,1\}}\biggl\{&(1-i)\min_{j\in\{0,1\}}\Bigl[(1-j)\inf_{\theta\in\mathbb{R}^l}\Bigl\{-\frac{\partial\phi}{\partial s}(\overline{t},\overline{x})+\lambda v\bigl(\overline{t},\overline{x}\bigr)-D_y\phi(\overline{t},\overline{x}).b(\overline{t},\overline{x};\theta)-f(\overline{t},\overline{x};\theta)\Bigr\}\\
&+j\bigl(v(\overline{t},\overline{x})-\mathcal{H}_{sup}^c v(\overline{t},\overline{x})\bigr)\Bigr]+i\bigl(v(\overline{t},\overline{x})-\mathcal{H}_{inf}^\chi v(\overline{t},\overline{x})\bigr)\biggr\}\leq 0.
\end{aligned}
\end{equation*}
The last inequality shows, using the expression of HJBI equation given in the proof of Proposition \ref{Proposition3.1}, that the function $v$ is a viscosity sub-solution to the HJBI equation. Similarly we prove the viscosity super-solution property. The proof of Theorem \ref{theorem4.2} is then finished.
\end{proof}
\begin{corollary}
Assuming $\textbf{H}_b$, $\textbf{H}_g$, $\textbf{H}_f$, $\textbf{H}_{c,\chi}$ and $\textbf{H}_{G}$, the value function of the zero-sum DG control problem is the limit, when $h$ goes to zero, of the approximate value function $v_h$, i.e., the limit of the unique bounded continuous solution of the approximate Hamilton-Jacobi-Bellman-Isaacs equation (HJBI$_{h}$). \qed
\end{corollary}
Now we move to the third contribution of the paper as mentioned in Remark \ref{Remark2.3}.
\section{Verification Theorem}\label{Sect.5}
This section uses the fact that the approximate value function converges to the value function of the considered DG control problem to provide a NE strategy for this game, whose definition is given by the following:
\begin{definition}[Nash-Equilibrium]
Given $(t,x)\in[0,T]\times\mathbb{R}^n$, we say that the zero-sum DG control problem studied admits $(\psi^*,v^*)\in\Psi\times\mathcal{V}$ as a NE if the two strategies $\psi^*$ and $v^*$ satisfies:
\begin{equation*}
\left\{
\begin{aligned}
J(t,x;\psi^*,v^*)&\geq J(t,x;\psi,v^*)\;\text{for all}\;\psi\in\Psi;\\
J(t,x;\psi^*,v^*)&\leq J(t,x;\psi^*,v)\;\text{for all}\;v\in\mathcal{V}.
\end{aligned}
\right.
\end{equation*} \qed
\end{definition}
In view of the above definition, the value function of the NE $(\psi^*,v^*)\in\Psi\times\mathcal{V}$ is defined for all $(t,x)\in[0,T]\times\mathbb{R}^n$ by $$V(t,x):=J\bigl(t,x;\psi^*,v^*).$$
\par We will be concerned here with the optimal strategies for our two-player, zero-sum, deterministic DG continuous and impulse controls problem. We first suppose, for $(t,x)\in[0,T]\times\mathbb{R}^n$, that a classical solution $v(t,x)$ of the HJBI equation and an approximate value function $v_h(t,x)$ exist and satisfy, for all $y\in\mathbb{R}^n$, $v(T,y)=G(y)$ and $v_h(T,y)=G(y)$, respectively. Next, let $h$ be a constant which tends to zero and $\Phi$ be defined as in Section \ref{Sect.2.3}. Then we construct the optimal strategies of each player $\psi^*:=\bigl(\theta^*(.),u^*:=(\tau_m^*,\xi_m^*)_{m\in\mathbb{N}^*}\bigr)$ and $v^*:=(\rho_k^*,\eta_k^*)_{k\in\mathbb{N}^*}$ in an inductive way as follows:
\begin{equation*}
\theta^*(.):=
\left\{
\begin{aligned}
&\theta^*(t^-)=:\theta_0^*\in\mathbb{R}^l\;\text{initial value of the optimal continuous control};\\
&\begin{aligned}
\theta^*(s)=\Bigl\{\theta^*\in\mathbb{R}^l:&\;v_h\bigl(s,y_{t,x}^{\psi^*,v^*}(s)\bigr)-(1-\lambda h)v_h\Bigl(s+h,y_{t,x}^{\psi^*,v^*}(s)+hb\bigl(s,y_{t,x}^{\psi^*,v^*}(s);\theta^*\bigr)\Bigr)\\
&-hf\bigl(s,y_{t,x}^{\psi^*,v^*}(s);\theta^*\bigr)=0\Bigr\},\;\text{where}\;s\in(t,T]\;\text{and}\;s\neq\tau_m^*,\rho_k^*\;\text{for all}\;m,k\geq 1,
\end{aligned}\\
&\text{with}\;\theta^*=\arg\sup_{\theta\in\mathbb{R}^l}\Bigl\{(1-\lambda h)v_h\bigl(s+h,y_{t,x}^{\psi^*,v^*}(s)+hb(s,y_{t,x}^{\psi^*,v^*}(s);\theta)\bigr)+hf(s,y_{t,x}^{\psi^*,v^*}(s);\theta)\Bigl\};
\end{aligned}
\right.
\end{equation*}
\begin{equation*}
u^*:=
\left\{
\begin{aligned}
&\tau_1^*=t^-\;\text{and}\;\xi_1^*=0\in U\subset\mathbb{R}^p;\\
&\text{and for any}\;m\geq 2,\\
&\tau_m^*=
\left\{
\begin{aligned}
&\inf\Biggl\{s>\tau_{m-1}^*:v_h\bigl(s,y_{t,x}^{\psi^*,v^*}(s^-)\bigr)-\Phi(h)\sup_{\xi\in U}\biggl\{v_h\Bigl(s,y_{t,x}^{\psi^*,v^*}(s^-)+g_\xi\bigl(s,y_{t,x}^{\psi^*,v^*}(s^-);\xi\bigr)\Bigr)\\
&-c\bigl(s,y_{t,x}^{\psi^*,v^*}(s^-);\xi\bigr)\biggr\}\geq0,\;\text{where}\;s<T\;\text{and}\;\tau_m^*\neq\rho_k^*\;\text{for all}\;k\geq 1\Biggr\};\\
&T\;\text{if the above set is empty};
\end{aligned}
\right.\\
&\xi_m^*=
\left\{
\begin{aligned}
&\arg\sup_{\xi\in U}\biggl\{v_h\Bigl(\tau_m^*,y_{t,x}^{\psi^*,v^*}({\tau_m^*}^-)+g_\xi\bigl(\tau_m^*,y_{t,x}^{\psi^*,v^*}({\tau_m^*}^-);\xi\bigr)\Bigr)-c\bigl(\tau_m^*,y_{t,x}^{\psi^*,v^*}({\tau_m^*}^-);\xi\bigr)\biggr\}\;\text{if}\;\tau_m^*<T;\\
&\xi\;\text{solution of}\;g_\xi\bigl(T,y_{t,x}^{\psi^*,v^*}(T^-);\xi\bigl)=0\;\text{if}\;\tau_m^*=T\;(\text{no intervention}),
\end{aligned}
\right.
\end{aligned}
\right.
\end{equation*}
and
\begin{equation*}
v^*:=
\left\{
\begin{aligned}
&\rho_1^*=t\;\text{and}\;\eta_1^*=\eta_1\in V\subset\mathbb{R}^q;\\
&\text{and for any}\;k\geq 2,\\
&\rho_k^*=
\left\{
\begin{aligned}
&\inf\Biggl\{s>\rho_{k-1}^*:v_h\bigl(s,y_{t,x}^{\psi^*,v^*}(s^-)\bigr)-\Phi(h)\inf_{\eta\in V}\biggl\{v_h\Bigl(s,y_{t,x}^{\psi^*,v^*}(s^-)+g_\eta\bigl(s,y_{t,x}^{\psi^*,v^*}(s^-);\eta\bigr)\Bigr)\\
&+\chi\bigl(s,y_{t,x}^{\psi^*,v^*}(s^-);\eta\bigr)\biggr\}\leq0,\;\text{where}\;s<T\;\Biggr\};\\
&T\;\text{if the above set is empty};
\end{aligned}
\right.\\
&\eta_k^*=
\left\{
\begin{aligned}
&\arg\inf_{\eta\in V}\biggl\{v_h\Bigl(\rho_k^*,y_{t,x}^{\psi^*,v^*}({\rho_k^*}^-)+g_\eta\bigl(\rho_k^*,y_{t,x}^{\psi^*,v^*}({\rho_k^*}^-);\eta\bigr)\Bigr)+\chi\bigl(\rho_k^*,y_{t,x}^{\psi^*,v^*}({\rho_k^*}^-);\eta\bigr)\biggr\}\;\text{if}\;\rho_k^*<T;\\
&\eta\;\text{solution of}\;g_\eta\bigl(T,y_{t,x}^{\psi^*,v^*}(T^-);\eta\bigl)=0\;\text{if}\;\rho_k^*=T\;(\text{no intervention}).
\end{aligned}
\right.
\end{aligned}
\right.
\end{equation*}
Next, we show that the above strategies are optimal and form a NE for the value function when the time discretization step $h$ goes to zero.
\par The following Theorem \ref{Verification} announces a NE for the DG control problem we have considered in this paper, it gives a verification result and confirms that $(\psi^*,v^*)$ defined in the above are optimal strategies for both players:
\begin{theorem}[Verification Theorem]\label{Verification}
Assuming that $(\psi^*,v^*)\in\Psi\times\mathcal{V}$ and letting $h$ goes to zero, if the value function of the zero-sum DG control problem is in $C^{1,1}\bigl([0,T]\times\mathbb{R}^n\bigr)$, then it satisfies
$$V(t,x)=J(t,x;\psi^*,v^*)\;\text{for any}\;(t,x)\in[0,T]\times\mathbb{R}^n.$$
\end{theorem}
\begin{proof}
We begin the proof by assuming that both HJBI equation and the approximate equation (HJBI$_{h}$) have solutions denoted, respectively, by $V(t,x)$ and $v_h(t,x)$ for $(t,x)\in[0,T]\times\mathbb{R}^n$. Next, we consider the following related discrete-time DG control problems involving continuous and impulse controls:
\begin{equation*}
\begin{aligned}
V_h^-(t,x)&:=\inf_{\beta\in\mathcal{B}_h}\sup_{\psi\in\Psi_h}J_h\bigl(t,x;\psi,\beta(\psi)\bigr);\\
V_h^+(t,x)&:=\sup_{\alpha\in\mathcal{A}_h}\inf_{v\in\mathcal{V}_h}J_h\bigl(t,x;\alpha(v),v\bigr),
\end{aligned}
\end{equation*}
where, for the time discretization step $h$ and $d\in\mathbb{D}:=\bigl\{0,1,2,\dots,\frac{T-t}{h}-1\bigr\}$, the discrete-time mapping $y_{t,x}^h:\mathbb{D}\rightarrow\mathbb{R}^n$ depends on controls $\psi$ and $v$, and determines the discrete-time state of the DG control problems ($V_h^-$) and ($V_h^+$) by the following recursion:
\begin{equation*}
\begin{aligned}
y_{t,x}^{h}(0)=&\;x;\\
y_{t,x}^{h}(d+1)=&\;y_{t,x}^{h}(d)+hb\bigl(t+dh,y_{t,x}^{h}(d);\theta(t+dh)\bigr)\prod_{m\geq 1}\ind_{\Bigl\{\tau_m\notin\bigl[t+dh,t+(d+1)h\bigr[\Bigr\}}\prod_{k\geq 1}\ind_{\Bigl\{\rho_k\notin\bigl[t+dh,t+(d+1)h\bigr[\Bigr\}}\\
&+\sum_{m\geq1}g_\xi\bigl(\tau_m,y_{t,x}^{h}(d);\xi_m\bigr)\ind_{\bigl[t+dh,t+(d+1)h\bigr[}(\tau_m)\prod_{k\geq 1}\ind_{\{\tau_m\neq\rho_k\}};\\
&+\sum_{k\geq1}g_\eta\bigl(\rho_k,y_{t,x}^{h}(d);\eta_k\bigr)\ind_{\bigl[t+dh,t+(d+1)h\bigr[}(\rho_k),
\end{aligned}
\end{equation*}
the discrete-time gain/cost functional $J_h$ is given by
\begin{equation*}
\begin{aligned}
J_h(t,x;\psi,v):=&\;h\sum_{d\in\mathbb{D}}f\bigl(t+d h,y_{t,x}^{h}(d);\theta(t+dh)\bigr)(1-\lambda h)^d\\
&-\sum_{m\geq 1}\sum_{d\in\mathbb{D}}c\bigl(\tau_m,y_{t,x}^h(d);\xi_m\bigr)(1-\lambda h)^d\ind_{\bigl[t+dh,t+(d+1)h\bigr[}(\tau_m)\prod_{k\geq 1}\ind_{\{\tau_m\neq\rho_k\}}\\
&+\sum_{k\geq 1}\sum_{d\in\mathbb{D}}\chi\bigl(\rho_k,y_{t,x}^h(d);\eta_k\bigr)(1-\lambda h)^d\ind_{\bigl[t+dh,t+(d+1)h\bigr[}(\rho_k)\\
&+G\biggl(y_{t,x}^h\Bigl(\frac{T-t}{h}\Bigr)\biggr)(1-\lambda h)^{\frac{T-t}{h}},
\end{aligned}
\end{equation*}
\begin{equation*}
\begin{aligned}
\Psi_h&:=\Bigl\{\text{Subset of}\;\Psi\;\text{consisting of all controls with constant values on each interval}\;\bigl[t+dh,t+(d+1)h\bigr[\Bigr\};\\
\mathcal{B}_h&:=\Bigl\{\text{Subset of}\;\mathcal{B}\;\text{consisting of all non-anticipative strategies of}\;\Psi_h\;\text{to}\;\mathcal{V}_h,\;\text{where}\;\mathcal{V}_h\;\text{is the set of impulse}\\&\text{controls with constant impulse values on each interval}\;\bigl[t+dh,t+(d+1)h\bigr[\Bigr\},
\end{aligned}
\end{equation*}
similarly we define the set $\mathcal{A}_h$. Following \cite{CD83} one might deduce, for any $t\in[0,T]$ and $x\in\mathbb{R}^n$, the representation formula $v_h(t,x)=V_h^-(t,x)=V_h^+(t,x)$. Hence by focusing only on the discrete-time control problem ($V_h^-$) we can deduce the optimal strategies. In other words, we use the fact that $v_h$ is the unique bounded continuous solution to the approximate equation (HJBI$_{h}$), the formula $v_h=V_h^-$ and the convergence $v_h\underset{h\rightarrow0}{\rightarrow}V$ to define some discrete-time optimal controls $$\psi^*_h:=\bigl(\theta^*_h(.),u^*_h:=({\tau_m^h}^*,{\xi_m^h}^*)_{m\in\mathbb{N}^*}\bigr)\in\Psi_h\times\mathcal{U}_h,\;\text{and}\;v^*_h:=({\rho_k^h}^*,{\eta_k^h}^*)_{k\in\mathbb{N}^*}\in\mathcal{V}_h,$$
for both discrete-time DG control problems ($V_h^-$) and ($V_h^+$). Since the function $v_h$ separate the domain $[t,T]\times\mathbb{R}^n$ into many regions including the following region:
\begin{equation*}
\begin{aligned}
\mathcal{R}:=\Bigl\{(s,y)\in[t,T]\times\mathbb{R}^n:\;&H_h\bigl(s,y,v_h(s,y)\bigr)=0,\;v_h(s,y)-\Phi(h)\mathcal{H}_{sup}^c v_h(s,y)\geq0\\
&\text{and}\;v_h(s,y)-\Phi(h)\mathcal{H}_{inf}^\chi v_h(s,y)\leq0\Bigr\},
\end{aligned}
\end{equation*}
then the expressions of the optimal impulse stopping times ${\tau_m^h}^*,{\rho_k^h}^*$ and values ${\xi_m^h}^*,{\eta_k^h}^*$ follow immediately and were given, respectively, by the aforementioned expressions $u^*$ and $v^*$ for $h$ tends to zero. We now focus on the optimal continuous control $\theta_h^*(.)$ by assuming, without loss of generality, that there are no impulse controls for both players, i.e. $\tau_1=\rho_1=T$, and proceeding as in \cite{CDI84}. It will be useful in what follows to consider the piece-wise constant extension $\tilde{y}_{t,x}^h(.)$ to $[t,T]$ of the mapping $s\to y_{t,x}^h(s/h)$ defined, for $k\in\{0,1,2,\dots,\frac{T-t}{h}\}$, on $\{t+kh\}$ by $\tilde{y}_{t,x}^h(s)=y_{t,x}^h\bigl([s/h]\bigr)$, where $[s/h]$ denotes the largest integer which is less than or equal to $s/h$. From the definition of the region $\mathcal{R}$, we deduce that there exists a function $\theta_h^*:\mathbb{R}^n\to\mathbb{R}^l$, such that for all $(s,y)\in[t,T]\times\mathbb{R}^n$ we have
\begin{equation}\label{equation5.1}
v_h(s,y)-(1-\lambda h)v_h\Bigl(s+h,y+hb\bigl(s,y;\theta_h^*(y)\bigr)\Bigr)-hf\bigl(s,y;\theta_h^*(y)\bigr)=0,
\end{equation}
where
$$\theta_h^*(y)=\arg\sup_{\theta\in\mathbb{R}^l}\Bigl\{(1-\lambda h)v_h\bigl(s+h,y+hb(s,y;\theta)\bigr)+hf(s,y;\theta)\Bigl\},$$
define then a discrete-time state mapping ${y_{t,x}^h}^*:\mathbb{D}\to\mathbb{R}^n$ by
$${y_{t,x}^{h}}^*(0)=x,\;\text{and}\;{y_{t,x}^{h}}^*(d+1)={y_{t,x}^{h}}^*(d)+hb\Bigl(t+dh,{y_{t,x}^{h}}^*(d);\theta_h^*\bigl({y_{t,x}^{h}}^*(d)\bigr)\Bigr),$$
and a function $\tilde{\theta}_h^*:[t,T]\to\mathbb{R}^l$ by
$$\tilde{\theta}_h^*(s)=\theta_h^*\Bigl({y_{t,x}^{h}}^*\bigl([s/h]\bigr)\Bigr),\;\text{for all}\;s\in[t,T].$$
Equation (\ref{equation5.1}) leads, for $d\in\mathbb{D}$, to $$v_h(s,y)=(1-\lambda h)^dv_h\bigl(s+dh,{y_{t,x}^{h}}^*(d)\bigr)+h\sum_{i=1}^{d-1}(1-\lambda h)^if\Bigl(s,{y_{t,x}^{h}}^*(i);\theta_h^*\bigl({y_{t,x}^{h}}^*(i)\bigr)\Bigr).$$
The fact that the control $\tilde{\theta}_h^*(.)$ has constant values on each interval $\bigl[t+dh,t+(d+1)h\bigr[$ and the boundedness of $v_h$ confirm that $\tilde{\theta}_h^*(.)$ in the optimal continuous control for the problem with no impulses, for which the expression was giving by $\theta^*(.)$. Therefore, using the representation formula $v_h=V_h^-$ we get $v_h(t,x)=J_h(t,x;\psi_h^*,v_h^*)$. Now, following \cite{CDI84}, we write $$\lim_{h\rightarrow 0}J_h(t,x;\psi_h^*,v_h^*)=\lim_{h\rightarrow 0}J(t,x;\psi_h^*,v_h^*),$$ we then get, from the convergence $v_h\underset{h\rightarrow0}{\rightarrow}V$, that $\lim_{h\rightarrow 0}\tilde{\theta}_h^*(.)$ represents the optimal continuous control for the value $V$. Thus $V(t,x)=J(t,x;\psi^*,v^*)$ for optimal controls given by $\theta^*(.)$, $u^*$ and $v^*$. Hence we obtain the thesis.
\end{proof}
Hence, the third contribution of the paper as mentioned in Remark \ref{Remark2.3}. The obtained results make us ready to introduce a new continuous-time portfolio optimization model as an application, and this is the subject of the next Section \ref{Sect.6}.
\section{Application to Continuous-Time Portfolio Optimization}\label{Sect.6}
An interesting framework of the theory of deterministic finite-time horizon, two-player, zero-sum, DGs involving continuous and impulse controls, developed in the present paper, is provided by the \textit{continuous-time portfolio optimization} problem. In this section we address an application of our results to the analysis of a \textit{new} continuous-time portfolio optimization model, in which the investor plays against the market and wishes to maximize his \textit{discounted terminal payoff}, or to minimize a given \textit{cost}. In Section \ref{Sect.6.1} below the dynamical system (S$^\pi$) describes the investor's \textit{wealth} at time $s\in[t,T]$, while the functional $J^\pi$ represents his \textit{discounted terminal gain/cost}. On one hand, the market (maximizing player$-\xi$) wishes to minimize the investor's discounted terminal payoff (i.e., maximize the gain functional $J^\pi$, where, on the other hand, the investor (minimizing player$-\eta$) uses an impulse control to re-balance his portfolio in order to minimize the given cost functional $J^\pi$. Thus, the value function represents the investor's lost in the \textit{worst-case scenario}. Hence, our results can be used to derive a new continuous-time portfolio optimization model.
\subsection{Formulation of a New Continuous-Time Portfolio Optimization Model}\label{Sect.6.1}
We describe hereafter our two-player, zero-sum, deterministic DG approach for continuous-time portfolio optimization problem in finite-time horizon. We first adjust the expressions of the dynamics $b$, $g_\xi$ and $g_\eta$ in the standing dynamical system (S) to get a new one (S$^\pi$), which characterizes the investor's wealth at each instant $s$ between the initial time $t$ and the horizon $T$. Next, we approach the resulted continuous-time portfolio optimization problem by the non-linear HJBI equation and its approximate equation (HJBI$_{h}$).
\subsubsection{Dynamic of the Portfolio's State} Our finite-time horizon deterministic DG approach leads to a new continuous-time portfolio optimization model in which the investor's wealth is described by the following dynamical system (S$^\pi$):
\begin{equation*}
\text{(S$^\pi$)}\;\left\{
\begin{aligned}
\dot W_{t,w}(s)&=W_{t,w}(s)\sum_{i=1}^{N}\omega_i^\pi(s)\dot{R_i}(s),\;s\neq\tau_m,\;s\neq\rho_k,\;s\in[t,T],\;\text{where}\;t\geq 0\;\text{and}\;T\in(0,+\infty);\\
W_{t,w}(\tau_m^+)&=W_{t,w}(\tau_m^-)\biggl(1+\sum_{i=1}^{N}\omega^\xi_{i,m}dR_i(\tau_m)\prod_{k\geq 1}\ind_{\{\tau_m\neq\rho_k\}}\biggr),\;\tau_m\in[t,T],\;\bigl[\omega^\xi_{1,m},\dots,\omega^\xi_{N,m}\bigr]^\top\neq 0;\\
W_{t,w}(\rho_k^+)&=W_{t,w}(\rho_k^-)\biggl(1+\sum_{i=1}^{N}\omega^\eta_{i,k}dR_i(\rho_k)\biggr),\;\rho_k\in[t,T],\;\bigl[\omega^\eta_{1,k},\dots,\omega^\eta_{N,k}\bigr]^\top\neq 0;\\
W_{t,w}(t^-)&=w\;(\text{investor's initial wealth}).
\end{aligned}
\right.
\end{equation*}
Here $N=l=p=q$ is the number of stocks in the market, $\top$ denotes transpose, and $R_i(s)$ is the function that describes the cumulative return of $i-$th stock up to time $s$ starting from $t$, where $dR_i(s)=\frac{dP_i(s)}{P_i(s)}$ for $P_i(s)$ being the price of $i-$th stock at time $s$. The mapping $W_{t,w}:[t,T]\rightarrow\mathbb{R}_+$ represents the investor's wealth at time $s\in[t,T]$ with initial value $w>0$ at time $t^-$. The wealth $W_{t,w}(s)$ gives the state of the investor's portfolio $\pi$ at time $s$ which is controlled by:
\begin{enumerate}[i.]
\item A \textit{continuous control} $\omega^\pi(.):=\bigl[\omega_1^\pi(.),\dots,\omega_N^\pi(.)\bigr]^\top$ which represents the investor's instantaneous portfolio composition, i.e., the portfolio's weights vector resulted from the market fluctuations. Thus, the vector $\omega^\pi(s)$, combined with the cumulative returns vector $\bigl[R_1(s),\dots,R_N(s)\bigr]^\top$, characterizes the investor's wealth at any time $s\in[t,T]$;
\item Two \textit{Impulse controls} $$u:=\Bigl(\tau_m,\omega^\xi_m:=\bigl[\omega^\xi_{1,m},\dots,\omega^\xi_{N,m}\bigr]^\top\Bigr)_{m\in\mathbb{N}^*},\;\text{and}\;v:=\Bigl(\rho_k,\omega^\eta_k:=\bigl[\omega^\eta_{1,k},\dots,\omega^\eta_{N,k}\bigr]^\top\Bigr)_{k\in\mathbb{N}^*},$$ which describe new investor's portfolio compositions at some jump instants $\tau_m$ and $\rho_k$, respectively. That is whenever the continuous control $\omega^\pi(.)$ doesn't perform, the market (player$-\xi$) uses a new optimal portfolio composition determined at each impulse instant $\tau_m$ by the impulse value $\omega^\xi_m$, while the investor (player$-\eta$) adjusts his portfolio at each impulse instant $\rho_k$ using the impulse value $\omega^\eta_k$ to outperform the market.
\end{enumerate}
\begin{remark}[Another Formulation]
If $r_i(s):=\dot R_i(s)=\frac{dR_i(s)}{ds}$ denotes the instantaneous return of $i-$th stock, i.e., $R_i(s):=\int_{t}^{s}r_i(\tau)d\tau$ is the cumulative return of $i-$th stock on $[t,s]$ satisfying $R_i(t)=0$, then our dynamical system (S$^\pi$) can be rewritten as follows:
\begin{equation*}
\left\{
\begin{aligned}
\dot W_{t,w}(s)&=W_{t,w}(s)\sum_{i=1}^{N}\omega_i^\pi(s)r_i(s),\;s\neq\tau_m,\;s\neq\rho_k,\;s\in[t,T],\;\text{where}\;t\geq 0\;\text{and}\;T\in(0,+\infty);\\
\dot W_{t,w}(\tau_m)&=W_{t,w}(\tau_m^-)\sum_{i=1}^{N}\omega^\xi_{i,m}r_i(\tau_m)\prod_{k\geq 1}\ind_{\{\tau_m\neq\rho_k\}},\;\tau_m\in[t,T],\;\bigl[\omega^\xi_{1,m},\dots,\omega^\xi_{N,m}\bigr]^\top\neq 0;\\
\dot W_{t,w}(\rho_k)&=W_{t,w}(\rho_k^-)\sum_{i=1}^{N}\omega^\eta_{i,k}r_i(\rho_k),\;\rho_k\in[t,T],\;\bigl[\omega^\eta_{1,k},\dots,\omega^\eta_{N,k}\bigr]^\top\neq 0;\\
W_{t,w}(t^-)&=w\;(\text{investor's initial wealth}).
\end{aligned}
\right.
\end{equation*} \qed
\end{remark}
Since $\omega^\pi(.),\;\omega^\xi_m$ and $\omega^\eta_k$ are three weights vectors, then the following constraint has to be satisfied:
\begin{equation*}
\text{(C)}\;\sum_{i=1}^{N}\omega_i^\pi(s)=\sum_{i=1}^{N}\omega_{i,m}^\xi=\sum_{i=1}^{N}\omega_{i,k}^\eta=1,\;\text{for any}\;s\in[t,T]\;\text{and}\;m,k\in\mathbb{N}^*.
\end{equation*}
\subsubsection{Continuous-Time Portfolio Optimization Problem}
Our zero-sum deterministic DG approach consists then in defining the investor's wealth $W_{t,w}(s)$ at time $s\in[t,T]$ by the solution of the following dynamical equation:
\begin{equation*}
\begin{aligned}
\text{(E2)}\;W_{t,w}(s)=&\;w+\int_{t}^{s}W_{t,w}(\tau)\sum_{i=1}^{N}\omega_i^\pi(\tau)\frac{dP_i(\tau)}{P_i(\tau)}+\sum_{m\geq 1}W_{t,w}(\tau_m^-)\sum_{i=1}^{N}\omega^\xi_{i,m}\frac{dP_i(\tau_m)}{P_i(\tau_m)}\ind_{[\tau_m,T]}(s)\prod_{k\geq 1}\ind_{\{\tau_m\neq\rho_k\}}\\
&+\sum_{k\geq 1}W_{t,w}(\rho_k^-)\sum_{i=1}^{N}\omega^\eta_{i,k}\frac{dP_i(\rho_k)}{P_i(\rho_k)}\ind_{[\rho_k,T]}(s).
\end{aligned}
\end{equation*}
\par We denote by $\psi:=\Bigl(\omega^\pi(.),u:=\bigl(\tau_m,\omega^\xi_m\bigr)_{m\in\mathbb{N}^*}\Bigr)\in\Psi$ and $v:=\bigl(\rho_k,\omega^\eta_k\bigr)_{k\in\mathbb{N}^*}\in\mathcal{V}$ the continuous-impulse control for the market (player$-\xi$) and the impulse control for the investor (player$-\eta$), respectively, and we assume that the investor reacts immediately to the market whereas the market is not so quick in reacting to the investor's moves, i.e., the investor's impulse action comes first whenever the impulse times for the two players coincide. Moreover, we assume that the investor does not consume wealth in the process of investing but is only interested to maximize his discounted terminal payoff, that is minimizing the following gain/cost functional:
\begin{equation*}
\begin{aligned}
J^\pi(t,w;\psi,v):=&\int_t^T f^\pi\bigl(s,W^{\psi,v}_{t,w}(s);\omega^\pi(s)\bigr)\exp\bigl(-\lambda(s-t)\bigr)ds\\
&-\sum_{m\geq 1}c^\pi\bigl(\tau_m,W^{\psi,v}_{t,w}(\tau_m^-);\omega_m^\xi\bigr)\exp\bigl(-\lambda(\tau_m-t)\bigr)\ind_{\{\tau_m\leq T\}}\prod_{k\geq 1}\ind_{\{\tau_m\neq\rho_k\}}\\
&+\sum_{k\geq 1}\chi^\pi\bigl(\rho_k,W^{\psi,v}_{t,w}(\rho_k^-);\omega_k^\eta\bigr)\exp\bigl(-\lambda(\rho_k-t)\bigr)\ind_{\{\rho_k\leq T\}}\\
&+G^\pi\bigl(W^{\psi,v}_{t,w}(T)\bigr)\exp\bigl(-\lambda(T-t)\bigr),
\end{aligned}
\end{equation*}
where the functional $J^\pi$ represents the investor's discounted terminal cost, with the following components:
\begin{enumerate}[i.]
\item The running gain/cost of integral type giving, for example, by the investor's stokes holding cost $l^\pi$ minus his instantaneous utility function $u^\pi$, that is $f^\pi(.,.;.):=(l^\pi-u^\pi)(.,.;.);$
\item The maximizing player's (market) $\bigl(\text{resp. minimizing player's (investor)}\bigr)$ cost function $c^\pi$ $\bigl(\text{resp.}\;\chi^\pi\bigr)$ which corresponds to the cost of selling or buying stokes at impulse instant $\tau_m$ $(\text{resp.}\;\rho_k)$;
\item The terminal gain/cost giving by the function $G^\pi$.
\end{enumerate}
\par Our portfolio model is then related to either one of the following optimization problems:
\begin{equation*}
\text{(P)}\;\left\{
\begin{aligned}
&\inf_{\beta\in\mathcal{B}}\sup_{\psi\in\Psi}\;J^\pi\bigl(t,w;\psi,\beta(\psi)\bigr),\;\text{or}\;\sup_{\alpha\in\mathcal{A}}\inf_{v\in\mathcal{V}}\;J^\pi\bigl(t,w;\alpha(v),v\bigr);\\
&\text{Subject to Equation (E2) and Constraint (C)}.
\end{aligned}
\right.
\end{equation*}
\subsection{Main Results and Portfolio Strategy}
We assume that the market moves according to the continuous control $\omega^\pi(.)$, creates jumps at impulse instants $\tau_m$ and tries to maximize the gain/cost functional $J^\pi$, and that the investor creates jumps at impulse instants $\rho_k$, obviously, trying to minimize $J^\pi$. We also make the assumption that the flow of funds is between the investor and the market which makes our zero-sum DG framework. Tacking into account the fact that the dynamical function $(s,w,\omega)\in[t,T]\times\mathbb{R}_+\times\mathbb{R}^N\rightarrow w\omega.P(s)\in\mathbb{R}_+$ satisfies, for a bounded $\mathbb{R}^N-$ valued function $P(s)$, the assumptions $\textbf{H}_b$ and $\textbf{H}_g$, and assuming that $\textbf{H}_f$, $\textbf{H}_{c,\chi}$ and $\textbf{H}_{G}$ hold for the functions $f^\pi,\;c^\pi,\chi^\pi$ and $G^\pi$, respectively. We might then use our results to solve the problem (P) and to conclude that the investor's maximal discounted terminal cost (i.e., the value function of the zero-sum deterministic DG control problem) can be characterized:
\begin{enumerate}[i.]
\item As the unique VS to the HJBI equation;
\item Or, as the limit of the approximate value function, i.e., the limit of the unique solution of the approximate equation (HJBI$_{h}$);
\item Or, by the optimal strategies of the NE of the zero-sum deterministic DG control problem.
\end{enumerate}
\par The following Corollary \ref{Corollary6.1} summarizes the discussion in the above by giving the \textit{portfolio strategy} for the investor and the related maximal \textit{lost} provided by the model we have developed:
\begin{corollary}\label{Corollary6.1}
A portfolio strategy $\Pi(s)$ for the investor is given, at time $s\in[t,T]$ for finite-time horizon $T$ and initial time $t$ in a market with $N$ stocks, by:
$$\Pi(s):=\Bigl\{\underset{\text{Instantaneous Market Compositions}}{\underbrace{\omega^\pi(s):=\bigl(\omega^{\pi}(s^\prime)\bigr)_{t\leq s^\prime\leq s}}};\;\underset{\text{Impulse Control of the Market}}{\underbrace{\omega^\xi(s):=\sum_{m\geq 1}\omega^{\xi}_{m}\ind_{[\tau_m,T]}(s)}};\;\underset{\text{Impulse Control of the Investor}}{\underbrace{\omega^\eta(s):=\sum_{k\geq 1}\omega^{\eta}_k\ind_{[\rho_k,T]}(s)}}\Bigr\},$$
where
\begin{equation*}
\left\{
\begin{aligned}
\omega^{\pi}(s)&:=\bigl[\omega_1^{\pi}(s),\omega_2^{\pi}(s),\dots,\omega_{N}^{\pi}(s)\bigr]^\top;\\
\omega^{\xi}_m&:=\bigl[\omega_{1,m}^{\xi},\omega_{2,m}^{\xi},\dots,\omega_{N,m}^{\xi}\bigr]^\top;\\
\omega^{\eta}_k&:=\bigl[\omega_{1,k}^{\eta},\omega_{2,k}^{\eta},\dots,\omega_{N,k}^{\eta}\bigr]^\top.
\end{aligned}
\right.
\end{equation*}
The optimal portfolio strategy $\Pi^*(s)$ is then described, for $s\in[t,T]$ and $\tau_m,\rho_k\leq s$, by the optimal sequences
\begin{equation*}
{\omega^{\pi}}^*(s),\;{\omega^\xi}^*:=(\tau_m^*,{\omega^{\xi}_m}^*)_{m\geq 1}\;\text{and}\;{\omega^\eta}^*:=(\rho_k^*,{\omega^{\eta}_k}^*)_{k\geq1},
\end{equation*}
of elements of $\mathbb{R}^N,\;[t,s]\times\mathbb{R}^N$ and $[t,s]\times\mathbb{R}^N$, respectively. The expressions of these optimal sequences are given by the verification theorem of Section \ref{Sect.5}. The investor's maximal lost is given by the value function $v(t,w)$ of the game for $w$ being the initial wealth at initial time $t$, that is by the solution of the problem (P) generated by the optimal portfolio strategy $\Pi^*(s)$. \qed
\end{corollary}
Thus, the fourth contribution of the paper as mentioned in Remark \ref{Remark2.3}. We now provide computational algorithms for our zero-sum DG control problem.
\section{Computational Algorithms}\label{Sect.7}
Here, we give numerical aspects describing the value functions for both players, and their NE strategy and state. More precisely, we propose two computational algorithms, Algorithm \ref{Alg.1} and \ref{Alg.2}, to find the approximate value function, i.e., the value function for a time discretization step tending to zero. Using these algorithms, the NE strategy will be deduced as well as the optimal evolution of the state for our DG control problem. Our algorithms are based on the \textit{Value Iteration} and \textit{Policy Iteration} techniques (see for example Alla \& al. \cite{AFK15} and Bokanowski \& al. \cite{BMZ09}) and the \textit{Explicit Euler Scheme}. The Remark \ref{Remark7.1} below gives a brief discussion of the implementation procedure.
\begin{remark}\label{Remark7.1}
Two algorithms will be given to compute the approximate value function $v_h$, the related NE strategy $\Bigl\{\psi^*:=\bigl(\theta^*(.),u^*\bigr);v^*\Bigr\}$ and the optimal state evolution $y_{t,x}^*(.)$. Algorithm \ref{Alg.1} describes the implementation of the NE strategy to perform the desired computations and it is divided into two phases:
\begin{enumerate}[]
\item \textbf{Phase 1.} Computes the three possible values in the whole time-space grid. This phase is divided into three steps:
\begin{enumerate}
\item[Step 1] \textbf{Recursive Computation for Continuous Control.} From the Definition \ref{def2.5} of the approximate Hamiltonian ($H_h$), the expression of the optimal continuous control $\theta^*(.)$ of Section \ref{Sect.5} and the terminal value $v_h(T,y_j)=G(y_j)$, one might recursively compute $v_h(s_i,y_j)$ for $(s_i,y_j)$ in a given time-space grid when only continuous control intervene by taking:
\begin{equation*}
\left\{
\begin{aligned}
&\theta_{i,j}=\arg\sup_{\theta\in\mathbb{R}^l}\Bigl\{(1-\lambda h)v_h\bigl(s_i+h,y_j+hb(s_i,y_j;\theta)\bigr)+hf(s_i,y_j;\theta)\Bigl\};\\
&v_h(s_i,y_j)=(1-\lambda h)v_h\bigl(s_i+h,y_j+hb(s_i,y_j;\theta_{i,j})\bigr)+hf(s_i,y_j;\theta_{i,j}).\\
\end{aligned}
\right.
\end{equation*}
The numerical computations require that for any time-space grid point $(s_i,y_j)$ the quantity $y_j+hb(s_i,y_j;\theta)$ remains in the space domain, all our numerical tests are such that this condition holds true. We use a linear interpolation operator $I$, $I[V]:[t,T]\times\mathbb{R}^n\rightarrow\mathbb{R}$ (see \cite{AFK15} and the Appendix of \cite{BC97}), to compute $v_h$ at any space point $y_j+hb(s_i,y_j;\theta)$, that is for a given vector $V_i=[V(s_i,y_1),\dots,V(s_i,y_j),\dots]^\top$ we have
\begin{equation}\label{equation7.1}
I[V_i]\bigl(s_{i},y_j+hb(s_i,y_j;\theta)\bigr):=V\bigl(s_{i},y_j+hb(s_i,y_j;\theta)\bigr).
\end{equation}
We mention that $\theta_{i,j}$ refers to $\theta(s_i)$, the value of the continuous control at time $s_i$ when the value of the state is $y_j$, and it is denoted by $\theta_{y_j}(s_i)$ in the Algorithm \ref{Alg.1} and in the graphical representation. 
\item[Step 2] \textbf{If the Minimizing Player-$\eta$ Intervene.} From the verification theorem, Theorem $\ref{Verification}$, one might conclude that the minimizing player-$\eta$ intervene at time $s_i$ only when
$$v_h\bigl(s_i,y_{t,x}^*(s_i)\bigr)-\Phi(h)\mathcal{H}_{inf}^\chi v_h\bigl(s_i,y_{t,x}^*(s_i)\bigr)>0,$$
where $y_{t,x}^*(s_i)$ denotes the optimal state which is progressively computed in \textbf{Phase 2} of Algorithm \ref{Alg.1}, and then a new approximate value has to be computed by solving the equation
\begin{equation}\label{equation7.2}
v_h(s_i,y_j)=\Phi(h)\inf_{\eta\in\mathbb{R}^q}\Bigl\{v_h\bigl(s_i,y_j+g_\eta(s_i,y_j;\eta)\bigr)+\chi(s_i,y_j;\eta)\Bigr\}.
\end{equation}
A \textit{Value Iteration Algorithm} is used to solve the equation (\ref{equation7.2}) for which the right hand side is computed by means of the linear interpolation operator $I$ of equation (\ref{equation7.1}) where the quantity $y_j+g_\eta(s_i,y_j;\eta)$ remains in the space domain. Algorithm \ref{Alg.2} describes the value iteration for solving the equation (\ref{equation7.2}) and gives the value for the minimizing player$-\eta$ in the whole space grid for a given time point $s_i$. The policy iteration of Algorithm \ref{Alg.1} will give the optimal size of the impulses.
\item[Step 3] \textbf{If the Maximizing Player-$\xi$ Intervene.} When the minimizing player-$\eta$ does not intervene at time $s_i$, the maximizing player-$\xi$ might intervene when $v_h\bigl(s_i,y_{t,x}^*(s_i)\bigr)-\Phi(h)\mathcal{H}_{sup}^cv_h\bigl(s_i,y_{t,x}^*(s_i)\bigr)<0$. Similarly a new approximate value for player-$\xi$ has to be computed using the value iteration method of Algorithm \ref{Alg.2}, the policy iteration of Algorithm \ref{Alg.1} will give the optimal size of the impulses.
\end{enumerate}
\item \textbf{Phase 2.} Generates successively the optimal controls $\Bigl\{\psi^*:=\bigl(\theta^*(s_i),u^*\bigr);v^*\Bigr\}$, the optimal state evolution $y_{s_1,y_k}^*(s_i)$ using the \textit{Explicit Euler Scheme} and the approximate value function $v_h\bigl(s_i,y_{s_1,y_k}^*(s_i)\bigr)$ for our zero-sum DG control problem. These outputs are denoted $\Bigl\{\psi^*:=\bigl(\theta_i^*,u^*:=(\tau_m^*,\xi_m^*)_{m\in\mathbb{N}^*}\bigr);v^*:=(\rho_k^*,\eta_k^*)_{k\in\mathbb{N}^*}\Bigr\}$, $y_i^*$ and $V_i^*$, respectively. \qed
\end{enumerate}
\end{remark}
\newpage
\begin{algorithm}[H]
\caption{\textsc{Policy Iteration for the Zero-Sum Game Problem Using the NE Strategy}}
\label{Alg.1}                          
\begin{algorithmic}[1]
\Require \\
Time-space grid $(s_i,y_j)$, where $s_i\in\{s_1=t,s_2,\dots,s_I=T\}$ and $y_j\in\{y_1,y_2,\dots,y_k,\dots,y_J\}$;\\
Initial state $x:=y_k\in\mathbb{R}^n$ for a carefully chosen $k$;\\
Initial continuous control $\theta_{0}^*\in\mathbb{R}^l$ at $s_1^-$;\\
Initial impulse values: $\xi_1^*\in\mathbb{R}^p$ at $\tau_1^*:=s_1^-$, and $\eta_1^*\in\mathbb{R}^q$ at $\rho_1^*:=s_1$;\\
Initial policies $\xi^{k}:=[\xi_{1}^{k},\dots,\xi_{J}^{k}]^\top\in\mathbb{R}^{p\times J}$ and $\eta^{k}:=[\eta_{1}^{k},\dots,\eta_{J}^{k}]^\top\in\mathbb{R}^{q\times J}$ at time $s_1^-$ and $s_1$, respectively;\\
Functions $b,g_\xi,g_\eta,f,c,\chi,G$ and $\Phi$, discount factor $\lambda$, discretization step $h$ and tolerance $\epsilon$.
\Ensure \\
NE strategy
$\bigl\{\theta_i^*:=\theta^*(s_i);u^*:=\{\tau_m^*,\xi_m^*\}_{m\in\mathbb{N}^*};v^*:=\{\rho_k^*,\eta_k^*\}_{k\in\mathbb{N}^*}\bigr\}$ for any $i$, where $\tau_m^*,\rho_k^*<s_I$;\\
Optimal state evolution in time grid given by $y_i^*:=y^*(s_i)=y^{\theta^*(.),u^*,v^*}_{s_1,y_k}(s_i)$ for any $i$;\\
Optimal value in time grid given by $V^*_{i}:=v_h(s_i,y_i^*)$ for any $i$.\newline
\State $V_{I,j}\gets G(y_j)\;\text{for all}\;j=1,\dots,J$; \Comment{Terminal value $V_{I,j}:=V(s_I,y_j)$.}
\State \textbf{Phase 1:} Backward Computation of the Values in Time-Space Grid.
\For {$i$ equals $I-1$ to $1$}
\State \textbf{Step 1:} Recursive Computation. \Comment{Computes $\theta_{i,j}:=\theta_{y_j}(s_i)$ and $V_{i,j}:=V(s_i,y_j)$ for any $i,j$.}
\For {$j$ equals $1$ to $J$}\Comment{$V_{i+1}$ denotes $[V_{i+1,1},\dots,V_{i+1,J}]^\top$.}
\State $\theta_{i,j}\gets\arg\sup_{\theta\in\mathbb{R}^l}\Bigl\{(1-\lambda h)I[V_{i+1}]\bigl(s_{i+1},y_j+hb(s_i,y_j;\theta)\bigr)+hf(s_i,y_j;\theta)\Bigr\}$;
\State $V_{i,j}\gets(1-\lambda h)I[V_{i+1}]\bigl(s_{i+1},y_j+hb(s_i,y_j;\theta_{i,j})\bigr)+hf(s_i,y_j;\theta_{i,j})$.
\EndFor
\State \textbf{Step 2:} The Value if the Minimizing Player-$\eta$ Intervene at Time $s_i$.
\State $\eta_i^{k+1}\gets\eta_i^k+\textbf{1}$ an initial guess; \Comment{Ensures the first iteration of while loop.}
\While{$\|\eta_i^{k+1}-\eta_i^k\|\geq\epsilon$} (or max iteration count reached)
\State $\eta_i^k\gets\eta_i^{k+1}$; \Comment{Sets a new loop to \textit{Evaluate} and \textit{Improve} the policy $\eta_{i}^{k+1}$.}
\State Policy \textit{Evaluation} Step:
\For {$j$ equals $1$ to $J$}\Comment{$\underline{V}_i^k$ denotes $[\underline{V}_{i,1}^k,\dots,\underline{V}_{i,J}^k]^\top$.}
\State Compute, using the \textit{Value Iteration} of the Algorithm \ref{Alg.2}, $\underline{V}_{i,j}^k$ solution of the equation
\begin{equation}\label{equation_7.3}
\underline{V}_{i,j}^k=\Phi(h)\Bigl(I[\underline{V}_i^k]\bigl(s_{i},y_j+g_\eta(s_i,y_j;\eta_{i,j}^k)\bigr)+\chi(s_i,y_j;\eta_{i,j}^k)\Bigl);
\end{equation}
\EndFor
\State Policy \textit{Improvement} Step:
\For {$j$ equals $1$ to $J$}
\State
\begin{equation*}
\begin{aligned}
\eta_{i,j}^{k+1}&\gets\arg\inf_{\eta\in\mathbb{R}^q}\Bigl\{\Phi(h)\Bigl(I[\underline{V}_i^k]\bigl(s_{i},y_j+g_\eta(s_i,y_j;\eta)\bigr)+\chi(s_i,y_j;\eta)\Bigl)\Bigr\};\\
\underline{V}_{i,j}^{k+1}&\gets\Phi(h)\Bigl(I[\underline{V}_i^k]\bigl(s_{i},y_j+g_\eta(s_i,y_j;\eta_{i,j}^{k+1})\bigr)+\chi(s_i,y_j;\eta_{i,j}^{k+1})\Bigl);
\end{aligned}
\end{equation*}
\EndFor
\EndWhile
\State $[\eta_{i,1},\dots,\eta_{i,J}]^\top\gets[\eta_{i,1}^{k+1},\dots,\eta_{i,J}^{k+1}]^\top$;
\State $[\underline{V}_{i,1},\dots,\underline{V}_{i,J}]^\top\gets[\underline{V}_{i,1}^{k+1},\dots,\underline{V}_{i,J}^{k+1}]^\top$.
\algstore{myalg}
\end{algorithmic}
\end{algorithm}

\begin{algorithm}[H]
\begin{algorithmic} [1]
\algrestore{myalg}
\State \textbf{Step 3:} The Value if the Maximizing Player-$\xi$ Intervene at Time $s_i$.
\State $\xi_i^{k+1}\gets\xi_i^k+\textbf{1}$ an initial guess; \Comment{Ensures the first iteration of while loop.}
\While{$\|\xi_i^{k+1}-\xi_i^k\|\geq\epsilon$} (or max iteration count reached)
\State $\xi_i^k\gets\xi_i^{k+1}$; \Comment{Sets a new loop to \textit{Evaluate} and \textit{Improve} the policy $\xi_{i}^{k+1}$.}
\State Policy \textit{Evaluation} Step:
\For {$j$ equals $1$ to $J$}\Comment{$\overline{V}_i^k$ denotes $[\overline{V}_{i,1}^k,\dots,\overline{V}_{i,J}^k]^\top$.}
\State Compute, using the \textit{Value Iteration} of the Algorithm \ref{Alg.2}, $\overline{V}_{i,j}^k$ solution of the equation
\begin{equation}\label{equation_7.4}
\overline{V}_{i,j}^k=\Phi(h)\Bigl(I[\overline{V}_i^k]\bigl(s_{i},y_j+g_\xi(s_i,y_j;\xi_{i,j}^k)\bigr)-c(s_i,y_j;\xi_{i,j}^k)\Bigl);
\end{equation}
\EndFor
\State Policy \textit{Improvement} Step:
\For {$j$ equals $1$ to $J$}
\State
\begin{equation*}
\begin{aligned}
\xi_{i,j}^{k+1}&\gets\arg\sup_{\xi\in\mathbb{R}^p}\Bigl\{\Phi(h)\Bigl(I[\overline{V}_i^k]\bigl(s_{i},y_j+g_\xi(s_i,y_j;\xi)\bigr)-c(s_i,y_j;\xi)\Bigl)\Bigr\};\\
\overline{V}_{i,j}^{k+1}&\gets\Phi(h)\Bigl(I[\overline{V}_i^k]\bigl(s_{i},y_j+g_\xi(s_i,y_j;\xi_{i,j}^{k+1})\bigr)-c(s_i,y_j;\xi_{i,j}^{k+1})\Bigl);
\end{aligned}
\end{equation*}
\EndFor
\EndWhile
\State $[\xi_{i,1},\dots,\xi_{i,J}]^\top\gets[\xi_{i,1}^{k+1},\dots,\xi_{i,J}^{k+1}]^\top$;
\State $[\overline{V}_{i,1},\dots,\overline{V}_{i,J}]^\top\gets[\overline{V}_{i,1}^{k+1},\dots,\overline{V}_{i,J}^{k+1}]^\top$.
\EndFor
\State \textbf{Phase 2:} Forward Deduction of the Optimal Controls (NE), State and Value in Time Grid.
\State $y^*_1,m,k\gets x,2,2$;
\For{$i$ equals $1$ to $I-1$} \Comment{Here, $j$ is such that $y_j\approx y^*_i$, i.e., $V_{i,j}\approx V(s_i,y^*_i)$.}
\If{$V_{i,j}>\Phi(h)\inf_{\eta\in\mathbb{R}^q}\Bigl\{I[V_i]\bigl(s_i,y^*_i+g_\eta(s_i,y^*_i;\eta)\bigr)+\chi(s_i,y^*_i;\eta)\Bigr\}$} \Comment{Miminizing Player$-\eta$ intervene.}
\State $\rho_k^*,\eta_k^*\gets s_i,\eta_{i,j}$;
\State $y^*_{i+1}\gets y^*_{i}+g_\eta(\rho_k^*,y^*_i;\eta_k^*)$;
\State $V^*_{i}\gets\underline{V}_{i,j}$;
\State $k\gets k+1$;
\ElsIf{$V_{i,j}<\Phi(h)\sup_{\xi\in\mathbb{R}^p}\Bigl\{I[V_i]\bigl(s_i,y_i^*+g_\xi(s_i,y_i^*;\xi)\bigr)-c(s_i,y_i^*;\xi)\Bigr\}$} \Comment{Maximizing Player$-\xi$ intervene.}
\State $\tau_m^*,\xi_m^*\gets s_i,\xi_{i,j}$;
\State $y^*_{i+1}\gets y^*_{i}+g_\xi(\tau_m^*,y^*_i;\xi_m^*)$;
\State $V^*_{i}\gets\overline{V}_{i,j}$;
\State $m\gets m+1$;
\Else \Comment{Optimal Continuous Control Intervene.}
\State $\theta^*_i\gets\theta_{i,j}$; \Comment{$\theta_{i}^*$ gets $\theta_{y^*_i}(s_i)$.}
\State $y^*_{i+1}\gets y^*_{i}+hb(s_i,y^*_i;\theta^*_i)$
\State $V^*_{i}\gets V_{i,j}$ 
\EndIf
\EndFor

\end{algorithmic}
\end{algorithm}
\begin{algorithm}[H]
\caption{\textsc{Value Iteration to Compute the Player's Value Function}}
\label{Alg.2}
\begin{algorithmic}[1]
\Require \\
Time $s_i$ and space grid $\{y_1,\dots,y_j,\dots,y_J\}$;\\
Functions $g_\eta,\chi$ (or, $g_\xi,c$) and $\Phi$, impulse value $\eta$ (or, $\xi$), time discretization step $h$ and tolerance $\epsilon$;\\
Initial value $V_i^k:=[V_{i,1}^k,\dots,V_{i,J}^k]^\top$.
\Ensure \\
The value $V_{i,j}:=V(s_i,y_j)$ solution of the equation (\ref{equation_7.3}) (or, (\ref{equation_7.4})) for any $j$.\newline
\State $V_i^{k+1}\gets V_i^k+\textbf{1}$ an initial guess; \Comment{Ensures the first iteration of while loop.}
\While{$\|V_i^{k+1}-V_i^k\|\geq\epsilon$} (or max iteration count reached)
\State $V_{i}^k\gets V_{i}^{k+1}$; \Comment{Sets a new loop to \textit{Improve} the value $V_{i}^{k+1}$.}
\For {$j$ equals $1$ to $J$} $$V_{i,j}^{k+1}\gets\Phi(h)\Bigl(I[V_i^k]\bigl(s_{i},y_j+g_\eta(s_i,y_j;\eta)\bigr)+\chi(s_i,y_j;\eta)\Bigl);$$
$$\biggl(\text{or,}\;V_{i,j}^{k+1}\gets\Phi(h)\Bigl(I[V_i^k]\bigl(s_{i},y_j+g_\xi(s_i,y_j;\xi)\bigr)-c(s_i,y_j;\xi)\Bigl)\biggr);$$
\EndFor
\EndWhile
\State $[V_{i,1},\dots,V_{i,J}]^\top\gets[V_{i,1}^{k+1},\dots,V_{i,J}^{k+1}]^\top$.
\end{algorithmic}
\end{algorithm}
\newpage
\section{Conclusion}
In this paper, we have considered a new class of deterministic finite-time horizon, two-player, zero-sum DGs, where the maximizing player takes continuous and impulse controls, while the minimizing player uses impulse control only. The aims were to optimize a discounted terminal gain/cost functional, approximate the value function, and describe an optimal strategy for the two players. After studying the related HJBI double-obstacle equation in the VS framework, we have proposed a discrete-time approximation scheme for this class of DGs given by the approximate equation (HJBI$_{h}$). We have further derived a verification result which analytically characterizes the equilibrium timing and level of impulses, and describes the optimal continuous actions. Our major contributions are the comparison principle, the convergence result for the approximate value function, and the verification theorem. Moreover, we have given some meaningful dynamics $b$, $g_\xi$ and $g_\eta$ to apply our results to continuous-time portfolio optimization problem, where the investor takes priority actions (impulses) only occasionally, while the market makes decisions both continuously and in specific impulse times, in such situation our results have been successfully implemented to derive a new continuous-time portfolio optimization model. Moreover, we have provided some computational algorithms to numerically determine the value function and the corresponding NE strategies and state evolution.
\par We intend to develop this work in two main directions in the future:
\begin{enumerate}
\item It would be interesting to consider a problem with feedback continuous control $\bigl(\text{i.e.,}\;\theta\;\text{depends on}\;y_{t,x}(s)\bigr)$, thus the instantaneous evolution of the state and the running gain/cost function at time $s$ become, respectively, $$b\Bigl(s,y_{t,x}(s);\theta\bigl(s,y_{t,x}(s)\bigr)\Bigl)\;\text{and}\;f\Bigl(s,y_{t,x}(s);\theta\bigl(s,y_{t,x}(s)\bigr)\Bigr);$$
\item Another extension of our work would be to adopt a machine learning approach based on the generative adversarial networks (GANs) \cite{GP-A14} to deep generate the value function and the corresponding NE and state evolution in the mini-max game framework of GANs (see also Wiese \& al. \cite{WKKK20}).
\end{enumerate}
\subsection*{Declarations}
The second author's research is financially supported by national center for scientific and technical research CNRST, Rabat, Morocco (Grant 17 UIZ 19). The authors declare no conflict of interest.
\fancyhead[LO]{REFERENCES}

\end{document}